\documentclass[a4paper]{amsart}

\usepackage[top=1.5in,right=1in,left=1in,bottom=1in,a4paper]{geometry}
\usepackage{amsmath}
\usepackage{amsthm}
\usepackage{amssymb}
\usepackage{amsfonts}
\usepackage{enumerate}
\usepackage{graphicx}
\usepackage[all]{xy}

\newtheoremstyle{theoremstyle}
  {10pt}      
  {5pt}       
  {\itshape}  
  {}          
  {\bfseries} 
  {:}         
  {.5em}      
  {}          

\newtheoremstyle{examplestyle}
  {10pt}      
  {5pt}       
  {}          
  {}          
  {\bfseries} 
  {:}         
  {.5em}      
  {}          

\theoremstyle{theoremstyle}
\newtheorem{theorem}{Theorem}[section]
\newtheorem*{theorem*}{Theorem}
\newtheorem{lemma}[theorem]{Lemma}
\newtheorem{proposition}[theorem]{Proposition}
\newtheorem*{proposition*}{Proposition}
\newtheorem{corollary}[theorem]{Corollary}
\newtheorem*{corollary*}{Corollary}

\theoremstyle{examplestyle}
\newtheorem{example}[theorem]{Example}
\newtheorem{definition}[theorem]{Definition}
\newtheorem*{definition*}{Definition}
\newtheorem{remark}[theorem]{Remark}
\newtheorem*{remarks*}{Remarks}

\newtheorem*{remark*}{Remark}

\newcommand{\comment}[1]{}

\newcommand{\pic}{\operatorname{Pic}}
\newcommand{\sh}[1]{\mathcal{#1}}
\newcommand{\rk}{\operatorname{rk}}
\newcommand{\Hom}{\operatorname{Hom}}
\newcommand{\Ext}{\operatorname{Ext}}
\newcommand{\End}{\operatorname{End}}

\newcommand{\Z}{\mathbb{Z}}

\newcommand{\Q}{\mathbb{Q}}

\newcommand{\on}{{[n]}}
\newcommand{\ot}{{[t]}}

\newcommand{\bs}{\operatorname{bs}}

\begin{document}

\title[Exceptional Sequences]{Exceptional sequences of invertible sheaves on rational surfaces}


\author{Lutz Hille}
\address{Mathematisches Institut\\Fachbereich Mathematik und Informatik der Universit\"at M\"unster\\
Einsteinstra\ss e 62\\48149 M\"unster\\Germany}
\email{lhill\_01@uni-muenster.de}

\author{Markus Perling}
\address{Fakult\"at f\"ur Mathematik\\Ruhr-Universit\"at Bochum\\Universit\"atsstra\ss e 150\\44780 Bochum\\Germany}
\email{Markus.Perling@rub.de}

\date{October 2008}

\thanks{The second author was supported by a research grant of the German Research Council (DFG)}

\begin{abstract}
In this article we consider exceptional sequences of invertible sheaves on smooth complete rational
surfaces. We show that to every such sequence one can associate a smooth complete toric surface in
a canonical way. We use this structural result to prove various theorems on exceptional and strongly
exceptional sequences of invertible sheaves on rational surfaces. We construct full strongly
exceptional sequences for a
large class of rational surfaces. For the case of toric surfaces we give a complete classification
of full strongly exceptional sequences of invertible sheaves.
\end{abstract}

\maketitle

\tableofcontents

\section{Introduction}\label{introduction}
The study of derived categories of coherent sheaves on algebraic varieties has gained much attention since
the mid-90's, with some of the main motivations coming from Kontsevich's homological mirror symmetry
conjecture \cite{Kontsevich94} and, evolving from this, the use of derived categories for D-branes in
superstring theory \cite{Douglas}. The object one studies is the derived category $D^b(X)$ of coherent
sheaves over a smooth algebraic variety $X$ defined over some algebraically closed field $\mathbb{K}$.
By definition, $D^b(X)$ is a categorial framework for the homological algebra of coherent sheaves
on $X$. It turns out that $D^b(X)$ carries a very rich structure and encodes information which
might not directly be visible from the geometry of $X$. For an overview we refer to the book
\cite{Huybrechts06} and the survey article \cite{Bridgeland06}. However, despite of many interesting
and deep results, the theory seems far from being developed enough to make $D^b(X)$ an easily accessible
object in any sense. A particular open problem is the construction of suitable generating sets, for which
the framework of exceptional sequences has been developed by the Seminaire Rudakov \cite{Rudakov90}:

\begin{definition*}
A coherent sheaf $\sh{E}$ on $X$ is called {\em exceptional}\, if $\Hom_{\sh{O}_X}(\sh{E}, \sh{E}) =
\mathbb{K}$ and $\Ext^i_{\sh{O}_X}(\sh{E}, \sh{E})$ $= 0$ for every $i \neq 0$. A sequence $\sh{E}_1,
\dots, \sh{E}_n$ of exceptional sheaves is called an {\em exceptional}\, sequence if
$\Ext^k_{\sh{O}_X}(\sh{E}_i, \sh{E}_j) = 0$ for all $k$ and for all $i > j$. If an exceptional sequence
generates $D^b(X)$, then it is called {\em full}. A {\em strongly} exceptional sequence is an
exceptional sequence such that $\Ext^k_{\sh{O}_X}(\sh{E}_i, \sh{E}_j) = 0$ for all $k > 0$ and all $i, j$.
\end{definition*}

If a full exceptional sequence $\sh{E}_1, \dots, \sh{E}_n$ exists on $X$ and $\langle \sh{E}_i \rangle$
denotes the minimal triangulated subcategory of $D^b(X)$ containing $\sh{E}_i$, then $\langle \sh{E}_1
\rangle, \dots \langle \sh{E}_n \rangle $ forms a {\em semi-orthogonal decomposition} of
$D^b(X)$, i.e. we have $\langle \sh{E}_j \rangle \subset \langle \sh{E}_i \rangle^\bot$ for all $i > j$.
Such decompositions naturally arise in birational geometry (see \cite{Orlov93}, \cite{Kawamata08}) and for
Fourier-Mukai transforms (see \cite{HillevandenBergh}).
Full strongly exceptional sequences provide an even stronger characterization of $D^b(X)$ in terms
of representation theory of algebras \cite{Happel88}.
By theorems of Baer \cite{Baer88} and Bondal \cite{Bondal90} for such a sequence there exists 
an equivalence of categories
\begin{equation*}
\operatorname{{\bf R}Hom}(\sh{T}, \, . \,) : D^b(X) \longrightarrow D^b(\End(\sh{T})-\operatorname{mod}),
\end{equation*}
where $\sh{T}:= \bigoplus_{i = 1}^n \sh{E}_i$, which is sometimes called a tilting sheaf. This way the
algebra $\End(\sh{T})$, at least in the derived sense, represents a non-commutative coordinate
system of $X$.

Strongly exceptional sequences have classically been known for the case of $\mathbb{P}^n$
(see \cite{Beilinson78} and \cite{DrezetLePotier}). However, exceptional or strongly exceptional
sequences must not exist in general, and their existence still is an open problem.
For instance, on Calabi-Yau varieties it follows from Serre duality that there
do not even exist exceptional sheaves. On the other hand, by now, exceptional sequences have
been constructed in many interesting cases, including certain types
of homogeneous spaces \cite{Kapranov86},
\cite{Kapranov88}, \cite{Kuznetsov05}, \cite{Samokhin07}, del Pezzo surfaces and almost
del Pezzo surfaces \cite{Gorodentsev89}, \cite{KuleshovOrlov}, \cite{Kuleshov97},
\cite{KarpovNogin98}, and some higher dimensional Fano varieties \cite{Nogin94}, \cite{Samokhin05}.

In this paper we consider exceptional sequences on smooth complete rational surfaces which consist of
invertible sheaves. This special setting is motivated by a conjecture of King \cite{King2}, which
states that on every smooth complete toric variety there exists a strongly exceptional sequence of
invertible sheaves. Invertible sheaves on toric varieties can be described in very explicit combinatorial
terms and a number of examples were well-known when the conjecture was stated. Also of interest here is the fact that toric
varieties can nicely be represented as moduli spaces of certain quiver representations and their universal
sheaf is a good candidate for a (partial) tilting sheaf. Examples of strongly exceptional sequences have
been given from this point of view in \cite{King2} and \cite{AltmannHille} (see also \cite{Broomhead06},
\cite{CrawSmith06}, \cite{BergmanProudfoot}). Other constructions have been given in \cite{CostaMiroRoig},
\cite{CostaMiroRoig05}, and for toric stacks
in \cite{BorisovHua08}. Typically, general constructions are only available for very
special situations such as iterated projective bundles, or small Picard number. It is known
that strongly exceptional sequences of invertible sheaves exist on the toric
$3$-Fanos, and computer experiments indicate that this is also true for $4$-Fanos. However,
general existence theorems are only available for exceptional sequences which are not strongly
exceptional. So it has been shown in \cite{Hille1} that exceptional sequences of invertible
sheaves exist on smooth toric surfaces. The existence of exceptional sequences which do not
necessarily consist of invertible sheaves has been shown for general smooth projective toric
stacks by Kawamata \cite{Kawamata1}. Despite a lot of positive evidence, the existence of
strongly exceptional sequences still is an open problem for toric varieties. In \cite{HillePerling06}
an example was given of a toric surface which does not admit a strongly exceptional sequence
of invertible sheaves, the second Hirzebruch surface iteratively blown up three times.
This counterexample at that time seemed somewhat mysterious, in particular because, having Picard number
$5$, it is surprisingly small. For general rational surfaces there is no bound for the Picard number.
This can be shown by well-known examples, such as simultaneous blow-ups of $\mathbb{P}^2$
in several points, by which any Picard number can be realized (see Theorem \ref{twotimesblowupexistence}).
In the toric case, explicit positive examples with higher Picard numbers
were known to the authors, including further blow-ups of the counterexample
(see example \ref{counterexampleblowup}). So the question is, what is the obstruction for the existence
of a (strongly) exceptional sequence of invertible sheaves on a toric or more general
rational surface? It turns out that toric surfaces are at the heart of the problem, even for the case
of general rational surfaces.
 The most important structural insight of this paper is the following
remarkable observation:

\begin{theorem*}[\ref{canonicaltoricsystemtheorem}]
Let $X$ be a smooth complete rational surface, let $\sh{O}_X(E_1), \dots, \sh{O}_X(E_n)$ be a
full exceptional sequence of invertible sheaves on $X$, and set $E_{n + 1} := E_1 - K_X$. Then
to this sequence there is associated in a canonical way a smooth
complete toric surface with torus invariant prime divisors $D_1, \dots, D_n$ such that
$D_i^2 + 2 = \chi\big(\sh{O}_X(E_{i + 1} - E_i)\big)$ for all $1 \leq i \leq n$.
\end{theorem*}

Of course, this theorem deserves a more detailed explanation which will be given below. For the
convenience of the reader we want first to present the most important consequences derived from
this. Our first main result shows the existence of exceptional sequences in general:

\begin{theorem*}[\ref{exceptionalexistence}]
On every smooth complete rational surface there exists a full exceptional sequence of invertible
sheaves.
\end{theorem*}

We point out that for rational surfaces this theorem is not a big surprise and can also be derived
from results of Orlov \cite{Orlov93}.
However, as noted above, an analogous theorem does not hold if we require the sequences
to be strongly exceptional. A necessary condition for the existence of a full strongly
exceptional sequence seems to be that the surface is not too far away from a minimal model.
By the Enriques classification, every smooth complete rational surface is a blow-up of the
projective plane or some Hirzebruch surface. In fact, we can prove that such sequences exist
on a surface which comes from blowing up a Hirzebruch surface once
or twice, possibly in several points in every step.

\begin{theorem*}[\ref{twotimesblowupexistence}]
Any smooth complete rational surface which can be obtained by blowing up 
a Hirzebruch surface two times (in possibly several points in each step) has a full strongly exceptional
sequence of invertible sheaves.
\end{theorem*}

In the toric case, we can show that the converse is also true:

\begin{theorem*}[\ref{toricbound}]
Let $\mathbb{P}^2 \neq X$ be a smooth complete toric surface. Then there exists a full strongly
exceptional sequence of invertible sheaves on $X$ if and only if $X$ can be obtained from a Hirzebruch
surface in at most two steps by blowing up torus fixed points.
\end{theorem*}

Note that the blow-up of $\mathbb{P}^2$ at any point is isomorphic to the first Hirzebruch surface. So
there is no loss of generality if only blow-ups of Hirzebruch surfaces are considered. In particular,
Theorem \ref{toricbound} implies that the Picard number of a toric surface on which a full strongly
exceptional sequence of invertible sheaves exists is at most 14. On the other hand, the example given
in \cite{HillePerling06} is a minimal example which does not satisfy the condition of the theorem.

Another important aspect of exceptional sequences is their relation to helix theory as developed
in \cite{Rudakov90}.

\begin{definition*}
An infinite sequence of sheaves $\dots, \sh{E}_i, \sh{E}_{i + 1}, \dots $ is called a {\em cyclic
(strongly) exceptional sequence} if there exists an $n$ such that $\sh{E}_{i + n} \cong \sh{E}_i
\otimes \sh{O}(-K_X)$ for
every $i \in \Z$ and if every winding (i.e. every subinterval $\sh{E}_{i + 1}, \dots, \sh{E}_{i + n}$)
forms a (strongly) exceptional sequence. A cyclic exceptional sequence is {\em full} if every winding
is a full exceptional sequence.
\end{definition*}

Our notion of cyclic strongly exceptional sequences is very close to the geometric helices of
\cite{BondalPolishchuk}, but we want to point out that these notions do not coincide, as we do not
require that our cyclic exceptional sequences are generated by mutations. In fact, if
we consider a winding $\sh{E}_{i + 1}, \dots, \sh{E}_{i + n}$ as the foundation of a helix, then
the $n$-th right mutation of $\sh{E}_i$ coincides with $\sh{E}_{i + n}$ up to a shift in the derived
category. By results of
\cite{Bondal90} a foundation of a helix generates the derived category precisely if any foundation
does. Hence a cyclic exceptional sequence is full if and only if it has any winding which is a full
exceptional sequence. By a result of Bondal
and Polishchuk, the maximal periodicity of a geometric helix on a surface is $3$, which implies
that $\mathbb{P}^2$ is the only rational surface which admits a full geometric helix. Our weaker notion
admits a bigger class of surfaces, but still imposes very strong conditions:

\begin{theorem*}[\ref{helixpicbound}]
Let $X$ be a smooth complete rational surface on which a full cyclic strongly exceptional sequence
of invertible sheaves exists. Then $\rk \pic(X) \leq 7$.
\end{theorem*}

So not even every del Pezzo surface admits such a sequence. However:

\begin{theorem*}[\ref{delPezzocyclicexistence}]
Let $X$ be a del Pezzo surface with $\rk \pic(X) \leq 7$, then there exists a full cyclic
strongly exceptional sequence of invertible sheaves on $X$.
\end{theorem*}

The condition that $-K_X$ is ample can be weakened in general. In the toric case we obtain a complete
characterization for toric surfaces admitting cyclic strongly exceptional sequences:

\begin{theorem*}[\ref{toriccyclicnef} \& \ref{toriccyclicexistence}]
Let $X$ be a smooth complete toric surface, then there exists a full cyclic strongly exceptional sequence
of invertible sheaves on $X$ if and only if $-K_X$ is nef.
\end{theorem*}

Note that cyclic strongly exceptional sequences have been considered before, most notably in physics
literature (see \cite{HananyHerzogVegh}, \cite{Aspinwall08}, \cite{BergmanProudfoot06},
\cite{HerzogKarp06}), but usually under different names. Theorems \ref{toriccyclicnef} and
\ref{toriccyclicexistence}
have been conjectured in this context. The particular interest here comes from the fact that
the total space $\pi: \omega_X \rightarrow X$ of the canonical bundle $\sh{O}_X(K_X)$
is a local Calabi-Yau manifold. It follows from results of Bridgeland \cite{Bridgeland05} that
a full strongly exceptional sequence $\sh{E}_1, \dots, \sh{E}_n$ on $X$ can be extended to a
cyclic strongly exceptional sequence iff the pullbacks
$\pi^*\sh{E}_1, \dots, \pi^*\sh{E}_n$ form a sequence on $\omega_X$ which is almost
exceptional in the sense that the $\pi^*\sh{E}_i$ generate $D^b(\omega_X)$ and
$\Ext^k(\pi^*\sh{E}_i, \pi^*\sh{E}_j) = 0$ for every $i, j$
and all $k > 0$ (however, due to the fact  that $\omega_X$ is not complete,
we cannot expect that any $\Hom$-groups among the $\pi^*\sh{E}_i$ vanish).
Another interesting observation is that for the toric singularities which arise from
contracting the zero section in $\omega_X$, the endomorphism algebras of $\bigoplus_{i = 1}^n
\pi^*\sh{E}_i$  give examples for non-commutative resolutions in the sense of van den Bergh
\cite{vandenBergh04a}, \cite{vandenBergh04b}.

\

Now we give some more technical explanations concerning Theorem \ref{canonicaltoricsystemtheorem}
and its consequences.
The key idea is astoundingly simple. Let $X$ be a smooth complete rational surface and
$E_1, \dots, E_n$ Cartier divisors on $X$ such
that $\sh{O}_X(E_1),$ $\dots, \sh{O}_X(E_n)$ form an exceptional sequence of
invertible sheaves. For these sheaves, there are natural isomorphisms
$\Ext^k_{\sh{O}_X}\big(\sh{O}_X(E_i),$ $\sh{O}_X(E_j)\big)$ $\cong H^k\big(X, \sh{O}_X(E_j - E_i)\big)$
and therefore it is convenient to bring this exceptional sequence into a normal
form by passing to differences. We set $A_i := E_{i + 1} - E_i$ for $1 \leq i < n$ and
$A_n := -K_X - \sum_{i = 1}^{n - 1} A_i$, where $K_X$ denotes the canonical divisor.
The reason for adding $A_n$ will become clear below. The fact that the $E_i$ form an exceptional sequence
then implies $H^k\big(X, \sh{O}_X(-\sum_{i \in I} A_i)\big) = 0$ for every interval
$I \subset [1, \dots, n - 1]$ and every $k > 0$.
It is an easy consequence of the Riemann-Roch
theorem that moreover the $A_i$ have the following properties:
\begin{enumerate}[(i)]
\item $A_i . A_{i + 1} = 1$ for $1 \leq i < n$ and $A_1 . A_n = 1$;
\item $A_i . A_j = 0$ for $i \neq j$, $\{i, j\} \neq \{1, n\}$, and
$\{i, j\} \neq \{k, k + 1\}$ for any $1 \leq k < n$;
\item $\sum_{i = 1}^n A_i = -K_X$.
\end{enumerate}

\begin{definition*}
We call a set of divisors on $X$ which satisfy the conditions (i), (ii), (iii) above a
{\em toric system}.
\end{definition*}

With respect to a toric system we consider the short exact sequence
\begin{equation*}
0 \longrightarrow \pic(X) \overset{A}{\longrightarrow} \Z^n \longrightarrow \Z^2 \longrightarrow 0,
\end{equation*}
where $A$ maps a divisor class $D$ to the tuple $(A_1 . D, \dots, A_n . D)$.
The images $l_1, \dots, l_n$ of the standard basis of $\Z^n$ in $\Z^2$ are the {\em Gale duals}
of $\mathcal{A} = A_1, \dots, A_n$. It is now an exercise in linear algebra (see Proposition
\ref{dualityproposition}) to show that the $l_i$
generate the fan of a smooth complete toric surface which we denote $Y(\mathcal{A})$.
This means, by passing from $E_1, \dots, E_n$ via its toric system to the vectors $l_1, \dots, l_n$, we
have a canonical way of associating a toric surface to a strongly exceptional sequence of invertible
sheaves on any rational surface. This correspondence is even stronger; as Gale duality is indeed
a duality, we can as well consider the $A_i$ as Gale duals of the $l_i$. But by a standard fact of toric
geometry, the Gale duals of the $l_i$ can be interpreted as the classes of the torus invariant prime
divisors $D_1, \dots, D_n$ on $Y(\mathcal{A})$. Hence, we can identify $\pic(X)$ and $\pic\big(Y(\mathcal{A})\big)$
and the respective intersection products in a natural way, such that $A_i^2 = D_i^2$ for all $i$.
In particular, note that the set of invariant irreducible divisors forms a toric system for any smooth
complete toric surface.

Implicitly, toric systems have already shown up in the classical analysis of del Pezzo surfaces.
In modern form, this seems first to be written in  the first edition of \cite{Manin86} (see also
\cite{Demazure80}). Consider $X$ a $t$-fold blow-up of $\mathbb{P}^2$, i.e. $X = X_t
\overset{b_t}{\rightarrow} X_{t - 1} \overset{b_{t - 1}}{\rightarrow} \cdots
\overset{b_2}{\rightarrow} X_1 \overset{b_1}{\rightarrow} \mathbb{P}^2$. Then we get a nice basis
$H, R_1, \dots, R_t$ of $\pic(X)$, where $H$ is the pull-back of the class of a line on
$\mathbb{P}^2$, and $R_i$ is the pull-back of the exceptional divisor of the blow-up $b_i$. This
basis diagonalizes the intersection product of $\pic(X)$, i.e. $H^2 = 1$, $R_i^2 = -1$ and
$H . R_i = 0$ for all $i$, and $R_i . R_j = 0$ for all $i \neq j$. For simplicity, let us assume
that $t > 5$. Then we construct a graph as follows. For the vertices, we set $A_0 := H - R_1 -
R_2 - R_3$ and $A_i := R_i - R_{i + 1}$ for $i = 1, \dots, t - 1$ and we draw an edge between
$A_i$ and $A_j$ whenever
$A_i . A_j \neq 0$. This way we obtain a graph of type $\mathbb{E}_t$ which is indefinite for
$t > 8$. For $t \leq 8$ it is shown in \cite{Manin86} that the set of divisors $\{D \in \pic(X)
\mid \chi(-D) = -K_X . D = 0\}$ forms a root system which is generated by the $A_i$. In case of
$t = 6$ this root system represents the symmetries of the famous $27$ lines on the cubic surface.
The system of divisors $A_0, \dots, A_{t - 1}$ is almost a toric system. We can turn it into
a proper toric system by removing $A_0$ and adding $A_t
:= R_t$, $A_{t + 1} := H - \sum_{i = 1}^t R_i$, $A_{t + 2} := H$, and $A_{t + 3} := H - R_1$.
This toric system always represents an exceptional sequence which is of the form
$\sh{O}_X, \sh{O}_X(R_1), \dots, \sh{O}_X(R_t), \sh{O}_X(H), \sh{O}_X(2H)$. In case that
the $b_i$ commute, this sequence is even strongly exceptional. Note that there
always are ambiguities concerning the enumeration of the $A_i$; we always can try to change it
cyclically or even choose the reverse enumeration.

This sequence gives an example of an exceptional sequence which is an augmentation of the standard
sequence on $\mathbb{P}^2$. On $\mathbb{P}^2$ there exists a unique toric system, which is of the form
$H$, $H$, $H$. After blowing up once, we can augment this toric system by inserting $R_1$ in any place
and subtracting $R_i$ in the two neighbouring positions, i.e., up to symmetries, we obtain a toric
system $H - R_1, R_1, H - R_1, H$ on $X_1$. Continuing with this, we essentially get two
possibilities on $X_2$, namely
\begin{align*}
& H - R_1 - R_2, R_2, R_1 - R_2, H - R_1, H \\
& H - R_1, R_1, H - R_1 - R_2, R_2, H - R_2.
\end{align*}
It is easy to see that all of these examples lead to strongly exceptional sequences for almost
all enumerations which keep the cyclic order. The only exception being the first one in the case
where $b_2$ is a blow-up of an infinitesimal point.
Here, we necessarily have to choose the enumeration of the $A_i$ such that $A_n = R_1 - R_2$.

Similarly, on any Hirzebruch surface $\mathbb{F}_a$ there exist, in fact infinitely many, toric systems
of the form $P, sP + Q, P, -(a + s)P + Q$ with $s \geq -1$, which correspond to strongly exceptional
sequences. Here,
$P$ and $Q$ are the two generators of the nef cone in $\pic(\mathbb{F}_a)$, where $P$ is the class of
a fiber of the $\mathbb{P}^1$-fibration $\mathbb{F}_a \rightarrow \mathbb{P}^1$ and $Q$ is the generator
with $Q^2 = a$. We can extend these toric systems along blow-ups in an analogous fashion. We call
toric systems obtained this way {\em standard augmentations} (see Definition
\ref{standardaugmentationdefinition}). It turns out that Theorem \ref{toricbound} is a consequence of
the following characterization of strongly exceptional sequences arising from standard augmentations.

\begin{theorem*}[\ref{augmentationboundtheorem}]
Let $\mathbb{P}^2 \neq X$ be a smooth complete rational surface which admits a full strongly
exceptional sequence whose associated toric system is a standard augmentation. Then $X$ can be
obtained by blowing up a Hirzebruch surface two times (in possibly several points in each step).
\end{theorem*}

Standard augmentations provide a straightforward procedure which allows to produce strongly exceptional
sequences of invertible sheaves on a large class of rational surfaces. It is natural to ask whether
it is actually possible to get all such sequences this way. The answer so far is: probably yes. Indeed,
Theorem \ref{toricbound} is a corollary of Theorem \ref{augmentationboundtheorem} and the following
result:

\begin{theorem*}[\ref{toricclassification}]
Let $X$ be a smooth complete toric surface, then every full strongly exceptional sequence of invertible
sheaves comes from a toric system which is a standard augmentation.
\end{theorem*}

Conjecturally, this Theorem should generalize to general rational surfaces. However, our result
is based on a rather detailed analysis of cohomology vanishing on toric surfaces which we cannot
easily extend to the general case. Moreover, a standard augmentation does not necessarily look
like a standard augmentation at the first glance. In the phrase ``comes from'' in above theorem
is hidden a normalization process which must be performed and, as such, is almost obvious (see
the end of section \ref{exceptionalsequencessection} for details), but whose necessity significantly
increases the difficulty of the classification. It turns out that in the toric case all ``difficult''
strongly exceptional sequences are related to cyclic exceptional sequences. These in turn are
easier to understand, but in no case it is a priori clear whether a given strongly exceptional
sequence is cyclic. We hope to obtain a more geometric understanding for this in future work.

{\bf Overview.} In section \ref{surfacegeometry}, after surveying
some standard facts on the geometry of smooth complete toric surfaces, we introduce toric systems and
explain their relation to toric surfaces. In section \ref{rationalsurfacesandtoricsystems} we derive
some elementary properties from cohomology vanishing and show that to every exceptional sequence on a
smooth complete rational surface there is associated a toric system. Section \ref{rationalacyclicsection}
contains some general results for cohomology vanishing on rational surfaces. Based on this, we prove
in section \ref{exceptionalsequencessection} our results for exceptional sequences on general
rational surfaces, except for Theorem \ref{augmentationboundtheorem}, which is proved in section
\ref{augmentationboundproofsection}. Sections \ref{toriccohomologyvanishing} to \ref{toricproofs}
are entirely devoted to the case of toric surfaces. In section \ref{toriccohomologyvanishing} we give
a detailed description of cohomology vanishing of divisors on smooth complete toric surfaces. Section
\ref{torictheorems}
contains the main results on strongly exceptional sequences on toric surfaces. In sections
\ref{straighteningsection} and \ref{toricproofs} we give a proof of Theorem \ref{toricclassification}.

{\bf Notation and general conventions.} For some positive integer $l$, we denote $[l] := \{1, \dots, l\}$.
If we use the letter $n$ (or $n - 1$, $n + 1$, $n + k$, etc.), we will usually assume that the elements of
$\on$ are in cyclic order in the sense that we consider $\on$ as a system of representatives of $\Z/n \Z$.
In particular, for some $i \in \on$ and some $j \in \Z$, we identify $i + j$ with the corresponding class
in $\on$. If we use some different letter, say $t$, then we will usually consider the standard total order
on the set $[t]$. Depending on context, we may also consider other partial orders on the set $[t]$. An
{\em interval} $I \subsetneq \on$ is a subset $I = \{i, i + 1, \dots, i + k\}$, where $i \in \on$,
$i + k \leq n$ and $0 \leq k < n - 1$. A {\em cyclic} interval $I \subsetneq \on$ is either an interval
or the union $I = I_1 \cup I_2$ of two intervals such that $1 \in I_1$ and $n \in I_2$. For any
$\Z$-module $K$, we will denote $K_\Q := K \otimes_\Z \Q$.
For some divisor $D$ on a variety $X$, we will usually omit the subscript $X$ for the
corresponding invertible sheaf $\sh{O}_X(D)$ if there is no ambiguity for $X$. We denote $h^i(D) :=
h^i\big(\sh{O}(D)\big) := \dim H^i\big(X, \sh{O}_X(D)\big)$. We will frequently make use of the fact that
for any Cartier divisor $D$ on an algebraic surface $X$ and any blow-up $b: X' \rightarrow X$ there are
isomorphisms $H^i\big(X', b^*\sh{O}_X(D)\big) \cong H^i\big(X, \sh{O}_X(D)\big)$ for every $i \in \Z$.

\

{\bf Acknowledgements.} We would like to thank both the Institut Fourier, Grenoble and the
Mathematisches Forschungsinstitut Oberwolfach, for their hospitality and generous support.

\section{The birational geometry of toric surfaces}\label{surfacegeometry}

For general reference on toric varieties, we refer to  \cite{Oda} and \cite{Fulton}. The specifics
for toric surfaces are taken from \cite{MiyakeOda} and \cite{Oda}. For Gale transformation, we
refer to \cite{GKZ} and \cite{OdaPark}. Let $X$ be a smooth complete toric surface
defined over some algebraically closed field $\mathbb{K}$. That is, there exists a two-dimensional
torus $T \cong (\mathbb{K}^*)^2$ acting on $X$ such that $T$ itself is embedded as maximal open
and dense orbit in $X$ on which the action restricts to the group multiplication of $T$. It is
clear that every such $X$ is rational.

We denote $M = \Hom(T, \mathbb{K}^*) \cong \Z^2$ and $N = \Hom(\mathbb{K}^*, T) \cong \Z^2$ the
character and cocharacter groups of $T$, respectively. The
toric surface $X$ is completely determined by a collection of elements $l_1, \dots, l_n \in N$
with the following properties. We assume that the $l_i$ are circularly ordered and indexed by
elements in $\on$. Then for every $i \in \on$ the pair $l_i, l_{i + 1}$ forms a positively
oriented basis of $N$. Moreover, for every such pair there exists no other $l_k$ such that
$l_k = \alpha_i l_i + \alpha_{i + 1} l_{i + 1}$ for some nonnegative integers $\alpha_i$,
$\alpha_{i + 1}$. Every pair $l_i, l_{i + 1}$ generates a two-dimensional rational polyhedral
cone in the vector space $N_\Q$, and the collection of faces of all these cones is the fan
$\Delta$ associated to $X$. There is a one-to-one correspondence of $1$-dimensional $T$-orbits
in $X$ and the rays in $\Delta$, i.e. the one-dimensional cones, which have the $l_i$ as
primitive vectors. The corresponding orbit closures
we denote by $D_i$. Every $D_i$ is isomorphic to $\mathbb{P}^1$, and for every $i$, the divisors $D_i$
and $D_{i + 1}$ intersect transversely in the torus fixed point associated to the cone generated by
$l_i$ and $l_{i + 1}$, thus $D_i . D_{i + 1} = 1$. This way, the $D_i$ form a cycle of rational curves
in $X$ of arithmetic genus $1$. Moreover, for every $i \in \on$ there exists the unique relation
\begin{equation*}
l_{i - 1} + a_i l_i + l_{i + 1} = 0,
\end{equation*}
where $a_i = D_i^2 \in \Z$ is the self-intersection number of $D_i$.

Clearly, if just the integers $a_i$ are known, we can reconstruct the $l_i$ from the $a_i$ up to an
automorphism of $N$. However, an arbitrary sequence of $a_i$'s does not necessarily lead to a well-defined
smooth toric surface. An admissible sequence $a_1, \dots, a_n$ is determined by the minimal model
program for toric surfaces. Whenever $a_i = -1$ for some $i$, we can equivariantly blow down the
corresponding $D_i$ and obtain another smooth toric surface $X'$ on which $T$ acts. This surface is
specified by a sequence $a_1', \dots, a_{i - 1}', a_{i + 1}', \dots a_n'$ (where, up to a cyclic change
of enumeration, we can assume that $1 < i < n$) such that $a_{i - 1}' = a_{i - 1} + 1$, $a_{i + 1}' =
a_{i + 1} + 1$, and $a_k' = a_k$ for $k \neq i - 1, i, i + 1$. Conversely, an equivariant blow-up at
some point $D_i \cap D_{i + 1}$ is described by changing $a_1, \dots, a_i, a_{i + 1}, \dots a_n$ to
$a_1, \dots, a_{i - 1}, a_i - 1, -1, a_{i + 1} - 1, a_{i + 2}, \dots, a_n$. This way, we arrive at the
same class of minimal models as in the case of general rational surfaces:

\begin{theorem}
Every toric surface can be obtained by a finite sequence of equivariant blow-ups of $\mathbb{P}^2$
or some Hirzebruch surface $\mathbb{F}_a$.
\end{theorem}

In particular, the sequences of self-intersection numbers associated to $\mathbb{P}^2$ and the $\mathbb{F}_a$
are $1, 1, 1$ for $\mathbb{P}^2$ and $0, a, 0, -a$
for $\mathbb{F}_a$. Every other admissible sequence $a_1, \dots, a_n$ can be obtained by successive
augmentation of one of these sequences by the aforementioned process. In particular, this implies

\begin{proposition}\label{twelveminusn}
Let $X$ be a smooth complete toric surface determined by self-intersection numbers $a_1, \dots, a_n$.
Then $\sum_{i = 1}^n a_i = 12 - 3n$.
\end {proposition}

There is also a local version of above theorem:

\begin{proposition}[\cite{MiyakeOda}]\label{smoothblowdownproposition}
Let $i < k$ such that $l_i, l_k$ form a positively oriented basis. Then there exists a sequence of blow-downs
from $X$ to a smooth complete toric surface $X'$ whose associated primitive vectors are $l_1, \dots, l_i,
l_k, \dots, l_n$.
\end{proposition}

The Picard group of $X$ is generated by the $T$-invariant divisors $D_1, \dots, D_n$. More precisely,
we have a short exact sequence
\begin{equation}\label{standardsequence}
0 \longrightarrow M \overset{L}{\longrightarrow} \Z^n \longrightarrow \pic(X) \longrightarrow 0,
\end{equation}
where $L = (l_1, \dots, l_n)$, i.e. the $l_i$ are considered as linear forms on $M$. The $i$-th element
of the standard basis of $\Z^n$ maps to the rational equivalence class of the divisor $D_i$.
There is no canonical choice of coordinates for $\pic(X)$, but there is a very natural and convenient
representation for toric divisors if considered as elements in the group of numerical equivalence classes
of curves $N_1(X)$. Consider the natural pairing on $X$:
\begin{equation*}
N_1(X) \otimes \pic(X) \longrightarrow \Z, \qquad (C, D) \mapsto C . D,
\end{equation*}
which is a non-degenerate bilinear form. The pairing is completely specified by the intersection products
of the $D_i$ among each other, which are given by
\begin{equation*}
D_i . D_j =
\begin{cases}
a_i & \text{if } i = j, \\
1 & \text{if } j \in \{i - 1, i + 1\}, \\
0 & \text{else}.
\end{cases}
\end{equation*}
Denote $\mathfrak{D} := (D_i . D_j)_{i, j = 1, \dots, n}$ the corresponding matrix. Then we
have a linear map $\Z^n \overset{\mathfrak{D}}{\rightarrow} \Z^n$ whose kernel is $M$, the
group of numerically trivial
$T$-invariant divisors. Given a $T$-invariant divisor $D := \sum_{i \in \on} c_i D_i$, its image
$\mathfrak{D}(D)$ is a
tuple of the form $(d_1, \dots, d_n)$, where $d_i := d_i(D) := c_{i - 1} + a_i c_i + c_{i + 1} = D . D_i$.
If we dualize sequence (\ref{standardsequence}), we get
\begin{equation}\label{dualstandardsequence}
0 \longrightarrow \pic(X)^* \longrightarrow \Z^n \overset{L^T}{\longrightarrow} N \longrightarrow 0,
\end{equation}
where $L^T$ denotes the transpose of $L$. The kernel of $L^T$ coincides with the image of $\mathfrak{D}$,
so that we can identify $\pic(X)^*$ with $N_1(X)$ in a natural way as subgroups of $\Z^n$. So, if considered
as a curve, the tuple $(d_1, \dots, d_n)$ is a natural representation of $D$ which does not depend on
the choice of a $T$-invariant representative. Moreover, by sequence (\ref{dualstandardsequence}) we have
for any tuple $(d_1, \dots, d_n) \in N_1(X)$ that
\begin{equation*}
\sum_{i \in \on} d_i l_i = 0.
\end{equation*}
By this we can identify $N_1(X)$ with the set of closed polygonal lines in $N_\Q$ whose segments are given
by some multiple of every $l_i$. We will make use of this and give some more detail in section
\ref{toriccohomologyvanishing}. Note that to determine whether some $D$ is nef, it suffices to test this
on the $T$-invariant divisors. We have:

\begin{proposition}\label{dinefproposition}
Let $D$ a $T$-invariant divisor on $X$, then
\begin{enumerate}[(i)]
\item for every $i \in \on$ we have $d_i = \deg \sh{O}(D)\vert_{D_i}$;
\item $D$ is nef iff $d_i \geq 0$ for every $i \in \on$.
\end{enumerate}
\end{proposition}

In particular:

\begin{proposition}\label{canonicalnefproposition}
Denote $K_X = - \sum_{i = 1}^n D_i$ the canonical divisor on $X$. Then $d_i(K_X) = - a_i - 2$
for all $i$. Then $-K_X$ is nef iff $a_i \geq -2$ for all $i$ and $-K_X$ is ample iff $a_i \geq -1$
for  $i$.
\end{proposition}

Note that on a smooth toric surface an invertible sheaf is ample if and only if it is very ample.
There are precisely 16 smooth complete toric surface whose anti-canonical divisor is nef
(including the 5 del Pezzo surfaces which admit a toric structure).
These are shown in table \ref{weakdelpezzofigure} in terms of the
self-intersection numbers $a_i$. In this table, the first four surfaces are given their standard names, the
other labels just reflect the length of the sequence $a_1, \dots, a_n$.

\begin{table}[htbp]
\centering
\begin{tabular}{c|c}
Name & self-intersection numbers $a_1, \dots, a_n$ \\ \hline
$\mathbb{P}^2$ & 1, 1, 1 \\
$\mathbb{P}^1 \times \mathbb{P}^1$ & 0, 0, 0, 0 \\
$\mathbb{F}_1$ & 0, 1, 0, -1 \\
$\mathbb{F}_2$ & 0, 2, 0, -2 \\
5a & 0, 0, -1, -1, -1 \\
5b & 0, -2, -1, -1, 1 \\
6a & -1, -1, -1, -1, -1, -1 \\
6b & -1, -1, -2, -1, -1, 0 \\
6c & 0, 0, -2, -1, -2, -1 \\
6d & 0, 1, -2, -1, -2, -2 \\
7a & -1, -1, -2, -1, -2, -1, -1 \\
7b & -1, -1, 0, -2, -1, -2, -2 \\
8a & -1, -2, -1, -2, -1, -2, -1, -2 \\
8b & -1, -2, -2, -1, -2, -1, -1, -2 \\
8c & -1, -2, -2, -2, -1, -2, 0, -2 \\
9 & -1, -2, -2, -1, -2, -2, -1, -2, -2 \\
\end{tabular}
\caption{The 16 complete smooth toric surfaces whose anti-canonical divisor is nef.}\label{weakdelpezzofigure}
\end{table}

The short exact sequence (\ref{standardsequence})  is an example for a {\em Gale} transform. By general
properties of Gale transforms, for any subset $I$ of $\{1, \dots, n\}$,
the set $L_I := \{l_i \mid i \in I\}$ forms a basis of $N$ iff the complementary set $\{D_i \mid i
\notin I\}$ forms a basis of $\pic(X)$, and $L_I$ is a minimal linearly dependent set iff the
complementary
set is a maximal subset of the $D_i$ which is contained in a hyperplane in $\pic(X)$. Moreover, we can
invert any Gale transform by considering the dual short exact sequence. So by the sequence
(\ref{dualstandardsequence}) we get back the $l_i$ from the $D_i$.

\begin{definition}\label{toricsystemdef}
Let $P$ be a free $\Z$-module of rank $n - 2$ together with a integral symmetric bilinear form
$\langle\, , \, \rangle$. A sequence of elements $A_1, \dots, A_n$ in $P$ is called an
{\em abstract toric system} iff it satisfies the following conditions:
\begin{enumerate}[(i)]
\item\label{toricsystemdefi} $\langle A_i , A_{i + 1} \rangle = 1$ for $i \in \on$;
\item\label{toricsystemdefii} $\langle A_i , A_j \rangle = 0$  for $i \neq j$ and $\{i, j\} \neq \{k, k + 1\}$
for all $k \in \on$;
\item\label{toricsystemdefiii} $\sum_{i = 1}^n \langle A_i, A_i \rangle = 12 - 3n$.
\end{enumerate}
\end{definition}

Clearly, for any given smooth complete toric surface $X$, the divisors $D_1, \dots, D_n$ form an abstract toric
system in $\pic(X)$ with respect to the intersection form. We show that the data specifying an abstract toric
system is equivalent to defining a toric surface.

\begin{proposition}\label{dualityproposition}
Let $P$, $\langle\, , \,\rangle$ as in definition \ref{toricsystemdef}, $A_1, \dots, A_n$ an abstract toric
system and consider the Gale duals $l_1, \dots, l_n$ in $N := \Z^n / P$ of the $A_i$. Then $N \cong \Z^2$ and
the $l_1, \dots, l_n$ generate the fan of a smooth complete toric surface $X$
with $T$-invariant irreducible divisors $D_1, \dots, D_n$ such that $D_i^2 = \langle A_i, A_i \rangle$ for
every $1 \leq i \leq n$. In particular, we can identify $P$ with $\pic(X)$ and $\langle \, , \, \rangle$
with the intersection form on $\pic(X)$.
\end{proposition}

\begin{proof}
For $n < 3$ there is nothing to prove and for $n = 3$ the statement is easy to see. So we assume without loss
of generality that $n \geq 4$. We first show that $\{A_j \mid j \neq i, i + 1\}$ forms a basis of $\pic(X)$ for
every $i \in \on$. This implies that $N \cong \Z^2$ and, by Gale duality,
that the complementary pairs of $l_i$ are bases of $N$. Up to cyclic renumbering, it suffices to show that
$A_1, \dots, A_{n - 2}$ is a basis of $\pic(X)$. We have $\langle A_1, A_2 \rangle = 1$, $\langle A_n, A_1
\rangle =  1$ and $\langle A_n, A_2 \rangle =  0$. As $\langle \, , \, \rangle$ is integral, this implies
that $A_1, A_2$ generate a subgroup of rank two of $P$. This subgroup is saturated, i.e. every element
in $P$ which can be represented by a rational linear combination of $A_1$ and $A_2$, can also be represented
by an integral linear combination of $A_1$ and $A_2$. We proceed by induction. Assume that $i < n - 2$ and
that $A_1, \dots, A_i$
are linearly independent and span a saturated subgroup of rank $i$ of $P$. For any linear combination
$B := \sum_{j = 1}^i \alpha_j A_j$, we have $\langle B , A_{i + 2} \rangle = 0$. But $\langle A_{i + 1},
A_{i + 2} \rangle = 1$ and therefore $A_{i + 1}$ cannot
be such a linear combination and thus is linearly independent of $A_1, \dots, A_i$. From integrality of
the bilinear form it follows that $A_1, \dots, A_{i + 1}$ forms a saturated subgroup of $P$. So by induction
$A_1, \dots, A_{n - 2}$ is a basis of $P$.

By Gale duality, we thus obtain a sequence of integral vectors $l_1, \dots, l_n$ in $N \cong \Z^2$, where every
pair $l_i, l_{i + 1}$ with $i \in \on$ forms a basis of $N$. Consider the quotient $P/A_i^\bot
\cong \Z$ for any $i$. By choosing an appropriate generator of $P/A_i^\bot$, we can identify the images
of $A_{i - 1}$ and $A_{i + 1}$ with $1$ and the image of $A_i$ with $a_i$. If we consider these elements
as the Gale
duals of $l_{i - 1}, l_i, l_{i + 1}$ alone, we see that for every $i$ we have a unique relation
$l_{i - 1} + a_i l_i + l_{i + 1} = 0$ for $a_i = \langle A_i, A_i \rangle \in \Z$.

It only remains to show that for every $l_k$ there do not exist $l_i, l_{i + 1}$ and $\alpha_i,
\alpha_{i + 1} \geq 0$ such that $l_k = \alpha_i l_i + \alpha_{i + 1} l_{i + 1}$. As the $l_i, l_{i + 1}$
form bases of $N$ for every $i$, we see that the ordering (clockwise or counterclockwise) of the $l_i$ might
result in several ``windings'' until closing up with the final pair $l_n, l_1$. Assume that we partition
the $l_i$ according to such windings, i.e. we group them to $W_1 = \{l_1, \dots, l_{k_1}\}$, $W_2 = \{
l_{k_1 + 1}, \dots, l_{k_2}\}, \dots, W_r = \{l_{k_{r - 1} + 1}, \dots, l_{k_r}\}$, where $k_r = n$. For every
two windings $W_j$, $W_{j + 1}$, we get that there exist $\alpha_j, \alpha_{j + 1}$ such that
$l_1 = \alpha_j l_{k_j} + \alpha_{j + 1} l_{k_{j + 1}}$. We now add additional rays: first, we add
$l_1^j = l_1$ for every $W_j$, second we add rays between $l_{k_j}$ and $l_1^j$ and between $l_1^j$ and
$l_{k_{j + 1}}$ such that any two neighbouring rays are lattice bases of $N$. This way, we obtain a stack
of $r$ fans in $N$, each of which corresponds to a smooth toric surface. We denote $n'$ the total number of
rays after performing this process and $a_i'$ the new intersection numbers; then we get by Propositions
\ref{twelveminusn} and \ref{smoothblowdownproposition}:
\begin{equation*}
\sum_{i} a_i' = \sum_k a_k - 3 (n' - n) = 12 r - 3 n' \Rightarrow \sum_{i} a_i = 12 r - 3 n = 12 - 3 n,
\end{equation*}
hence $r = 1$.
\end{proof}

So we define:

\begin{definition}
Let $\mathcal{A} = A_1, \dots, A_n$ be an abstract toric system, then we write $Y(\mathcal{A})$ for the
associated toric surface.
\end{definition}

As we have seen, toric systems provide an alternative way to describe toric surfaces. Assume $X$ is a toric
surface, specified by lattice vectors $l_1, \dots, l_n$ in $N$ and $D_1, \dots, D_n$ the associated torus
invariant divisors, which form a toric system. Then an equivariant blow-down $X \rightarrow X'$ is described
by removing some $l_i$ with $l_i = l_{i - 1} + l_{i + 1}$. This induces an embedding of $\pic(X')$ in $\pic(X)$
as a hyperplane such that $D_i . D = 0$ for all $D \in \pic(X')$. This corresponds to removing $D_i$ and
projecting $D_1, \dots, \widehat{D_i}, \dots D_n$ to $\pic(X')$. More explicitly, for abstract toric systems
this can be formulated as:

\begin{lemma}
Let $A_1, \dots, A_n$ be an abstract toric system in $P$ and $i$ such that $\langle A_i, A_i \rangle = -1$. Then
$A_1, \dots, A_{i - 2}, A_{i - 1} + A_i, A_{i + 1} + A_i, A_{i + 2}, \dots A_n$ is a toric system
as well which is contained in the hyperplane $A_i^\bot$ with intersection product
$\langle \, , \, \rangle\vert_{A_i^\bot}$.
\end{lemma}

\begin{proof}
Denote $L := (l_1, \dots, l_n)$ the matrix whose columns are the $l_i$, $L' := (l_1, \dots, \widehat{l_i},
\dots, l_n)$, and consider $A := (A_1, \dots, A_n)$ as $n$-tuple of linear forms on $P^*$. Then the statement is
equivalent to describing the map $A'$ with respect to in the following diagram:
\begin{equation*}
\xymatrix{
0 \ar[r] & (P')^* \ar@{^{(}->}[d] \ar^{A'}[r] & \Z^{n - 1} \ar@{^{(}->}[d] \ar^{L'}[r] & \Z^2 \ar@{=}[d]
\ar[r] & 0\\
0 \ar[r] & P^* \ar^{A}[r] & \Z^n \ar^{L}[r] & \Z^2 \ar[r] & 0,
}
\end{equation*}
which is a straightforward computation.
\end{proof}

So we denote:

\begin{definition}
Let $A_1, \dots, A_n$ be an abstract toric system and $i$ such that $\langle A_i, A_i \rangle = -1$. Then we
call $A_1, \dots, A_{i - 2}, A_{i - 1} + A_i, A_{i + 1} + A_i, A_{i + 2}, \dots, A_n$ its {\em blow-down}.
\end{definition}

For a given abstract toric system $\mathcal{A}$, the sum $\sum_i A_i$ corresponds to the anti-canonical divisor
of $Y = Y(\mathcal{A})$. A small computation shows that the Euler characteristics of the $-A_i$ vanishes:

\begin{lemma}\label{toricsystemonquadric}
Let $\mathcal{A} = \{A_1, \dots, A_n\}$ be an abstract toric system, then for all $i$:
\begin{equation*}
\chi_{Y(\mathcal{A})} (-A_i) = 1 + \frac{1}{2}(\langle A_i, A_i \rangle - \langle \sum_j A_j, A_i \rangle)
= 0.
\end{equation*}
\end{lemma}

\begin{proof}
We just note that $\langle \sum_j A_j , A_i \rangle = \langle A_{i - 1}, A_i \rangle + \langle A_i, A_i \rangle
+ \langle A_{i + 1}, A_i \rangle = 2 + \langle A_i, A_i \rangle$.
\end{proof}

Note that in general for two given toric systems $\mathcal{A}$ and $\mathcal{A}'$ the sums
$\sum_{i = 1}^n A_i$ and $\sum_{i = 1}^n A_i'$ do not coincide. This can most trivially be
seen in the case where $\mathcal{A}' = -\mathcal{A}$. Any integral orthogonal
transformation maps toric systems to toric systems and in general such transformations
do not leave $\sum_{i = 1}^n A_i$ invariant, as we show in the following example.

\begin{example}\label{abstracttoricsystemexample}
As in the introduction, consider $X$ to be a $t$-fold blow-up of $\mathbb{P}^2$ with
$H, R_1, \dots, R_t$ a basis of $\pic(X)$. Denote
\begin{equation*}
\mathfrak{R}^{i} = \{E \in \pic(X) \mid \chi(-E) = 0 \text{ and } -K_X . E = i\}
\end{equation*}
for every $i \in \Z$. It follows from the Riemann-Roch formula that $E^2 = i - 2$ for
every $E \in \mathfrak{R}^i$ (compare Lemma \ref{acyclicitylemma} below). Now, for
any $i, s \in \Z$ with $(i - 2) s = -2$, and any $E \in \mathfrak{R}^{i}$ we can define a
reflection $r_E$ on $\pic(X)$ by setting
\begin{equation*}
r_E(D) = s (E.D) E + D
\end{equation*}
for any $D \in \pic(X)$. Such a reflection clearly respects the intersection product.
However, by definition, such a reflection preserves the anti-canonical divisor if
and only if $E \in \mathfrak{R}^0$. If we take the abstract toric system
\begin{equation*}
R_1 - R_2, R_2 - R_3, \dots, R_{t - 1} - R_t, R_t, H - \sum_{i = 1}^t R_i, H, H - R_1
\end{equation*}
from the introduction and apply, say, $r_{R_1}$ to it, where $R_1 \in \mathfrak{R}^1$,
then we obtain
\begin{equation*}
-R_1 - R_2, R_2 - R_3, \dots, R_{t - 1} - R_t, R_t, H + R_1 - \sum_{i = 2}^t R_i, H, H + R_1.
\end{equation*}
These divisors add up to $r_{R_1}(-K_X) = 3H + R_1 - \sum_{i = 2}^t R_i = -K_X + 2R_1$.
\end{example}

For constructing and analyzing abstract toric systems, we will need a weaker version:

\begin{definition}\label{shorttoricsystemdef}
Let $P$ be a free $\Z$-module of rank $n - 2$ together with a integral symmetric bilinear form
$\langle\, , \, \rangle$. A sequence of elements $A_1, \dots, A_r$ with $r < n$ in $P$ is called
a {\em short toric system} if it satisfies the following conditions:
\begin{enumerate}[(i)]
\item $\langle A_i , A_{i + 1} \rangle = 1$ for $1 \leq i < r$ and
$\langle A_1, A_r \rangle = 1$;
\item $\langle A_i , A_j \rangle = 0$ for $i \neq j$, $\{i, j\} \neq \{1, r\}$, and $\{i, j\} \neq
\{k, k + 1\}$ for all $k \in [r - 1]$.
\end{enumerate}
\end{definition}

There are two natural ways for constructing short toric systems from abstract toric systems:

\begin{example}\label{shorttoricsystemexample1}
Let $A_1, \dots, A_n$ be an abstract toric system, $t > 1$ and $I_1, \dots, I_t \subset \on$ a partition
of $\on$ into cyclic intervals such that $I_j \cup I_{j + 1}$ ($I_1 \cup I_t$, respectively) form a cyclic
interval
for every $j$. Let $A'_j = \sum_{k \in I_j} A_k$, then $A'_1, \dots, A'_t$ is a short toric system.
\end{example}

\begin{example}\label{shorttoricsystemexample2}
Let $X$ be a smooth complete rational surface and $b: X' \rightarrow X$ a blow-up. If $A_1, \dots, A_n$
is an abstract toric system in $\pic(X)$ with respect to the intersection form, then $b^* A_1, \dots,
b^* A_n$ is a short toric system in $\pic(X')$.
\end{example}

\section{Rational surfaces and toric systems}\label{rationalsurfacesandtoricsystems}

Let $X$ be a smooth complete rational surface. From now on we fix $n := \pic(X) + 2$. Recall that
on a rational surface every invertible sheaf is exceptional. For any two divisors $D, E$ on $X$,
we have natural isomorphisms $\Ext^i_{\sh{O}_X}\big(\sh{O}(D), \sh{O}(E)\big) \cong H^i\big(X,
\sh{O}(E - D)\big)$. Let $E_1, \dots, E_n \in \pic(X)$ such that $\sh{O}(E_1), \dots, \sh{O}(E_n)$
form an exceptional sequence, then $H^k\big(X, \sh{O}(E_i - E_j)\big) = 0$ for all $i > j$
and every $k \geq 0$. If, moreover, the sequence is strongly
exceptional, we additionally get $H^k\big(X, \sh{O}(E_i - E_j)\big) = 0$ for all $i, j$ and all
$k > 0$. This leads to the following definition:

\begin{definition}
Let $D \in \pic(X)$, then $D$ is called
\begin{enumerate}[(i)]
\item {\em numerically left-orthogonal to $\sh{O}_X$} if $\chi(-D) = 0$,
\item {\em left-orthogonal to $\sh{O}_X$} if $h^i(-D) = 0$ for all $i$, and
\item {\em strongly left-orthogonal to $\sh{O}_X$} if it is left-orthogonal to $\sh{O}_X$ and
$h^i(D) = 0$ for all $i > 0$.
\end{enumerate}
\end{definition}

We will usually omit the reference to $\sh{O}_X$ and simply say that $D$ is, e.g. left-orthogonal. The
strength of above conditions is completely determined by $h^1$-vanishing:

\begin{lemma}\label{h1vanishinglemma}
Let $D \in \pic(X)$ be numerically left-orthogonal. Then $D$ is left-orthogonal iff
$h^1(-D) = 0$. If $-K_X$ is effective, then $D$ is strongly left-orthogonal iff $h^1(-D) = h^1(D) = 0$.
\end{lemma}

\begin{proof}
By assumption $\chi(-D) = 0$. Then clearly $h^1(-D) = 0$ iff $h^0(-D) = h^2(-D) = 0$ iff
$D$ is left-orthogonal. It remains to show the ``strongly'' part for $h^1(D) = 0$. For this
we have to show that $h^2(D) = 0$. By Serre duality, we have $h^2(D) = h^0(K_X - D)$. If
$h^0(K_X - D) \neq 0$, we get an inclusion $h^0(-K_X) \subset h^0(-D)$, but this is impossible,
because $h^0(-D) = 0$ and $-K_X$ is effective.
\end{proof}

By Riemann-Roch we have $\chi(D) = 1 + \frac{1}{2}(D^2 - K_X . D)$ for any divisor $D$, by which we get by
symmetrization and anti-symmetrization:
\begin{align*}
\chi(D) + \chi(-D) & = 2 + D^2 \quad \text{ and}\\
\chi(D) - \chi(-D) & = -K_X . D.
\end{align*}
By numerical left-orthogonality we have $\chi(-D) = 1 + \frac{1}{2}(D^2 + K_X . D) = 0$
(compare this also to Lemma \ref{toricsystemonquadric}), which directly implies:

\begin{lemma}\label{acyclicitycorollary}\label{acyclicitylemma}
Let $D, E \in \pic(X)$ numerically left-orthogonal, then
\begin{enumerate}[(i)]
\item\label{acyclicitylemmai}
$\chi(D) = -K_X . D$;
\item\label{acyclicitycorollaryv}
$D^2 = \chi(D) - 2$; in particular, if $D$ is strongly left-orthogonal, then $D^2 = h^0(D) - 2$;
\item\label{acyclicitycorollaryii}
$D + E$ is numerically left-orthogonal iff $E . D = 1$ iff $\chi(D) + \chi(E) = \chi(D + E)$;
in particular, if $D, E, E + D$ are strongly left-orthogonal, then $h^0(D + E) = h^0(D) + h^0(E)$;
\item\label{acyclicitycorollaryiv}
$E - D$ is numerically left-orthogonal iff $D . E = \chi(D) - 1$; in particular, if $D, E, E - D$
are strongly left-orthogonal, then $h^0(D) \leq h^0(E)$ and $D . E = h^0(D) - 1$.
\end{enumerate}
\end{lemma}

Clearly, if $\sh{O}(E_1), \dots \sh{O}(E_n)$ is a full exceptional sequence, then $n = \rk K_0(X) = \rk \pic(X)
+ 2$ and all the differences
$E_j - E_i$ for $i > j$ are left-orthogonal and in particular numerically left-orthogonal. We set
$A_i := E_{i + 1} - E_i$ for
$1 \leq i < n$ and $A_n := -K_X - \sum_{i = 1}^{n - 1} A_i$. Then
by Lemma \ref{acyclicitycorollary} we get:
\begin{enumerate}
\item\label{quasiconditionsii} $A_i . A_{i + 1} = 1$ for $i \in \on$;
\item\label{quasiconditionsiii} $A_i . A_j = 0$  for $i \neq j$ and $\{i, j\} \neq \{k, k + 1\}$ for some
$k \in \on$;
\item\label{quasiconditionsi} $\sum_{i = 1}^n A_i = -K_X$.
\end{enumerate}
Therefore we get an abstract toric system from an exceptional sequence. Note that in general not every
abstract toric system can be of this form, as $\sum_{i = 1}^n A_i = -K_X$ implies $(\sum_{i = 1}^n A_i)^2 =
12 - 3n$, but not vice versa, as example \ref{abstracttoricsystemexample} shows. But with this stronger
condition, we pass from abstract toric systems to actual toric systems:

\begin{definition}
Let $X$ a smooth complete rational surface. Then a {\em toric system} (on $X$) is an abstract toric system
$A_1, \dots, A_n \in \pic(X)$ such that $\sum_{i = 1}^n A_i = -K_X$.
\end{definition}

Note that after passing from the $E_1, \dots, E_n$ to $\mathcal{A} = A_1, \dots, A_n$, the construction
of the toric surface $Y(\mathcal{A})$ is entirely canonical. In particular, we conclude
the following remarkable observation:

\begin{theorem}\label{canonicaltoricsystemtheorem}
Let $X$ be a smooth complete rational surface. Then to any full exceptional sequence of invertible sheaves
on $X$ with associated toric system $\mathcal{A}$ we can
associate in a canonical way a smooth complete toric surface $Y(\mathcal{A})$ with torus invariant
prime divisors $D_1, \dots, D_n$ such that $D_i^2 = A_i^2$ for every $i \in \on$.
\end{theorem}

A toric system generates an infinite sequence of invertible sheaves
\begin{equation*}
\dots, \sh{O}(-A_n), \sh{O}_X, \sh{O}(A_1), \sh{O}(A_1 + A_2), \dots,
\sh{O}(\sum_{i = 1}^{n - 1}A_i), \sh{O}(-K_X), \sh{O}(-K_X + A_1), \dots
\end{equation*}
If some subsequence of length $n$ of this sequence is a strongly exceptional sequence, we will follow
the convention that the toric system is enumerated such that this sequence can be written as
$\sh{O}_X$, $\sh{O}(A_1)$, $\sh{O}(A_1 + A_2)$, $\dots$, $\sh{O}(\sum_{i = 1}^{n - 1}A_i)$. In
particular, $\sum_{i \in I} A_i$ is strongly left-orthogonal for every interval $I \subset [n - 1]$.
In general we will assume nothing about the strong left-orthogonality of $A_n$. If the toric system gives
rise to a cyclic strongly exceptional sequence, then $\sum_{i \in I} A_i$ is strongly left-orthogonal
for every {\em cyclic} interval $I \subset \on$.

\begin{definition}
We say that a toric system $A_1, \dots, A_n$ is (cyclic, strongly) {\em exceptional} if the associated
sequence of invertible sheaves $\sh{O}_X$, $\sh{O}(A_1)$, $\dots$, $\sh{O}(\sum_{i = 1}^{n - 1}A_i)$
generates a (cyclic, strongly) exceptional sequence.
\end{definition}

Note that a priori a toric system and the conditions on cohomology vanishing do not completely determine
the ordering
of the $A_i$. In particular, if $A_1, \dots, A_n$ is a cyclic (strongly) exceptional
toric system, then
so is $A_n, \dots, A_1$. If $A_1, \dots, A_n$ is a (strongly) exceptional toric system, then so is
$A_{n - 1}, \dots, A_1, A_n$.

\section{Left-orthogonal divisors on rational surfaces}\label{rationalacyclicsection}

Any smooth complete rational surface $X$ can be obtained by a sequence of blow-ups $X = X_t
\overset{b_t}{\longrightarrow} X_{t - 1} \overset{b_{t - 1}}{\longrightarrow} X_{t - 2}
\overset{b_{t - 2}}{\longrightarrow} \cdots \overset{b_1}{\longrightarrow} X_0$, where $X_0$ is
either $\mathbb{P}^2$ or some Hirzebruch surface $\mathbb{F}_a$. If we fix the sequence of
morphisms $b_t, \dots, b_1$, we obtain a natural basis of $\pic(X)$ with respect to this sequence
as follows. If $X_0 = \mathbb{P}^2$, we denote as before $H$ the hyperplane class of $\mathbb{P}^2$, and
for every $b_i$, we denote $R_i$ the class of the associated exceptional divisor in $\pic(X_i)$.
For simplicity, we identify $H$ and the $R_i$ with their pullbacks in $\pic(X)$. Every blow-up
increases the rank of the Picard group by one and the pullback yields an inclusion of $\pic(X_{i - 1})$
into $\pic(X_i)$ as a hyperplane. Then $R_i$ is additional generator, which is orthogonal to
$\pic(X_{i - 1})$ with respect to the intersection product. We have the following relations:
\begin{equation*}
H^2 = 1, \quad R_i^2 = -1, \quad H.R_i = 0 \text{ for all } i, \quad \text{ and } R_i. R_j = 0 \quad
\text{ for all } i \neq j.
\end{equation*}
In particular, we have $t = \rk \pic(X) - 1$. So, in the case where $X$ is a blow-up of $\mathbb{P}^2$,
we easily get a basis of $\pic(X)$ which diagonalizes the intersection product. In the case where
$X_0 = \mathbb{F}_a$ for some $a \geq 0$, we start with a basis $P, Q$ of $\pic(\mathbb{F}_a)$ as before,
and by the same process, we obtain a basis $P, Q, R_1, \dots, R_t$ of $\pic(X)$, where $t = \rk \pic(X)
- 2$.
Here, the most convenient choice for our purpose is $P, Q$ to be the integral generators of the nef
cone in $\pic(\mathbb{F}_a)_\Q$ such that $P^2 = 0$ and $Q^2 = a$. So we get
\begin{align*}
& P^2 = 0, \quad Q^2 = a, \quad P.Q = 1, \quad R_i^2 = -1, \\ & P.R_i = Q.R_i = 0 \text{ for all } i,
\ \text{ and } R_i. R_j = 0 \ \text{ for all } i \neq j.
\end{align*}
Often our arguments below do not depend on the choice of $X_0$, and for simplicity we will often
leave this choice implicit and assume that $t = n - 3$ or $t = n - 4$ as it fits.

\begin{definition}
Let $D \in \pic(X)$, then we denote the projection of $D$ to $\pic(X_i)$ by $(D)_i$.
\end{definition}

The projection $(D)_i$ just is `forgetting' the coordinates $R_t$, $R_{t - 1}$, $\dots, R_{i + 1}$,
i.e. if $D = \alpha P + \beta Q + \sum_{j = 1}^t \gamma_j R_j$ or $D = \beta H + \sum_{j = 1}^t
\gamma_j R_j$, respectively, then $(D)_i = \alpha P + \beta Q + \sum_{j = 1}^i \gamma_j R_j$ or
$(D)_i = \beta H + \sum_{j = 1}^i \gamma_j R_j$, respectively.

By Lemma \ref{h1vanishinglemma}, left-orthogonality is determined by numerical left-orthogonality and
$h^1$-vanishing. Our strategy to understand (strongly) left-orthogonal divisors will be to start with
$h^1$-vanishing and then to establish numerical left-orthogonality. For this, we first need a couple
of lemmas related to $h^0$- and $h^1$-vanishing.

\begin{lemma}\label{coordinatedegreelemma}
Let $E$ be an irreducible $(-1)$-divisor and $X'$ the surface obtained from blowing down $E$. If
$D$ is the pullback to $X$ of some divisor on $X'$, then for every $k \in \Z$, we have
$\deg \sh{O}(D + kE)\vert_E = -k$.
\end{lemma}

\begin{proof}
For $k \in \Z$ consider the short exact sequence 
\begin{equation*}
0 \longrightarrow \sh{O}(D + (k - 1)E) \longrightarrow \sh{O}(D + kE) \longrightarrow \sh{O}_E(D + kE)
\longrightarrow 0.
\end{equation*}
Then, for the Euler characteristics, we get $\chi\big(\sh{O}_E(D + kE)\big) = \chi(D + kE) - \chi(D +
(k - 1)E)$ $= 1 - k$, where the latter equality follows from Riemann-Roch and $D . E = 0$. Hence
$\sh{O}_E(D + kE) \cong \sh{O}_E(-k)$ and the assertion follows.
\end{proof}

We use this to investigate $h^0$- and $h^1$-vanishing. If a divisor has nonzero $h^1$, then so has its
preimage under blow-up. For $h^0$ and $h^2$, we have the opposite picture:

\begin{lemma}\label{pullbackcohomologyvanishinglemma}
Let $D$ and $E$ as in Lemma \ref{coordinatedegreelemma}.
\begin{enumerate}[(i)]
\item\label{pullbackcohomologyvanishinglemmai} If $h^0(D) = 0$, then $h^0(D + kE) = 0$ for all $k \in \Z$.
\item\label{pullbackcohomologyvanishinglemmaii} If $h^1(D) \neq 0$, then $h^1(D + kE) \neq 0$ for all $k \in
\Z$.
\item\label{pullbackcohomologyvanishinglemmaiii} If $h^2(D) = 0$, then $h^2(D + kE) = 0$ for all $k \in \Z$.
\end{enumerate}
\end{lemma}

\begin{proof}
For $k = 0$ there is nothing to prove. If $k > 0$, we do induction on $k$. Consider the short
exact sequence
\begin{equation*}
0 \longrightarrow \sh{O}(D + (k - 1)E) \longrightarrow \sh{O}(D + k E) \longrightarrow \sh{O}_E(D + kE)
\longrightarrow 0.
\end{equation*}
By lemma \ref{coordinatedegreelemma}, we have $\deg \sh{O}(D + kE)\vert_E = -k$ and therefore
$h^0\big(\sh{O}(D + kE)\big) = 0$. So by the long exact cohomology sequence we get $h^0(D + (k - 1)E) =
h^0(D + k E)$, $h^1(D + (k~-~1)E) \leq h^1(D + k E)$, and $h^2(D + (k~-~1)E) \geq h^2(D + k E)$.
For (\ref{pullbackcohomologyvanishinglemmai}), we have by induction assumption $h^0(D + (k - 1)E) = 0$ and so
$h^0(D + k E) = 0$. For (\ref{pullbackcohomologyvanishinglemmaii}), we have by induction assumption
$h^1(D + (k - 1)E) > 0$ and so $h^1(D + k E) > 0$. For (\ref{pullbackcohomologyvanishinglemmaiii}), we
have by induction assumption $h^2(D + (k - 1)E) = 0$ and so $h^2(D + k E) = 0$.

For $k < 0$, we do induction from $k + 1$ to $k$. In this case, we consider the short exact sequence
\begin{equation*}
0 \longrightarrow \sh{O}(D + k E) \longrightarrow \sh{O}(D + (k + 1) E) \longrightarrow \sh{O}_E(D + (k + 1)E)
\longrightarrow 0.
\end{equation*}
So $\deg \sh{O}(D + (k + 1)E)\vert_E = -k - 1 \geq 0$ and therefore $h^1\big(\sh{O}(D + (k + 1)E)\big) = 0$.
Then by the long exact cohomology sequence, we get $h^0(D + kE) \leq h^0(D + (k + 1)E)$, $h^1(D + kE) \geq
h^1(D + (k + 1)E)$, and $h^2(D + kE) = h^2(D + (k + 1)E)$.
For (\ref{pullbackcohomologyvanishinglemmai}), we have by induction assumption
$h^0(D + (k + 1)E) = 0$ and so $h^0(D + k E) = 0$. For (\ref{pullbackcohomologyvanishinglemmaii}), we have
by induction assumption $h^1(D + (k + 1)E) > 0$ and so $h^1(D + k E) > 0$.
For (\ref{pullbackcohomologyvanishinglemmaiii}), we have by induction assumption
$h^2(D + (k + 1)E) = 0$ and so $h^2(D + k E) = 0$.
\end{proof}

\begin{definition}
Let $D \in \pic(X)$ with $(D)_0 \neq 0$. Then we call $D \in \pic(X)$ {\em pre-left-orthogonal} with
respect to $X_0$ iff $h^0\big((-D)_0\big) = h^1(-D) = 0$, and
{\em strongly pre-left-orthogonal} if it is pre-left-orthogonal and $h^1(D) = 0$.
\end{definition}

Note the little twist that for pre-left-orthogonality we do not just require $h^0$-vanishing, but instead
have conditions on $X_0$. This makes the following an immediate consequence of Lemma
\ref{pullbackcohomologyvanishinglemma}:

\begin{corollary}
If $D$ is pre-left-orthogonal, then so is $(D)_i$ for $i = 1, \dots, t$.
\end{corollary}

If $D$ is a pre-left-orthogonal divisor on $X_{t - 1}$, then in general $D + \gamma_t R_t$ will only be
pre-left-orthogonal for a few possible values of $\gamma_t$. The following lemma gives some sufficient
conditions.

\begin{lemma}\label{unwindinglemma}
Let $D \in \pic(X)$ and $k \geq l \geq 0$. If $D - kR_t$ is pre-left-orthogonal, then $D - lR_t$ is also
pre-left-orthogonal. If $D - k R_t$ is strongly pre-left-orthogonal, then so is $D - lR_t$.
\end{lemma}

\begin{proof}
We do both cases by induction on $l$, starting with $l = k$. For $l = k$, there is nothing to show. Also
$(D - kR_t)_0 = (D - l R_t)_0$, so there is nothing to show for $h^0$.
Assume now
that $k > l > 0$ and $D - lR_t$ is pre-left-orthogonal. We consider the short exact sequence
\begin{equation*}
0 \longrightarrow \sh{O}(-D + (l - 1)R_t) \longrightarrow \sh{O}(-D + lR_t) \longrightarrow
\sh{O}_{R_t}(-D + lR_t) \longrightarrow 0.
\end{equation*}
By lemma
\ref{coordinatedegreelemma} we have $\deg \sh{O}_E(-D + lR_t) = -l < 0$, and thus
$h^0\big(\sh{O}_{R_t}(-D + lR_t) \big) = 0$. Then by the long exact cohomology sequence $h^1(-D + (l - 1)R_t)
\leq h^1(-D + lR_t) = 0$ and the first assertion follows by induction. If $D - lE$ is strongly
pre-left-orthogonal, we consider the following short exact sequence
\begin{equation*}
0 \longrightarrow \sh{O}(D - lR_t) \longrightarrow \sh{O}(D - (l - 1)R_t) \longrightarrow \sh{O}_{R_t}(D -
(l - 1)R_t) \longrightarrow 0.
\end{equation*}
Again, by lemma \ref{coordinatedegreelemma} he have $\deg \sh{O}_{R_t}(D - (l - 1) R_t) = l - 1 \geq 0$ and
therefore $h^1\big(\sh{O}_{R_t}(D - (l - 1)R_t) \big) = 0$. Then by the long exact cohomology sequence, we
have $0 = h^1(D - lR_t) \geq h^1(D - (l - 1)R_t) \geq 0$ and the second assertion follows by induction.
\end{proof}

Now we classify (strongly) pre-left-orthogonal divisors on $\mathbb{P}^2$ and on the $\mathbb{F}_a$.
Denote $H$ the class of a line on $\mathbb{P}^2$. As the condition of $h^1$-vanishing is vacuous for
invertible sheaves on $\mathbb{P}^2$, we trivially observe:

\begin{proposition}\label{P2leftorthogonals}
A divisor on $\mathbb{P}^2$ is (pre-)left-orthogonal iff it is strongly (pre-)left-orthogonal. The
pre-left-orthogonal divisors are given by $k H$, where $k > 0$, and the left-orthogonal divisors are
$H, 2H$.
\end{proposition}

For the case of a Hirzebruch surface $\mathbb{F}_a$, we choose $P, Q$ as before and the following
statements can be seen rather straightforwardly, for instance by using toric methods as in
\cite{HillePerling06}, \cite{perling07a}.

\begin{proposition}\label{Hirzebruchpreleftorthogonals}
The pre-left-orthogonal divisors on a Hirzebruch surface are:
\begin{enumerate}[(i)]
\item on $\mathbb{F}_0$: $P + kQ$, $kP + Q$ for $k \in \Z$, $k P + l Q$ for $k, l \geq 2$;
\item on $\mathbb{F}_a$, with $a > 0$: $P$, $k P + Q$ for $k \in \Z$, $k P + l Q$ for $k \geq 1 - a$ and
$l \geq 2$;
\end{enumerate}
A pre-left-orthogonal divisors is strongly pre-left-orthogonal iff it is not of the type $P + kQ$ or $k P + Q$
for $k < -1$ or of type $k P + l Q$ for $l \geq 2$ and $k < \max \{-1, 1 - a\}$.
\end{proposition}

\begin{proposition}\label{Hirzebruchleftorthogonals}
Let $\mathbb{F}_a$ be a Hirzebruch surface.
\begin{enumerate}[(i)]
\item If $a = 0$, then the left-orthogonal divisors are given by $P + kQ$, $kP + Q$ for $k \in \Z$.
\item If $a > 0$, then the left-orthogonal divisors are given by $P$, $k P + Q$ for $k \in \Z$,
and $(1 - a)P + 2Q$.
\item Left-orthogonal divisors of type $k P + Q$ or $P + kQ$ are strongly left-orthogonal iff $k \geq -1$.
Divisors of type $(1 - a)P + 2Q$ are strongly left-orthogonal iff $a \leq 2$.
\end{enumerate}
\end{proposition}

In coordinates chosen with respect to a minimal model $X_0$, the anti-canonical divisor on $X$ can be
written as
\begin{align*}
-K_X & = 3H - \sum_{i = 1}^t R_i \quad \text{ or} \\
-K_X & = (2 - a) P + 2Q - \sum_{i = 1}^t R_i, \text{ respectively}.
\end{align*}
For $X_0 = \mathbb{P}^2$ and some divisor $D = \beta H + \sum_{i = 1}^t \gamma_i R_i$, we get by Riemann-Roch
the following formulas for the Euler characteristics of $D$:
\begin{align}
\chi(D) & = \binom{\beta + 2}{2} - \sum_i \binom{\gamma_i}{2} \label{eulercharp2}\\
\chi(-D) & = \binom{\beta - 1}{2} - \sum_i \binom{\gamma_i + 1}{2},\label{antieulercharp2}
\end{align}
where we write $\binom{x}{2} = \frac{1}{2} x (x - 1)$ for any $x \in \mathbb{Z}$.
For $X_0 = \mathbb{F}_a$ and $D = \alpha P + \beta Q + \sum_{i = 1}^t \gamma_i R_i$, we get:
\begin{align}
\chi(D) & = (\alpha + 1)(\beta + 1) + a \binom{\beta + 1}{2} - \sum_i \binom{\gamma_i}{2} \label{eulercharfa}\\
\chi(-D) & = (\alpha - 1)(\beta - 1) + a \binom{\beta}{2} - \sum_i \binom{\gamma_i + 1}{2}\label{antieulercharfa}
\end{align}
If $\chi(-D) = 0$, we obtain linear equations for $\chi(D) = -K_X D $ in either coordinates:
\begin{align*}
\chi(D) & = 3 \beta + \sum_i \gamma_i \\
\chi(D) & = 2\alpha + (2 + a)\beta + \sum_i \gamma_i.
\end{align*}

We now look at the case where $(D)_0 = 0$. In this case, we have to take into account the relative
configuration of $R_i$ and $R_j$.

\begin{definition}
Assume $i, j > 0$ and denote $x_j$ and $x_i$ the points on $X_{j - 1}$ and $X_{i - 1}$, respectively, which
are blown up by the maps $b_j$ and $b_i$. We define a partial order $\geq$ on the set $\{R_1, \dots, R_t\}$ by
setting $R_i \geq R_i$ for every $i$ and $R_j \geq R_i$ if $j > i$ and $b_i \circ \cdots \circ b_{j - 1}(x_j) =
x_i$.
\end{definition}

Now we get:

\begin{proposition}\label{verticalleftorthogonals}
Let $D \in \pic(X)$ such that $(D)_0 = 0$. Then $D$ is left-orthogonal if there exists $i \in [t]$ and
$S \subset [t] \setminus \{i\}$ such that $D = R_i - \sum_{j \in S} R_j$ and $R_i \ngeq R_j$ for all $j \in S$.
Moreover, $D$ is strongly left-orthogonal
iff it is of the form $R_i$ for some $i \in [t]$ or of the form $R_i - R_j$ such that $R_i$ and $R_j$ are
incomparable with respect to the partial order $\geq$.
\end{proposition}

\begin{proof}
Note that for $(D)_0 = 0$, by Lemma \ref{pullbackcohomologyvanishinglemma}
(\ref{pullbackcohomologyvanishinglemmaiii}), we can always assume that $h^2(D) = h^2(-D) = 0$.
Let $D = \sum_i \gamma_i R_i$, then $\chi(-D) = 0$ by formula (\ref{antieulercharp2}) or
(\ref{antieulercharfa}) yields:
\begin{equation*}
\sum_j \binom{\gamma_j + 1}{2} = 1.
\end{equation*}
But then there is precisely one $i \in [t]$ with $\gamma_i \in \{1, -2\}$ and $\gamma_j \in \{0, -1\}$ for
all other $j$. If $\gamma_i = -2$, we consider $R_i$ as irreducible divisor on $X_i$ and we consider
the following part of a long exact cohomology sequence:
\begin{equation*}
H^1\big(X_i, \sh{O}_{X_i}(\sum_{j \in S}R_j)\big) \longrightarrow H^1\big(X_i, \sh{O}_{X_i}(2R_i +
\sum_{j \in S}R_j)\big) \longrightarrow H^1\big(X_i, \sh{O}_{2R_i}(2R_i + \sum_{j \in S}R_j)\big)
\longrightarrow 0
\end{equation*}
for some $S \subset [i]$.
As $\chi(\sum_{j \in S}R_j) = 1 = h^0(\sum_{j \in S}R_j)$, we get $h^1\big(\sh{O}_{X_i}(\sum_{j \in S}R_j)
\big) = 0$ and thus $h^1\big(\sh{O}_{X_i}(2R_i + \sum_{j \in S}R_j)\big) = h^1\big(\sh{O}_{2R_i}(2R_i +
\sum_{j \in S}R_j)\big)$. By lemma \ref{pullbackcohomologyvanishinglemma} we can assume without loss of
generality that $i \geq j$ for all $j \in S$. Then we get $\sh{O}_{2R_i}(2R_i + \sum_{j \in S}R_j) \cong
\sh{O}_{2R_i}(2R_i)$ and we compute $\chi\big(\sh{O}_{2R_i}(2R_i)\big) = \chi\big(\sh{O}(2 R_i)\big) - 1 = -1$
and thus $h^1\big(\sh{O}_{X_0}(2R_i + \sum_{j \in S}R_j)\big) \neq 0$.

So we are left with divisors of the form $R_i - \sum_{j \in S} R_j$ for some $S \subset [t]$. By Serre
duality, we have
$h^2(-R_i + \sum_{j \in S} R_j) = h^0(K_X + R_i - \sum_{j \in S} R_j) \leq h^0\big((K_X + R_i -
\sum_{j \in S} R_j)_0\big) = h^0\big((K_X)_0) = 0$. If there exists $k \in S$ such that $R_i
\geq R_k$, then $R_k - R_i$ is effective, and $-R_i + \sum_{j \in S} R_j$ is a sum of effective
divisors and therefore $h^0(-R_i + \sum_{j \in S} R_j) \neq 0$. If there exists $k \in S$ such that
$R_i$ and $R_k$ are incomparable, then we may assume that this $k$ is minimal with respect
to $\geq$. Then $h^0(R_k - R_i) = 0$, and by lemma \ref{pullbackcohomologyvanishinglemma}
we can conclude that $h^0(-R_i + \sum_{j \in S} R_j) = 0$.
The remaining possibility is that $R_j \geq R_i$ for all $j \in S$. In that case, denote $E_i$ the
strict transform on $X$ of the exceptional divisor of the blow-up $b_i$. Then there exists $T_i
\subset [t]$ such that $E_i$ is rationally equivalent to $R_i - \sum_{j \in T_i} R_j$. Then
$-R_i + \sum_{j \in S} R_j$ is rationally equivalent to $\sum_{j \in S \setminus T_i} R_j -
\sum_{j \in T_i \setminus S} R_j - E_i$. If any of $S \setminus T_i$ or $T_i \setminus S$ are
empty, we have $h^0(-R_i + \sum_{j \in S} R_j) = 0$. Otherwise, if any $R_k$, $R_l$ wit
$k \in S \setminus T_i$ and $l \in T_i \setminus S$ are incomparable, then $h^0(R_k - R_l) =
h^0(-R_i + \sum_{j \in S} R_j) = 0$. If not, we choose $k \in T_i \setminus S$. Then there
exists $T_k \subset [t]$ such that $E_k = R_k - \sum_{l \in T_k} R_l$ and we iterate
our previous argument until we get the difference of two incomparable $R_j$ or we can
write $-R_i + \sum_{j \in S} R_j$ as the inverse of an effective divisor.

So, unless there exists $j \in S$ with $R_i \geq R_j$, we can now conclude together with
$\chi(-R_i + \sum_{j \in S} R_j) = 0$ that $h^i(-R_i + \sum_{j \in S} R_j) = 0$ for all $i$. This
shows the first assertion. For strong
left-orthogonality, we necessarily need $\chi(R_i - \sum_{j \in S} R_j) \geq 0$, which is the case iff
$S$ is empty or $S = \{j\}$ for some $j \neq i$. A divisor $R_i$ always is strongly left-orthogonal.
For $R_i - R_j$ we have $\chi(R_i - R_j) = 0$, and $h^1(R_i - R_j) = 0$ is equivalent to $h^0(R_i - R_j)
= 0$. But this is in turn is equivalent to incomparability of $R_i$ and $R_j$.
\end{proof}

For $(D)_0 \neq 0$, we have the following statement:

\begin{proposition}\label{negativegammas}
If $D = (D)_0 + \sum_i \gamma_i R_i$ is left-orthogonal and $(D)_0$ is pre-left-orthogonal, then
$\gamma_i \leq 0$ for all $i$.
\end{proposition}

\begin{proof}
Assume $\chi(-D) = 0$ and $\gamma_k > 0$ for some $k$, then $\chi(-D + \gamma_k R_k) =
\binom{\gamma_k + 1}{2} > 0$. As $h^0\big((-D)_0\big) = 0$, we also have $h^0(-D + \gamma_k R_k)
= 0$. Therefore we have $\chi(-D + \gamma_k R_k) = h^2(-D + \gamma_k R_k) - h^1(-D + \gamma_k R_k) > 0$,
hence $h^2(-D + \gamma_k R_k) > 0$. But by Serre duality, $h^2(-D) = h^0(K_X + D) \geq h^0(K_X + D -
\gamma_k R_k) = h^2(-D + \gamma_k R_k) > 0$, which is a contradiction to the left-orthogonality of $D$, and
the assertion follows.
\end{proof}

\begin{remark}\label{negativegammasremark}
Note that in the case where $D$ is strongly left-orthogonal but $(D)_0$ is not strongly pre-left-orthogonal,
this implies that $h^0\big((D)_0\big) = 0$ and therefore $h^0(D) = 0$. But then $-D$ is left-orthogonal, too,
and $(-D)_0$ is strongly pre-left-orthogonal.
\end{remark}

We now consider some special cases concerning proposition \ref{negativegammas}.

\begin{lemma}\label{specialorthogonals}
Let $X$ be a smooth complete rational surface, $D$ a very ample and strongly left-orthogonal divisor
on $X$. Consider a blow-up $b: \tilde{X} \rightarrow X$  in four points $x_1, x_2, x_3, x_4$, where $x_1$
and $x_2$ are on $X$ and $x_3$ and $x_4$ are infinitesimal
points lying over $x_1$ and $x_2$, respectively. Denote $R_1, \dots, R_4$ the pullbacks of the exceptional
divisors of $b$ to $\pic(\tilde{X})$, then the divisors $D - R_i$ and $D - R_i - R_j$ with $i \neq j$ are
strongly left-orthogonal on $\tilde{X}$.
\end{lemma}

\begin{proof}
It follows directly from our previous discussions that the divisors $D - R_i$ and $D - R_i - R_j$ are
left-orthogonal. It remains to show that $h^1(D - R_i) = h^1(D - R_i - R_j) = 0$. By lemma
\ref{acyclicitylemma} (\ref{acyclicitycorollaryii}) we know $\chi(D - R_i) = \chi(D) - 1$ and
$\chi(D - R_i - R_j) = \chi(D) - 2$. So it suffices to show that $h^0(D - R_i - R_j)< h^0(D - R_i)
< h^0(D)$ for any $i \neq j$. But this is an immediate consequence of \cite{Hart2}, V.4, Remark
4.0.2 and preceding remarks.
\end{proof}

\section{Exceptional sequences of invertible sheaves on rational surfaces}\label{exceptionalsequencessection}

We first show that cyclicity for exceptional sequences of invertible sheaves is no additional condition:

\begin{proposition}\label{exceptionalsequencesalwayscyclic}
Let $X$ be a smooth complete rational surface. Then every exceptional sequence of invertible sheaves is cyclic.
\end{proposition}

\begin{proof}
Let $A_1, \dots, A_n$ be an exceptional toric system. Then for every interval $I \subset [n - 1]$
we have $h^i(-A_I) = 0$ for every $i$. By Serre duality, we get $h^i(-A_I) = h^i(K_X + A_{\on
\setminus I}) = h^{2 - i}(-A_{\on \setminus
I}) = 0$ for every $i$. So $A_J$ is left-orthogonal for every cyclic interval $J$ of $\on$ and $A_1, \dots,
A_n$ corresponds to a cyclic exceptional sequence.
\end{proof}

On $\mathbb{P}^2$, there is a unique toric system which gives rise to a cyclic strongly
exceptional sequence, but, as we will see for the case of Hirzebruch surfaces, Proposition
\ref{exceptionalsequencesalwayscyclic} does not hold for {\em strongly} exceptional
sequences in general. Recall that $P, Q$ are generators of the nef cone
of the Hirzebruch surface $\mathbb{F}_a$, where $P^2 = 0$, $Q^2 = a$, and $P. Q = 1$.

\begin{proposition}\label{Hirzebruchtoricsystems}
On a Hirzebruch surface $\mathbb{F}_a$ there are the following toric systems:
\begin{enumerate}[(i)]
\item\label{Hirzebruchtoricsystemsi} $P, sP + Q, P, -(a + s)P + Q$ for $s \in \Z$ for any $a$;
\item\label{Hirzebruchtoricsystemsii} $-\frac{a}{2}P + Q, P + s (-\frac{a}{2}P + Q), -\frac{a}{2}P + Q, P
- s (-\frac{a}{2}P + Q)$ for $s \in \Z$ and $a$ even.
\end{enumerate}
Toric systems of type (\ref{Hirzebruchtoricsystemsi}) are always exceptional. They are strongly exceptional
for $s \geq -1$, where $A_4 = -(a + s)P + Q$. They are cyclic strongly exceptional iff $s \geq -1$ and
$a + s \leq 1$.

Toric systems of type (\ref{Hirzebruchtoricsystemsii}) are almost never exceptional. The exceptions are
for $a = 0$, where type (\ref{Hirzebruchtoricsystemsii}) is symmetric to type
(\ref{Hirzebruchtoricsystemsi}) by exchanging $P$ and $Q$, and for $a = 2$ and $s = 0$, which then
coincides with a toric system of type (\ref{Hirzebruchtoricsystemsi}) and is cyclic strongly exceptional.
\end{proposition}

\begin{proof}
Any toric system must represent a Hirzebruch surface. Therefore, for any toric system $A_1$, $A_2$,
$A_3$, $A_4$ we can assume that $A_1^2 = A_3^2 = 0$ and $A_2^2 = -A_4^2 = -b$ for some $b \in \Z$.
So for a general element $\alpha P + \beta Q$ with $\alpha, \beta \in \Z$, the equations $(\alpha P
+ \beta Q)^2 = 0$ and $\chi(-\alpha P - \beta Q) = 0$ have always the solution $\alpha = 1$, $\beta = 0$.
If $a$ is even, we get a second solution, $\alpha = -\frac{a}{2}$ and $\beta = 1$. The condition
$A_1 . A_3 = 0$ can only be fulfilled if $A_1 = A_3 = P$, or if $A_1 = A_3 = -\frac{a}{2}P + Q$.

In the first case, using $A_1 . A_2 = A_1 . A_4 = 1$ and $A_2 . A_4 = 0$, we get that $A_2 = sP + Q$
and $A_4 = -(a + s)P + Q$ for some $s \in \Z$ which indeed form a toric system way for every $s \in \Z$.

In the second case with $a$ even, we similarly compute that $A_2 = P + s(-\frac{a}{2}P + Q)$ and
$A_4 = P - s(-\frac{a}{2}P + Q)$ for some $s \in \Z$.

The classification of exceptional sequences (cyclic or strong) among these follows by inspection of the
classification of (strongly) left-orthogonal divisors of proposition \ref{Hirzebruchleftorthogonals}.
\end{proof}

\begin{remark}
From Proposition \ref{Hirzebruchtoricsystems} follows that for a toric system $\mathcal{A}$ on
a Hirzebruch surface $\mathbb{F}_a$, the associated Hirzebruch surface $Y(\mathcal{A})$ is isomorphic to
$\mathbb{F}_b$, where $b - a$ is even.
\end{remark}

As in the previous section, we assume that a sequence of blowups $X = X_t \longrightarrow \cdots
\longrightarrow X_0$ is fixed, where $X_0$ is $\mathbb{P}^2$ or some $\mathbb{F}_a$, together with a
corresponding basis of $\pic(X)$, either $H, R_1, \dots, R_t$ if $X_0 \cong \mathbb{P}^2$, or
$P, Q, R_1, \dots, R_t$ if $X_0 \cong \mathbb{F}_a$.
Any toric system $\mathcal{A} = A_1, \dots, A_{n - t + i}$ on some $X_i$ pulls back to a short toric
system on $X$ in the sense of Definition \ref{shorttoricsystemdef} (see Example
\ref{shorttoricsystemexample2}). Such a short toric system can
easily be extended to a toric system by using the $R_{i + 1}, \dots, R_t$ as follows. For any
$i + 1 \leq j_1 \leq t$ we denote $\mathcal{A}_1$ the sequence
\begin{equation*}
A_1, \dots, A_{s - 1}, A_s - R_{j_1}, R_{j_1}, A_{s + 1} - R_{j_1}, A_{s + 2}, \dots,  A_{n - t + i},
\end{equation*}
which augments $\mathcal{A}$ at some position $s$. Note that this augmentation is understood in the cyclic
sense, i.e. we do not exclude $s = n - t + i$. If $i = t - 1$, then this sequence is a toric system on $X$;
otherwise, it is again a short toric system. Inductively, for
$1 < k < t - i$ we can in the same way augment $\mathcal{A}_{k - 1}$ to a short toric system $\mathcal{A}_k$
by some $R_{j_k}$ for $j_k \in \{i + 1, \dots, t\} \setminus \{j_l \mid 1 \leq l < k\}$ and finally we arrive
at a toric system $\mathcal{A}_{t - i}$. Of course, $\mathcal{A}_{t - i}$ also depends on the positions at
which the $\mathcal{A}_k$ have been augmented. A toric system obtained this way in general cannot be interpreted
as successive augmentation via pullbacks from the $X_j$ with $i < j < t$ as we have not imposed any condition
on the ordering of the $j_k$. We will see below that the interesting augmentations which are obtained this way
are precisely those which are augmentations via pullbacks.

\begin{definition}\label{standardaugmentationdefinition}
We call an exceptional toric system on $\mathbb{P}^2$ or $\mathbb{F}_a$ a {\em standard toric
system}. On a smooth complete rational surface $X$, we call a toric system which is the augmentation
of a standard toric system a {\em standard augmentation}. A standard augmentation is
{\em admissible} if it contains no element of the form $R_i - \sum_{j \in S} R_j$ such that
$R_j \leq R_i$ for some $j \in S$.
\end{definition}

Note that the condition of admissibility is precisely the condition of Proposition \ref{verticalleftorthogonals}
on left-orthogonality of divisors of the form $R_i - \sum_{j \in S} R_j$. This condition implies that a standard
augmentation is admissible iff there exists a bijection $j : [t] \rightarrow [t]$, $k \mapsto j_k$ such that
$R_k \geq R_l$ iff $R_{j_k} \geq  R_{j_l}$. Then we can rearrange the ordering of the blow-ups accordingly
such that $X = X_{j_t} \rightarrow \cdots \rightarrow X_{j_1} \rightarrow X_0$ and the augmentation then can be
considered as an successive augmentation along these blow-ups. The following proposition shows that this way we
get many exceptional sequences in the form of standard augmentations.

\begin{proposition}\label{exceptionalsequencesexist}
Every standard augmentation yields a full exceptional sequence on $X$ iff it is admissible.
\end{proposition}

\begin{proof}
Let $\mathcal{A} = A_1, \dots, A_n$ be the augmented sequence. If $X_0 = \mathbb{P}^2$, we can renumber this
sequence such that $A_n$ is of the form $H - \sum_{i \in S} R_i$ for some subset $S$ of $[t]$. We claim that
$\mathcal{A} = A_1, \dots, A_{n - 1}$ yield an exceptional sequence iff it is admissible. That is, every
$A_I := \sum_{i \in I} A_i$ for some non-cyclic interval $I \subset [n - 1]$ is left-orthogonal iff
$\mathcal{A}$ is admissible. Clearly, every such $A_I$ is numerically
left-orthogonal. We have two cases. First, $l H - \sum_{i \in T} R_i$ with $T \subset [t]$ and $l \in \{1, 2\}$.
By Serre duality we get $h^2(-lH + \sum_{i \in T} R_i) = h^0(-(3 - l)H + \sum_{i \notin T} R_i) = 0$ and thus
$l H - \sum_{i \in T} R_i$ is left-orthogonal (without any condition on admissibility).
Second, we have $A_I = R_i - \sum_{i \in T} R_i$ with $T \subset [t]$,
which is left-orthogonal by proposition \ref{verticalleftorthogonals} iff $R_j \nleq R_i$ for all $j \in T$.
In particular, all $A_I$ are of this form iff $\mathcal{A}$ is admissible. 

If $X_0 = \mathbb{F}_a$, we can renumber the sequence such that $A_n$ is of the form $Q - (a + n)P -
\sum_{i \in S} R_i$ for some subset $S$ of $[t]$. Then for $A_I$ we have three cases. First, 
$P - \sum_{i \in T} R_i$ with $T \subset [T]$. By Serre duality we get $h^2(-P + \sum_{i \in T} R_i)
= h^0(-2Q - (1 - a)P - \sum_{i \notin T} R_i) = 0$ and so $P - \sum_{i \in T} R_i$ is left-orthogonal.
Second, we have $Q + nP -  \sum_{i \in T} R_i$ with $T \subset
[T]$ and $n \in \Z$. Again, by Serre duality, we get $h^2(-Q - nP + \sum_{i \in T} R_i) = h^0(-Q - (2 - n
- a)P - \sum_{i \notin T} R_i) = 0$ and thus $Q + nP -  \sum_{i \in T} R_i$ is left-orthogonal.
Third, we have $A_I = R_i -
\sum_{i \in T} R_i$ with $T \subset [t]$, which is left-orthogonal by proposition
\ref{verticalleftorthogonals} iff $R_j \nleq R_i$ for all $j \in T$.
In particular, all $A_I$ are of this form iff $\mathcal{A}$ is admissible.

We have seen now that a standard augmentation is admissible iff all $A_I$ are left-orthogonal.
It follows directly from the results of \cite{Orlov93} that standard augmentations are full.
\end{proof}

So, by observing that we can lift any standard sequence on some $X_0$ to an admissible standard augmentation
on $X$, the following is an immediate consequence of Proposition \ref{exceptionalsequencesexist}:

\begin{theorem}\label{exceptionalexistence}
Every smooth complete rational surface has a full exceptional sequence of invertible sheaves.
\end{theorem}

Let us denote $b_i: X_i \longrightarrow X_{i - 1}$ the $i$-th blow-up in the sequence $X = X_t \rightarrow
\cdots \rightarrow X_0$. We assume that $b_i$ can be partitioned into two sets $S_1 := \{b_1, \dots, b_s\}$
and $S_2 := \{b_{s + 1}, \dots b_t\}$ for $1 < s \leq t$ such that the $b_i$ within $S_l$ for $l \in \{1, 2\}$
commute. In other words,
we assume that $X$ can be obtained from $\mathbb{P}^2$ or $\mathbb{F}_a$ by two times simultaneously
blowing up (possibly) several points.

\begin{theorem}\label{p2blowupexistence}
With above assumptions on $X$ and $X_0 = \mathbb{P}^2$, the following is a full strongly exceptional
toric system:
\begin{equation*}
R_s, R_{s - 1} - R_s, \dots, R_1 - R_2, H - R_1, H - R_{s + 1}, R_{s + 1} - R_{s + 2}, \dots,
R_{t - 1} - R_t, R_t, H - \sum_{i = 1}^t R_i.
\end{equation*}
\end{theorem}

\begin{proof}
We have to check that $\sum_{i \in I} A_i$ is strongly left-orthogonal for every interval $I \subset
[n - 1]$. Here we have $A_1 = R_s$ and $A_n = H - \sum_{i = 1}^t R_i$.  There are precisely four types
of divisors which can be represented in this way, namely $R_i$, $R_i - R_j$ for $R_i, R_j$ incomparable,
$H$, $2H$, $H - R_i$ and $2H - R_i - R_j$ for $i \neq j$. The divisors $H$, $2H$ are clearly strongly
left-orthogonal. The left-orthogonality of $R_i$ and $R_i - R_j$ follows
from proposition \ref{verticalleftorthogonals}, the left-orthogonality of $H - R_i$ and $2H - R_i - R_j$
from Lemma \ref{specialorthogonals}. The toric system clearly is an admissible standard augmentation and so
from Proposition \ref{exceptionalsequencesexist} it follows that the resulting exceptional sequence is full.
\end{proof}

Analogously, we get:

\begin{theorem}\label{fablowupexistence}
With above assumptions on $X$ and $X_0 = \mathbb{F}_a$ for some $a \geq 0$ and $n \geq -1$, the following
is a full strongly exceptional toric system:
\begin{gather*}
R_s, R_{s - 1} - R_s, \dots, R_1 - R_2, P - R_1, nP + Q, P - R_{s + 1}, R_{s + 1} - R_{s + 2}, \dots,\\
R_{t - 1} - R_t, R_t, -(a + n)P + Q - \sum_{i = 1}^t R_i.
\end{gather*}
\end{theorem}

\begin{proof}
Here, $\sum_{i \in I} A_i$ is of the form $R_i$, $R_i - R_j$ for $R_i, R_j$ incomparable,
$P$, $nP + Q$ with $n \geq -1$, $P - R_i$, $nP + Q - R_i$ for $n \geq 0$, and $nP + Q - R_i - R_j$ for
$n \geq 1$. The divisors $P$, $nP + Q$ clearly are strongly left-orthogonal (see Lemma
\ref{Hirzebruchleftorthogonals}). The left-orthogonality of $R_i$ and $R_i - R_j$ follows
from Proposition \ref{verticalleftorthogonals}, the left-orthogonality of $nP + Q - R_i$ and $nP + Q
- R_i - R_j$ from Lemma \ref{specialorthogonals}. The cases $P - R_i$ and $Q - R_i$ are clear because
$P$ and $Q$ are globally generated.
Also, the toric system is an admissible augmentation of a
standard sequence and so from proposition \ref{exceptionalsequencesexist} it follows that the resulting
exceptional sequence is full.
\end{proof}

The following theorem is an immediate consequence of Theorem \ref{fablowupexistence}.

\begin{theorem}\label{twotimesblowupexistence}
Any smooth complete rational surface which can be obtained by blowing up 
a Hirzebruch surface two times (in possibly several points in each step) has a full strongly exceptional
sequence of invertible sheaves.
\end{theorem}

\begin{remark}
Note that for the existence of strongly exceptional sequences it suffices to consider $X_0 =
\mathbb{F}_a$ for some $a \geq 0$, as every blow-up of $\mathbb{P}^2$ factorizes through a blow-up of
$\mathbb{F}_1$. Nevertheless, as we will see later on, for cyclic strongly exceptional sequences it will
be advantageous also to consider augmentations coming from $\mathbb{P}^2$.
\end{remark}

The converse of Theorem \ref{twotimesblowupexistence} is true for strongly exceptional sequences
coming from standard augmentations:

\begin{theorem}\label{augmentationboundtheorem}
Let $\mathbb{P}^2 \neq X$ be a smooth complete rational surface which admits a full strongly
exceptional standard augmentation then $X$ can be obtained by blowing up a Hirzebruch surface
two times (in possibly several points in each step).
\end{theorem}

We prove this theorem in section \ref{augmentationboundproofsection}.

\begin{remark}
We will see in Theorem \ref{toricclassification} that in the toric case every full strongly exceptional
sequence of invertible sheaves is equivalent to a strongly exceptional standard augmentation which
implies (Theorem \ref{toricbound}) that a toric surface different from $\mathbb{P}^2$ admits such
a sequence iff it can be obtained by blowing up a Hirzebruch surface at most two times. So, in a sense,
the existence of a full strongly exceptional sequence of invertible sheaves can be considered as a
geometric characterization of a surface. Presumably, Theorem \ref{toricclassification} should
generalize to all rational surfaces, but at present it is not clear to us whether the procedure
of sections \ref{toriccohomologyvanishing} to \ref{toricproofs} can be generalized in an effective way.
\end{remark}

The following theorem gives a strong constraint on the existence of cyclic strongly exceptional sequences
of invertible sheaves on rational surfaces in general:

\begin{theorem}\label{helixpicbound}
Let $X$ be a smooth complete rational surface on which a full cyclic strongly exceptional sequence of
invertible sheaves exists. Then $\rk \pic(X) \leq 7$.
\end{theorem}

\begin{proof}
Let $\mathcal{A} = A_1, \dots, A_n$ be the associated toric system. As every $A_i$ is strongly left-orthogonal,
it follows that $\chi(A_i) \geq 0$ for every $i$. Therefore by Proposition \ref{canonicalnefproposition} the
anti-canonical bundle of the associated toric surface $Y(\mathcal{A})$ must be nef. From the classification of
such toric surfaces (see table \ref{weakdelpezzofigure}) it follows that
$\rk \pic(X) = \rk \pic\big(Y(\mathcal{A})\big) \leq 7$.
\end{proof}

In particular, Theorem \ref{helixpicbound} implies that not even every del Pezzo surface has a cyclic
strongly exceptional sequence of invertible sheaves. However, if $\rk \pic(X) \leq 7$, we have the
following positive result:

\begin{theorem}\label{delPezzocyclicexistence}
Let $X$ be a del Pezzo surface with $\rk \pic(X) \leq 7$, then there exists a full cyclic
strongly exceptional sequence of invertible sheaves on $X$.
\end{theorem}

\begin{proof}
Recall that a del Pezzo surface is either $\mathbb{P}^1 \times \mathbb{P}^1$ or a blow-up of
$\mathbb{P}^2$ in at most 8 points (see \cite{Demazure80}). The case $\mathbb{P}^1 \times
\mathbb{P}^1$ is clear from Proposition \ref{Hirzebruchtoricsystems}. For the other cases, by
our assumptions it suffices to assume that $X$ is a blow-up of $\mathbb{P}^2$ in
at most 6 points $x_1, \dots, x_6$. Moreover, it suffices to only consider the maximal case,
i.e. $\rk \pic(X) = 7$ and  the cases of smaller rank will follow immediately. We first give an
example for a cyclic exceptional toric system and then show that it is cyclic strongly exceptional.
We fix a blow-down $X \rightarrow \mathbb{P}^2$ and denote $R_1, \dots, R_6$ the exceptional
divisors and $H$ the class of a line on $\mathbb{P}^2$. Then by Proposition
\ref{exceptionalsequencesexist} the following is a full cyclic exceptional sequence:
\begin{gather*}
H - R_1 - R_2 - R_5, R_2, R_1 - R_2, H - R_1 - R_3 - R_4,\\
R_4, R_3 - R_4, H - R_3 - R_5 - R_6, R_6, R_5 - R_6.
\end{gather*}
To show that a toric system $A_1, \dots, A_6$ is cyclic strongly exceptional, we have to show that
for every
cyclic interval $I \subset [6]$ the sum $A_I := \sum_{i \in I} A_i$ is strongly left-orthogonal. There
are several possible cases what $A_I$ can be. First, if $A_I = R_i$ for some $i \in [6]$ or $A_I
= R_i - R_j$ for some $i \neq j \in [6]$, strong left-orthogonality follows from Proposition
\ref{verticalleftorthogonals}. The next cases are of the form $H - R_i$, $H - R_i - R_j$ and
$H - R_i - R_j - R_k$, respectively, where $i, j, k$ pairwise distinct. Analogous to the arguments
in the proof of \ref{specialorthogonals}, we have to discuss the existence of base points.
As $H$ is very ample, its associated complete linear system does not have base points.
So $h^0(H - R_i) < h^0(H)$, and we
conclude as in the proof of \ref{specialorthogonals} that $H - R_i$ is strongly left-orthogonal.
For any two $x_i, x_j$, we can find  a line on $\mathbb{P}^2$ which does pass through $x_i$
but not through $x_j$. So, the linear system $\vert H - R_i \vert$ is base point free and 
$H - R_i - R_j$ is strongly left-orthogonal for any $i \neq j$. The divisor $H - R_i - R_j$ is not
base point free. Its base points lie on the line connecting $x_i$ and $x_j$. But as $X$ is del
Pezzo, none of the other $x_k$ lie on this line. So we have $h^0(H - R_i - R_j - R_k) <
h^0(H - R_i - R_j)$ and thus $H - R_i - R_j - R_k$ is strongly left-orthogonal.
Similarly, using \cite{Hart2}, V.4, Corollary 4.2, we see that $2H - \sum_{i \in S} R_i$ is
strongly left-orthogonal for any $S \subsetneq [6]$.
The next cases are of the form $3H - \sum_{i \in S} R_i$, where $S \subseteq [6]$ and
$\vert S \vert \geq 4$. As $\vert S \vert < 7$, it follows from \cite{Hart2}, V.4, Proposition
4.3, that these are strongly left-orthogonal, too.
The remaining cases are of the form $3H - 2 R_i - \sum_{k \neq i, j} R_k$ with $i \neq j
\in [6]$. By \cite{Hart2}, V.4, Proposition 4.3, $3H - \sum_{k \neq j} R_k$ has no base points,
therefore $h^0(3H - 2 R_i - \sum_{k \neq i, j} R_k) < h^0(3H - \sum_{k \neq j} R_k)$ and
$3H - 2 R_i - \sum_{k \neq i, j} R_k$ is strongly left-orthogonal
\end{proof}

\begin{remark}
Note that for a del Pezzo surface $X$ with $\rk \pic(X) \leq 7$ the toric system of the type as given
in the proof of Theorem \ref{delPezzocyclicexistence} in general is not the only possibility.
It is an exercise to write down all possible admissible standard augmentations for $X_0 = \mathbb{P}^2$
and to check the conditions whether the resulting toric system is cyclic and strong. For example,
for $X$ del Pezzo, the strongly exceptional toric systems as given in Theorem \ref{p2blowupexistence}
are cyclic iff $t \leq 3$. Moreover, it follows from the proof of Theorem \ref{delPezzocyclicexistence}
that the conditions on $X$
can be weakened in general. Though the toric system given in the proof
does require that no three points are collinear, it admits a configuration of 6 points lying on
a conic and certain configurations of infinitely near points. We will see in Theorems
\ref{toriccyclicnef} and \ref{toriccyclicexistence}  that at least in the toric case the
existence of such sequences is equivalent to $-K_X$ nef.
\end{remark}

We conclude this section with some more technical properties of strongly exceptional sequences.
As before, we assume that a sequence of blow-downs to a minimal surface $X_0$ is chosen.
First we consider parts of a toric system which are ``vertical'' with respect to $X_0$:

\begin{lemma}\label{glueinglemma}
Let $A_1, \dots, A_k \in \pic(X)$ such that $A_i . A_{i + 1} = 1$ for $1 \leq i < k$ and $A_i . A_j = 0$
else such that $A_I := \sum_{i \in I} A_i$ is strongly left-orthogonal and $(A_I)_0 = 0$ for every interval
$I \subset [k]$. Then this system is, up to reversing the order of the $A_i$, of one of the following shapes:
\begin{enumerate}[(i)]
\item $A_1 = R_{i_1} - R_{i_2}, A_2 = R_{i_2} - R_{i_3}, \dots, A_k = R_{i_k} - R_{i_{k + 1}}$,
\label{shortorthogonal}
\item $A_1 = R_{i_1} - R_{i_2}, A_2 = R_{i_2} - R_{i_3}, \dots, A_{k - 1} = R_{i_{k - 1}} -
R_{i_k}, A_k = R_{i_k}$, \label{longorthogonal}
\end{enumerate}
where the $R_{i_l}$ are pairwise incomparable.
\end{lemma}

\begin{proof}
By proposition \ref{verticalleftorthogonals} every $A_I$ must be of the form $R_i$ or $R_i - R_j$
for some $i, j \in [t]$ such that $R_i$ and $R_j$ are incomparable. Moreover, $(R_{i_p} -  R_{i_q})
. (R_{i_s} -  R_{i_t}) = 1$ iff either $q = s$ and
$p \neq t$ or $q \neq s$ and $p = t$. Moreover, $(R_{i_p} -  R_{i_q}) . (R_{i_s} -  R_{i_t}) = 0$ iff
$\{p, q\} \cap \{s, t\} = \emptyset$. This readily implies that the sequence $A_1, \dots, A_k$ must be
of one of the above forms.
\end{proof}

For the parts of a toric system which are not vertical to $\pic(X_0)$, we would like to have a normal form.
Let $\sh{O}(E_1)$, $\dots$, $\sh{O}(E_n)$ be a strongly exceptional sequence and $A_1, \dots, A_n$ its
associated toric system. One of the requirements is that
$\Hom\big(\sh{O}(E_i), \sh{O}(E_j)\big) = H^0\big(X, \sh{O}(-\sum_{k = j}^{i - 1} A_i)\big) = 0$ for $i > j$.
So, clearly, for any $1 \leq i < n$ with $\chi(A_i) = 0$, we can exchange $E_i$ and $E_{i + 1}$ such that
$\sh{O}(E_1)$, $\dots$, $\sh{O}(E_{i + 1})$, $\sh{O}(E_i)$, $\dots$, $\sh{O}(E_n)$ also
forms a strongly exceptional sequence. The toric system then becomes:
\begin{equation*}
A_1, \dots, A_{i - 2}, A_{i - 1} + A_i, -A_i, A_{i + 1} + A_i, A_{i + 2}, \dots, A_n.
\end{equation*}
We introduce the following notion with this operation in mind.

\begin{definition}
Let $\mathcal{A} = A_1, \dots, A_n$ be a toric system. If $\mathcal{A}$ gives rise to a (cyclic) strongly
exceptional sequence, we say that $\mathcal{A}$ is in {\em normal form} with respect to $X_0$ if
$(A_i)_0$ is either zero or strongly pre-left-orthogonal for every $1 \leq i < n$ (for every $1 \leq i
\leq n$, respectively).
\end{definition}

Assume that $\mathcal{A}$ gives rise to a strongly exceptional sequence and is not in normal form, i.e.
there exists some $A_i$ with $1 \leq i < n$ such that $(A_i)_0$ is non-trivial and not strongly
pre-left-orthogonal. This implies that $\chi(A_i) = h^0(A_i) = 0$ and there are no homomorphisms between
$\sh{O}(D_{i - 1})$ and $\sh{O}(D_i)$. In fact, there exists a maximal interval $I \subset [n - 1]$
containing $i$ such that $\Hom\big(\sh{O}(D_k), \sh{O}(D_l)\big) = 0$ for every $k, l \in I$. Clearly,
any reordering of the $D_k$ with $k \in I$ is a strongly exceptional sequence, too. We are going to
show that every strongly exceptional sequence comes, up to such reordering, from a toric system in normal
form.

\begin{proposition}\label{normalorderingproposition}
Let $X$ be a smooth complete rational surface and $X_0$ a minimal model for $X$. Then any (cyclic) strongly
exceptional sequence of invertible sheaves on $X$ can be reordered such that the associated toric system is
in normal form with respect to $X_0$.
\end{proposition}

\begin{proof}
Let $\sh{O}(E_1)$, $\dots$, $\sh{O}(E_n)$ be a strongly exceptional sequence and $\mathcal{A} = A_1, \dots,
A_n$ its associated toric system. As remarked above, for any interval $[k, \dots, l + 1] \subset [n]$ such
that $\chi(A_i) = 0$ for every $k \leq i \leq l$, we can exchange the positions of any two $\sh{O}(E_i)$,
$\sh{O}(E_j)$ with $i, j \in I$. In particular, if we want to move $\sh{O}(E_{l + 1})$ to the leftmost
position, it is easy to see that the toric system becomes
\begin{equation*}
\dots, A_{k - 1}, \sum_{i = k}^l A_i, -\sum_{i = k + 1}^l A_i, A_{k + 1}, \dots, A_{l - 1}, A_l + A_{l + 1},
A_{l + 2}, \dots
\end{equation*}
Let $1 \leq l < n$ be minimal such that $(A_l)_0$ is non-trivial and not strongly pre-left-orthogonal. Then
exchanging $\sh{O}(E_{l + 1})$ with $\sh{O}(E_l)$ changes the toric system to $\dots, A_{i - 2},
A_{i - 1} + A_i, -A_i, A_{i + 1} + A_i, A_{i + 2}, \dots$ such that $(-A_l)_0$ is strongly
pre-left-orthogonal. Now possibly $(A_{i - 1} + A_i)_0$ is no longer strongly pre-left-orthogonal. In
this case we iterate moving $\sh{O}(E_{l + 1})$ to the left. This process eventually stops, because of one
of two reasons. First, $\sh{O}(E_{l + 1})$ ends up at the most left position and we are getting
$-\sum_{i = 1}^l A_i, A_1, \dots, A_{l - 1}, A_l + A_{l + 1}, A_{l + 2}, \dots, A_n + \sum_{i = 1}^l A_i$.
Second, $\sh{O}(E_{l + 1})$ is at $(k + 1)$-th position, but $(\sum_{i = k}^l A_i)_0$ is strongly
pre-left-orthogonal. 
Consequently, after moving $\sh{O}(E_{l + 1})$, the smallest $1 \leq l' < n$ such that $(A_{l'})_0$ is
non-trivial and not strongly pre-left-orthogonal is strictly bigger than $l$. So, by iterating this exchange
process, we end up with a toric system in normal form.

If $\sh{O}(E_1)$, $\dots$, $\sh{O}(E_n)$ is a cyclic strongly exceptional sequence, we are free to
cyclically change the enumeration of the $A_i$. In particular, from the general classification of toric
surfaces,
it follows that there cannot be a cyclic interval $I \subset \on$ of length bigger than $n - 3$ such that
$h^0(A_i) = 0$ for every $i \in I$. Moreover, if $\mathcal{A}$ is not in normal form, we can choose the
enumeration of the $A_i$ the way that, if $(A_l)_0$ is non-trivial and not strongly pre-left-orthogonal,
then $h^0(A_1) > 0$ and we choose $l < k < n$ minimal such that  $h^0(A_k) > 0$. This implies that
$(\sum_{i = 1}^p A_i)_0$ and $(\sum_{i = q}^k A_i)_0$ are strongly pre-left-orthogonal for every
$1 \leq p < k$ and every $1 < q \leq k$. So the part $A_1, \dots, A_k$ is in normal form. We iterate
this and eventually all of $\mathcal{A}$ will be in normal form.
\end{proof}

\section{Proof of Theorem \ref{augmentationboundtheorem}}\label{augmentationboundproofsection}

Assume first that $X_0 = \mathbb{P}^2$ and $\mathcal{A}_0 = H, H, H$. If we blow
up $X_1 \rightarrow X_0$, then there is, up to cyclic change of enumeration, only one possible
augmentation $\mathcal{A}_1 = H - R_1, R_1, H - R_1, H$. But $X_1$ is isomorphic to $\mathbb{F}_1$ and if
we choose the usual generators $P$, $Q$ of the nef cone of $X_1$ as a basis of $\pic(X_1)$, we get
$P = H - R_1$, $Q = H$. In these coordinates we have $\mathcal{A}_1 = P, Q - P, P, Q$, which is the
unique cyclic strongly exceptional standard toric system on $\mathbb{F}_1$. So, to prove the theorem
it suffices to consider standard toric
systems coming from Hirzebruch surfaces according to the classification of Proposition
\ref{Hirzebruchtoricsystems}.
We assume that $X$ is obtained by a sequence of blow-ups $X = X_t \rightarrow \cdots \rightarrow X_0$
of a Hirzebruch surface $X_0 \cong \mathbb{F}_a$ and denote $P, Q, R_1, \dots, R_t$ the corresponding
basis of $\pic(X)$.

For any divisor $D$ we denote $\bs(D)$ the base locus of the complete linear system $\vert D \vert$.
Note that for any effective divisor $D$ the condition $\chi(-D) = 0$ is equivalent to the arithmetic
genus of $D$ being zero. It is straightforward to check that in this case the underlying reduced divisor
$D_{\rm red}$ also has arithmetic genus zero. So, because $h^0(R_i) = 1$ for every $i$, the divisor
class $R_i$ is represented by a unique, possibly non-reduced, effective divisor of arithmetic genus zero
and $\bs(R_i)$ coincides with the support of this divisor whose arithmetic genus is also zero. The
image of $\bs(R_i)$ in $X_0$ is contained in some fiber of the ruling $\mathbb{F}_a
\rightarrow \mathbb{P}^1$ which we denote by $f_i$ and which represents the divisor class $P$.

For any $R_i$ we denote $E_i$ the strict transform on $X$ of the corresponding exceptional divisor
on $X_i$. By abuse of notation we also use $E_i$ for the strict transforms on the $X_j$ with $j \geq i$.
Any effective divisor $D$ whose support contains $E_i$ can be written $D = D' + n_i E_i$ where $D'$
is effective and does not have any component with support $E_i$. We call $n_i$ the {\em multiplicity}
of $E_i$ with respect to $D$. By abuse of notion we will also sometimes call $n_i$ the multiplicity
of $R_i$.

We recall that the $R_i$ form a partially ordered set. The maximal elements have the property that
$E_i^2 = -1$. Any maximal chain of $R_i$ contains precisely one maximal element. All maximal elements
are incomparable and can be blown down simultaneously. In the nicest cases we will see that
the maximal length of maximal chains will be at most two and that $X$ can be blown down to $X_0$
in two steps. However, the most part of our analysis in this section will be concerned with the
cases where there exist maximal chains of bigger length. In general there will be only very few
of these and, if such chains exist, we will have to look for some other way to blow down to some
minimal model $X_0'$ which might not coincide with $X_0$. For this we will need exceptional divisors
which do not coincide with one of the $E_i$. These exceptional divisors can be the strict transform
of some fiber $f_i$ or, in the case $X_0 \cong \mathbb{F}_1$, of the unique divisor on $X_0$ with
self-intersection $-1$.
Note in the sequel we will consider the case where blow-ups are only over a fixed fiber $f$.
This will be without loss of
generality, because in our conclusion at the end of this section we will make use of the fact
that $f_i \neq f_j$ implies that $R_i$ and $R_j$ are incomparable.

\begin{lemma}\label{schlunzlemma}
We use notation as before.
\begin{enumerate}[(i)]
\item\label{schlunzlemmai} For any $i$, the divisor class $P - R_i$ is strongly left-orthogonal and
$\bs(P - R_i)$ contains $\bs(R_j)$ for every $R_j$ with $f_j = f_i$ and $R_i \nleq R_j$.
\item\label{schlunzlemmaiii} If the multiplicity of $R_i$ with respect to the total transform
of $f_i$ is greater than $1$, then $\bs(P - R_i)$ contains $\bs(R_j)$ for every $R_j$ with $f_j = f_i$.
\item\label{schlunzlemmaii} For any $i \neq j$, the divisor class $P - R_i - R_j$ is strongly left-orthogonal
iff either $f_i \neq f_j$ or $R_i$ and $R_j$ are comparable (say, $R_i \leq R_j$) and
$\bs(P - R_i)$ does not contain $\bs(R_j)$.
\end{enumerate}
\end{lemma}

\begin{proof}
Clearly we have $\chi(-P + R_i) = 0$, $h^k(-P + R_i) = 0$ for all $k$, and $\chi(P - R_i) = 1$. To
show that $P - R_i$ is strongly left-orthogonal we need only to show that $h^0(P - R_i) = 1$. But
this follows from the fact that the divisor class $P$ is nef and therefore base-point free and hence
$h^0(P - R_i) = h^0(P) - 1 = 1$. The divisor class $P - R_i$ is nontrivial and its base locus is a
curve of arithmetic genus zero which projects to $f_i$. The total transform of $f_i$ is a representative
of $P$ in $\pic(X)$ and contains the base loci of all the $R_j$ with $f_j = f_i$. By subtracting $R_i$
from $P$, we at most (but not necessarily) cancel the base loci of those $R_j$ with $R_i \leq R_j$ and
(\ref{schlunzlemmai}) follows. If the multiplicity of $R_i$ with respect to the total transform of
$f_i$ is greater than $1$, then the multiplicities of all $E_j$ with $R_i \leq R_j$ with respect to
$P$ is strictly smaller than their multiplicities with respect to $R_i$. Therefore $\bs(P - R_i)$
also contains $\bs(R_j)$ for $R_i \leq R_j$ and (\ref{schlunzlemmaiii}) follows.
From (\ref{schlunzlemmai}) it follows that statement (\ref{schlunzlemmaii})
essentially is a case distinction for determining when $\bs(R_j)$ is not contained in the base
locus of $P - R_i$.
\end{proof}

\begin{lemma}\label{gewenzlemma}
Consider the divisor $Q$ on $X_0 \cong \mathbb{F}_1$ and some strongly left-orthogonal divisor class
$Q - R_i - R_j$ on $X$ with $R_i$, $R_j$ incomparable and $f := f_i = f_j$. Then $\bs(Q - R_i - R_j)$
contains the total transform of $f$.
\end{lemma}

\begin{proof}
The class $Q$ is the pullback of the class of lines in $\mathbb{P}^2$.
Denote $p$ the image of $R_i$ in $X_0$, then we can identify the linear system
$\vert Q - R_i \vert$ with the set of lines passing through the image of $p$ in
$\mathbb{P}^2$. If $R_j$ lies over some other point of $f$ than $p$, then
$\bs(Q - R_i - R_j)$ fixes two points on $f$ and thus contains $f$.
If $R_j$ also lies over $P$, then we first observe that $\bs(Q - R_i)$ contains
$\bs(R_k)$ for all $k \neq i$ and $R_k \leq R_i$. So, the condition that $Q - R_i - R_j$
is strongly left-orthogonal implies that $R_i$ is minimal and hence $R_i \leq R_j$,
which is a contradiction.
\end{proof}

\begin{lemma}\label{fiberlemma}
Let $A_1, \dots, A_n$ a strongly exceptional toric system on $X$ which is a standard augmentation
of $P, sP + Q, P, -(a + s)P + Q$ with $s \geq -1$ for some choice of $X_0 \cong \mathbb{F}_a$
such that $f_i = f_j$ for all $i, j$.
Assume that $A_k, A_{k + 1}, \dots A_l$ for some $1 \leq k < l < n$ is a subsequence of the
toric system which contains the two slots around one of the $P$, i.e. $(A_{k - 1})_0, (A_{l + 1})_0
\notin \{0, P\}$, $(A_p)_0 = P$ for one $k \leq p \leq l$, and $(A_q)_0 = 0$ for all $k \leq q \leq l$
with $q \neq p$. Then $A_k, \dots, A_l$ is, up to possible order inversions, of one of these
forms:
\begin{enumerate}[(i)]
\item\label{fiberlemmai} $R_{i_{l - k}}, R_{i_{l - k - 1}} - R_{i_{k - l}}, \dots, R_{i_2} - R_{i_3},
R_{i_1} - R_{i_2}, P - R_{i_1}$, where the $R_{i_j}$ are pairwise incomparable;
\item\label{fiberlemmaii} $R_{i_1}, P - R_{i_1} - R_{i_2}, R_{i_2} - R_{i_3}, \dots, R_{i_{l - k - 1}}
- R_{i_{k - l}}, R_{i_{k - l}}$, where $R_{i_1} \leq R_{i_j}$ and the $R_{i_j}$ are pairwise incomparable
for $j > 1$.
\end{enumerate}
\end{lemma}

\begin{proof}
After the first augmentation we get $R_{i_1}, P - R_{i_1}$. In the second step, we can extend this
sequence in the middle, or to the left, or to the right. By extending in the middle, we get
$R_{i_1} - R_{i_2}, R_{i_2}, P - R_{i_1} - R_{i_2}$ which implies that $\bs(R_{i_2}) \notin
\bs(R_{i_1}) \cup \bs(P - R_{i_1})$, where the right hand side coincides with the total transform
of a fiber $f_i$ on $X$, which is not possible.
By extending to the left, we get $R_{i_2}, R_{i_1} - R_{i_2},
P - R_{i_1}$ with the necessary condition that $\bs(R_{i_1}) \cap \bs(R_{i_2}) = \emptyset$ and therefore
$R_{i_1}, \dots, R_{i_2}$ are incomparable. By iterating to the left, we obtain that the $R_{i_j}$ are
pairwise incomparable and therefore we arrive at the form (\ref{fiberlemmai}). If we extend to the
right instead, we get $R_{i_1}, P - R_{i_1} - R_{i_2}, R_{i_2}$ and by Lemma \ref{schlunzlemma}
(\ref{schlunzlemmaii}) $R_{i_1}$, $R_{i_2}$ must be comparable. In the next step, we extend without
loss of generality to the right and get $R_{i_1}, P - R_{i_1} - R_{i_2}, R_{i_2} - R_{i_3}, R_{i_3}$
where $R_{i_1}$, $R_{i_3}$ are comparable and $R_{i_2}$, $R_{i_3}$ are incomparable. If we extend
to the left in the next step, this implies that the pairs $R_{i_1}$, $R_{i_4}$ and $R_{i_2}$, $R_{i_3}$
are incomparable, but $R_{i_2}$ and $R_{i_3}$ are comparable to $R_{i_1}$ and $R_{i_4}$ respectively,
which is not possible. So, we can continue extending only to the right and we inductively obtain that
the $R_{i_j}$ are pairwise incomparable for $j > 1$ and $R_{i_1}$ is comparable with every $R_{i_j}$.
If $l - k > 2$, this implies that $R_{i_1} \leq R_{i_j}$ for every $j > 1$.
\end{proof}

Now we consider standard augmentations starting from a standard sequence $P, s P + Q, P, -(s + a)P + Q$
with $s \geq -1$ on $X_0$. For this, we have four ``slots'', in which we can insert the $R_i$ successively.
The augmented toric system is of the form $A_1, \dots, A_n$, where possibly $A_n$ is only left-orthogonal
but not strongly left-orthogonal. For $(A_n)_0$, there are four possibilities, namely $(A_n)_0 = 0$,
$(A_n)_0 = P$, $(A_n)_0 = sP + Q$ and $(A_n)_0 = -(s + a)P + Q$. We will first consider
the last case.

\begin{proposition}\label{standardaugmentationbound}
Let $\mathcal{A} = A_1, \dots, A_n$ be a strongly exceptional toric system which is a standard
augmentation of the toric system $P, sP + Q, P, -(s + a)P + Q$ with $s \geq -1$ on
$X_0 \cong \mathbb{F}_a$ such that $(A_n)_0 = -(s + a)P + Q$ and $f_i = f_j$ for all $i, j$.
Then $X$ can be obtained
from blowing up a Hirzebruch surface two times (in possibly several points in each step).
\end{proposition}

\begin{proof}
We denote $f$ the distinguished fiber such that $f = f_i$ for all $i \in \ot$.
Because $(A_n)_0 = -(s + a)P + Q$ the toric system has two subsequences which are of the form
as stated in Lemma \ref{fiberlemma}. This implies that there is a partition of the set
$\{R_1, \dots, R_t\}$ into two subsets $S_1 := \{R_{i_1}, \dots, R_{i_r}\}$, $S_2 :=
\{R_{j_1}, \dots, R_{j_s}\}$ such that, if nonempty, the elements in each of these subsets
either are (\ref{fiberlemmai}) incomparable or (\ref{fiberlemmaii}) $R_{i_1} \leq R_{i_k}$
and the $R_{i_k}$ incomparable for all $k > 1$ ($R_{j_1} \leq R_{j_k}$ and the $R_{j_k}$
incomparable for all $k > 1$, respectively). If both $S_1$ and $S_2$ are empty, there is
nothing to prove. If one of $S_1$, $S_2$ is empty, then the length of a maximal chain of
comparable elements among the $R_i$ is at most two and the proposition follows. So we assume
that $S_1$ and $S_2$ both are nonempty. If $S_1$ and $S_2$ both satisfy property (\ref{fiberlemmai}),
then again the length of a maximal chain of comparable elements among the $R_i$ is at most
two and the proposition follows. If both satisfy property (\ref{fiberlemmaii}), then we
have two cases. The first is that $R_{i_1}$, $R_{j_1}$ both are minimal. Then again the
length of a maximal chain of comparable elements among the $R_i$ is at most two. The
second case is that only one of these, say $R_{i_1}$, is minimal and $R_{i_1} = R_1$ without
loss of generality.
On $X_1$ we have $f^2 = -1$ and we can choose to either blow down $R_1$ or $f$. If we choose
$f$, then we obtain another of basis for $\pic(X_1)$ given by $P', Q', R_1'$, where $P' = P$,
$R_1' = P - R_1$ and $Q' = Q + \delta P - R_1$, where $\delta = 1$ if $R_1$ corresponds to
a blow-up of a point on the zero section of the fibration $\mathbb{F}_a \rightarrow \mathbb{P}^1$,
and $\delta = 0$ otherwise. If we complete this basis to a basis of $\pic(X)$ by using
the $R_i$ with $i > 1$, the sequence $R_{i_1}, P - R_{i_1} - R_{i_2}, R_{i_2} - R_{i_3},
\dots, R_{i_{r - 1}} - R_{i_r}, R_{i_r}$ becomes $P' - R_1', R_1' - R_{i_2},
R_{i_2} - R_{i_3}, \dots, R_{i_{r - 1}} - R_{i_r}, R_{i_r}$ with $R_1', R_{i_2}, \dots R_{i_r}$
pairwise incomparable. So we have reduced to the case that $S_1$ satisfies property
(\ref{fiberlemmai}) and $S_2$ satisfies property (\ref{fiberlemmaii}). Moreover, we
can assume that $R_{j_1}$ is not minimal as otherwise we can choose another basis
as we did before and reduce to the case that both $S_1$ and $S_2$ satisfy property
(\ref{fiberlemmai}).

In the remaining case, the length of a maximal chain of comparable $R_i$ is either
two or three. If it is two, the proposition follows. In the case where it is three,
we assume without loss of generality that $\mathcal{A}$ is an augmentation of a strongly
exceptional toric system on $X_3$ with $R_1 \leq R_2 \leq R_3$ such that $R_1 \in S_1$ and
$R_2, R_3 \in S_2$. Then the divisor $P - R_2 - R_3$ is strongly left-orthogonal and by
Lemma \ref{schlunzlemma} (\ref{schlunzlemmaiii}) it follows that the multiplicity of
$R_2$ with respect to the total transform of $f$ is one. In particular, $R_2$ does
not come from a blow-up of the intersection of $f$ with $E_1$ on $X_1$. If we now go
back to $X$, then the $R_{i_k}$ are incomparable with $R_1$ and hence with $R_2$,
because $R_1 \leq R_2$. Thus the $R_{i_k}$ are minimal. So, by blowing down simultaneously
all $E_i$ with $E_i^2 = -1$
(which includes $E_3$) on $X$, we arrive at the surface $X_2$. Here, we have
$f^2 = -1$ and $E_2^2 = -1$. So, we can blow-down these two divisors simultaneously
and arrive at some Hirzebruch surface $X_0'$.
\end{proof}

If $s + a > 1$, it follows by Lemma \ref{pullbackcohomologyvanishinglemma} that necessarily
$(A_n)_0 = -(s + a)P + Q$. If $s + a \leq 1$, then possibly $(A_n)_0 \in \{0, P\}$ and the
standard toric system $P, s P + Q, P, -(s + a)P + Q$ must be cyclic strongly exceptional on
$\mathbb{F}_a$ for which, by Proposition \ref{Hirzebruchtoricsystems}, there are only four
possibilities. Our first step will be to reduce these to one.

\begin{proposition}\label{helixaugmentationgauge}
Let $\mathcal{A} = A_1, \dots, A_n$ a toric system on $X$ with $(A_n)_0 \neq -(s + a)P + Q$. Then
there exists a sequence of blow-downs $X = X_t' \rightarrow \cdots X_0'$ such that $X_0' \cong
\mathbb{F}_1$ and $\mathcal{A}$ is an augmentation of the toric system $P', Q', P', Q' - P'$ on $X_0'$.
\end{proposition}

\begin{proof}
As argued before, $\mathcal{A}$ necessarily is an augmentation of a cyclic strongly exceptional
standard sequence. In particular, $X_0 \cong \mathbb{F}_a$ with $0 \leq a \leq 2$. If $a = 1$,
there is nothing to prove. If $a = 0$, there are, up to symmetry by exchanging $P$ and $Q$, two
such toric systems, $P, Q, P, Q$ and $P, P + Q, P, -P + Q$. If we consider the blow-up $X_1
\rightarrow X_0$, then in the first case, there exists, up to cyclic reordering and order
inversion, only one
possible augmentation which is given by $P - R_1, R_1, Q - R_1, P, Q$ which is a cyclic strongly
exceptional toric system on $X_1$. If we consider some
projection $X_0 \rightarrow \mathbb{P}^1$ such that $P$ represents a general fiber, then
the divisor $P - R_1$ is rationally equivalent to the strict transform under the blow-up and
has self-intersection $(-1)$. If we blow down this divisor, we obtain $X_1 \rightarrow X_0'
\cong \mathbb{F}_1$. If we denote $P', Q'$ the corresponding divisors in $\pic(\mathbb{F}_1)$,
then we get a change of coordinates in $\pic(X_1)$ via $P = P'$, $Q = Q' - R_1'$, and
$R_1 = P' - R_1'$. In this basis the toric system is given as $R_1', P' - R_1', Q' - P', P',
Q' - R_1'$ and the assertion follows in this case. We proceed similarly in the second case.
As $h^0(P - Q) = 0$ and $(A_n)_0 \neq P - Q$ by assumption, the only possible augmentation
(up to cyclic reordering and order inversion) on $X_1$ is given by $P - R_1, R_1, P + Q - R_1,
P, -P + Q$. By the
same change of coordinates as before we get $R_1', P' - R_1', Q', P', Q' - P' - R_1'$ and
the assertion follows for this case.

Now assume that $a = 2$. Then the only cyclic strongly exceptional toric system is given
by $P, Q - P, P, Q - P$ and the only possible augmentation on $X_1$ is given by
$P - R_1, R_1, Q - P - R_1, P, Q - P$. The base locus of the complete linear system of
the divisor $Q - P$ consists of one fixed component which is the zero section of the fibration
$\mathbb{F}_2 \rightarrow \mathbb{P}^1$. Therefore, if $X_1$
is a blow-up on the zero section, we have $h^0(Q - P - R_1) = h^0(Q - P) = 2$ and $Q - P - R_1$ is
not strongly left-orthogonal and thus necessarily $(A_n)_0 = Q - P$ which is a contradiction
to our assumptions. So we can assume without loss of generality that $X_1$ is a blow-up of $X_0$
at some point which is not on the zero section. In this case we can conclude as before that
there exists a blow-down to $X_0' \cong \mathbb{F}_1$ and a corresponding change of
coordinates $P = P' - R_1'$, $Q = Q' + P' - R_1'$, and $R_1 = P' - R_1'$ such that the
toric system is represented as $R_1', P' - R_1', Q' - P', P', Q' - R_1'$ which is of
the required form.
\end{proof}

\begin{proposition}\label{f1augmentationbound}
Let $\mathcal{A} = A_1, \dots, A_n$ be a strongly exceptional toric system which is a standard
augmentation of the toric system $P, Q, P, -P + Q$ on $X_0 \cong \mathbb{F}_1$. Then $X$ can be obtained
by blowing up a Hirzebruch surface two times (in possibly several points in each step).
\end{proposition}

\begin{proof}
We will only consider the case $(A_n)_0 \in \{0, P\}$. Otherwise, the result follows from
Proposition \ref{standardaugmentationbound}. We will denote $b$ the zero section (respectively
its strict transform) of the $\mathbb{P}^1$-fibration $X_0 \rightarrow \mathbb{P}^1$ with
$b^2 = -1$ on $X_0$. Note that in some steps below we will have to blow-down $b$ to arrive
at some convenient minimal model $X_0'$. Strictly speaking, this would require us not only
to consider blow-ups of a fixed fiber $f$ but rather the general case. However, in these
few cases this would only increase the number of case distinctions without changing the
arguments. So we will keep the assumption that all blow-ups lie above one distinguished
fiber $f$.

First note that $h^0(-P + Q) = 1$ and therefore any divisor of the form $-P + Q - R_i - R_j$
cannot be strongly left-orthogonal. This together with the condition $(A_n)_0 \neq -P + Q$
implies that we can use at most one of the two slots around $-P + Q$ in the toric system $P, Q, P,
-P + Q$ for augmentations. Moreover, for any $(A_n)_0$, we can assume that the augmentations
in the two slots surrounding one of the $P$'s are strongly left-orthogonal and therefore we
get there a subsequence of $\mathcal{A}$ which corresponds to one of the two shapes given in
Lemma \ref{fiberlemma}. The slot between $P$ and $Q$ can be augmented at most three times
because $h^0(Q) = 3$. Because of our general assumption that all blow-ups lie over the same
fiber, we can even conclude by Lemma \ref{gewenzlemma} that this slot even can be extended
at most two times. For the same reason, if this slot has been extended two times, then the
other slot neighbouring $Q$ cannot be extended any more. Denote $S_1$
the subset of the $R_i$ used for augmenting the two slots around $P$. We have seen that $S_1$
can consist of at most three elements. In the maximal case, we have $S_1 = \{R_{i_1}, R_{i_2},
R_{i_3}\}$ such that $R_{i_1} \leq R_{i_2}, R_{i_3}$ and $R_{i_2}, R_{i_3}$ incomparable.
In this case, the remaining two slots cannot be augmented without violating our condition
on $(A_n)_0$ and thus the assertion follows. So we assume from now that $S_1$ consists of
at most two elements, which may be comparable or not. Also note that the base locus of
$P - Q$ coincides with the support of the total transform of $b$ on $X$. Therefore,
in the cases where either $R_{i_1}$ and $R_{i_2}$ are comparable, or $S_1 = \{R_{i_1}\}$
and $R_{i_1}$ is used for augmentation in the slot between $-P + Q$ and $P$, these
divisors cannot come from blowing up points on or above $b$.

If $(A_n)_0 = P$, then the content of the two slots neighbouring this ``bad'' $P$ must be
of the form as given in Lemma \ref{glueinglemma} (\ref{longorthogonal}). That is, we
have a partition of the set of the $R_i$ into three sets, $S_1$, $S_2$, $S_3$, where the
latter two each consist of pairwise incomparable elements. If both $S_2$, $S_3$ are empty,
the assertion follows. If $S_1$ consists of two elements,
then only one of $S_2$, $S_3$ can be nonempty, say $S_2$. If the two elements in $S_1$ are
incomparable, we have thus a partition into two subsets of incomparable elements and the
assertion follows. If the two elements in $S_1$ are comparable, i.e. $R_{i_1} \leq R_{i_2}$,
then we have (up to order inversion) the subsequence $R_{i_1}, P - R_{i_1} - R_{i_2}, R_{i_2}$
in $\mathcal{A}$. By Lemma \ref{schlunzlemma} the divisor $R_{i_1}$ must have multiplicity
$1$ with respect to $P$ and $R_{i_2}$ cannot come from blowing up a point on the fiber $f$.
By this, after blowing down all $R_i$ with $E_i^2 = -1$ we are left with at most one chain of
length $2$, containing at least one $E_i$ with $E_i^2 = -1$ and we have $f^2 = -1$. So, by
simultaneously blowing down these two divisors we arrive at some minimal surface $X_0$
and the assertion follows. If $S_1$ consists of only one element, then we have two possibilities.
If $R_{i_1}$ is used for augmentation in the slot between $-P + Q$ and $P$, then the other
slot neighbouring $-P + Q$ is blocked for further augmentation and only one more slot is
free for augmentation by incomparable $R_i$. So we can blow down $X$ to $X_0$ in at
most two steps. If $R_{i_1}$ is used for augmentation in the slot between $P$ and $Q$,
then we can have two nonempty sets $S_2$, $S_3$. Let us assume that the elements
in $S_2$ are used for augmentation between $Q$ and $P$, and the elements in $S_3$
for augmentation between $-P + Q$ and $P$. As the base locus of $-P + Q$ contains the support
of the total
transform of $b$ on $X$, none of the $R_i \in S_3$ are lying over any point
of $b$. So, if $R_{i_0}$ lies over some point in $b$, then it can be part of a chain
of comparable $R_i$ of length two. Hence, the maximal such length is at most two
for all $R_i$. Hence the assertion follows. If $R_{i_0}$ does not lie over some point
of $b$, then the maximal length of a chain of comparable $R_i$ which lie over some
point of $b$ is one, and after simultaneously blowing down the $E_i$ with $E_i^2 = -1$,
there is no such $R_i$ left. But then $R_{i_0}$ might still be part of a chain of length
$2$, which will be the only such chain and another simultaneous blow-down will leave
one of the components of this chain. However, now we can additionally blow down
$b$ instead and we will arrive at some other $X_0'$ within two steps and the assertion
follows.

If $(A_n)_0 = 0$, then $A_n$ is located in one of the slots and the subsequence of
$\mathcal{A}$ in this slot can be of one of the following forms:
\begin{equation}\label{emptyaugmentation1}
Q - P - R_{j_1}, R_{j_1} - \sum_{k = 2}^r R_{j_k}, R_{j_r}, R_{j_{r - 1}} - R_{j_r}, \dots,
R_{j_2} - R_{j_3}, P - R_{j_1} - R_{j_2} - F,
\end{equation}
\begin{multline}\label{emptyaugmentation2}
Q - R_{j_1} - R_{j_2} - G, R_{j_2} - R_{j_3}, \dots, R_{j_{r - 1}} - R_{j_r}, R_{j_r},
R_{j_1} - \sum_{i = 2}^r R_{j_i} - \sum_{i = 1}^s R_{k_i}, \\ R_{k_s}, R_{k_{s - 1}} - R_{k_s},
\dots, R_{k_1} - R_{k_2}, P - R_{j_1} - R_{k_1} - H,
\end{multline}
where $F, G, H$ denote some possible additional summands coming from augmentations in the
neighbouring slots. We denote $S_2 := \{R_{j_2}, \dots, R_{j_r}\}$ and $S_2 := \{R_{k_1},
\dots, R_{k_s}\}$. In the case \ref{emptyaugmentation1}, we have $R_{j_1} \leq R_{j_i}$
and the $R_{j_i}$ incomparable for all $i > 1$. In the case \ref{emptyaugmentation1}, we
have $R_{k_1} \leq R_{j_i}$ and the $R_{j_i}$ incomparable for all $i$.
If both $S_2$ and $S_3$ are empty, then $\mathcal{A}$ is an augmentation
by the elements of $S_1$ and by $R_{j_1}$ and one possible augmentation by some $R_i$ in the
remaining slot. Then $R_i$ and $R_{j_1}$ must be comparable. These can form a chain of length
at most three which cannot lie over $b$. Therefore we can conclude as before that we can
blow-down the surface $X$ to a surface $X_0'$ in at most two steps.
If $S_2$ consists of two incomparable elements, then the other neighbouring slot of $Q$
is blocked for augmentations and the remaining augmentation must be of the form
(\ref{emptyaugmentation1}) with $R_{j_1}$ (and thus all the $R_{j_i}$) not lying over
$b$. So, if there exists a
chain of length three, this chain cannot lie over $b$ and again we can blow-down in
two steps to some $X_0'$.
If $S_1$ consists of two comparable elements then the remaining augmentation
must be of the form (\ref{emptyaugmentation1}) where $S_2 = \emptyset$, as $G \neq 0$.
Then we possibly have a maximal chain of length four, where at least one of the
elements in $S_1$ and one of $R_{j_1}$ and $R_{k_i}$ involved have multiplicity one,
and all the $R_{k_i}$ incomparable. With similar arguments as before, we can always
blowing down $X$ to some $X_0'$ in two steps by possibly contracting the fiber $f$.

In the remaining cases we have to consider $S_1$ consisting of one or two elements.
The arguments are completely analogous to the previous arguments and we leave these
to the reader.
\end{proof}

We conclude that Theorem \ref{augmentationboundtheorem} follows from Propositions
\ref{helixaugmentationgauge} and \ref{f1augmentationbound} in the case $(A_n)_0 \neq
-(a + s)P + Q$. For the case $(A_n)_0 = -(a + s)P + Q$ we note that if $f_i \neq f_j$
then $R_i$ and $R_j$ are incomparable. Moreover, from the proof of Proposition
\ref{standardaugmentationbound} we see that in order to blow-down to some $X_0'$
we may have to blow-down the strict transform of some fiber. But any such choices
can be made simultaneously. This proves Theorem \ref{augmentationboundtheorem}.

\section{Divisorial cohomology vanishing on toric surfaces}\label{toriccohomologyvanishing}\label{toricleftorthogonals}

Let $X$ be a smooth complete toric surface whose associated fan is generated by lattice vectors $l_1,
\dots, l_n$ and recall that $\pic(X)$ is generated by the $T$-invariant divisors $D_1, \dots,
D_n$. Recall from section \ref{surfacegeometry} that, besides the coordinates associated to a minimal
model $X_0$, we have two further coordinatizations for $\pic(X)$. The first is given by choosing for a
given divisor $D$ a $T$-invariant representative $D \sim \sum_{i = 0}^n c_i D_i$ such that we can identify
this representative with a tuple $(c_1, \dots, c_n)$ in $\Z^n$. The second coordinatization is given
by tuples $(d_1, \dots, d_n) \in \Z$ such that $\sum_{i \in \on} d_i l_i = 0$. The $d_i$ are uniquely
determined by the $c_i$ by the relations $d_i = c_{i - 1} + a_i c_i + c_{i + 1}$ for $i \in \on$. The
$c_i$ are determined by the $d_i$ up to a character $m \in M$, that is, $\sum_{i = 0}^n c_i D_i \sim
\sum_{i = 0}^n c'_i D_i$ iff there exists some $m \in M$ such that $c'_i = c_i + l_i(m)$ for all
$i \in \on$. In what follows, we will use all of these coordinatizations for the classification of
strongly left-orthogonal divisors on $X$.

Now assume that for a given divisor $D$, a $T$-invariant representative $D \sim \sum_{i = 0}^n c_i D_i$
is chosen. Then we can associate to $D$ a hyperplane arrangement $\{H_i\}_{i \in \on}$ in $M_\Q$ which
is given by hyperplanes
\begin{equation*}
H_i := \{m \in M_\Q \mid l_i(m) = -c_i\}.
\end{equation*}
The twist $c_i \mapsto c_i + l_i(m)$ for some $m \in M$ then corresponds to a translation of the hyperplane
arrangement by the lattice vector $-m$. The action of $T$ induces an eigenspace decomposition of the space
of global sections of
$\sh{O}(D)$:
\begin{equation*}
H^0\big(X, \sh{O}(D)\big) \cong \bigoplus_{m \in M} H^0\big(X, \sh{O}(D)\big)_m.
\end{equation*}
The nontrivial isotypical components $H^0\big(X, \sh{O}(D)\big)_m$ are one-dimensional and we have
\begin{equation*}
H^0\big(X, \sh{O}(D)\big)_m \neq 0 \quad \text{ iff } \quad l_i(m) \geq -c_i \ \text{ for all } i \in \on
\end{equation*}
for $m \in M$, i.e. the non-vanishing isotypical components correspond to the characters which are
contained in a distinguished chamber of the hyperplane arrangement.

\begin{definition}
Let $D = \sum_{i = 1}^n c_i D_i$ be a torus invariant divisor, then we denote $G_D := \{m \in M \mid
H^0(X, \sh{O}(D))_m \neq 0\} = \{m \in G_D \mid l_i(m) \geq -c_i \text{ for all } i \in \on\}$ and
$G_D^\circ := \{m \in G_D \mid l_i(m) > -c_i \text{ for all } i \in \on\}$.
\end{definition}

As the set $G_D$ counts the global sections of a $T$-invariant divisor $D$, by Serre duality, the set
$G_D^\circ$ can naturally be associated with a $T$-eigenbasis of $H^2\big(X, \sh{O}(-D)\big)$. Namely,
the canonical divisor on $X$ is given by $K_X = -\sum_{i = 1}^n D_i$ and $h^2(-D) = h^0(K_X + D) =
\vert G_D^\circ \vert$. We want to interpret strong (pre-)left-orthogonality as a problem of counting
lattice points,
starting from $G_D$ for some strongly pre-left-orthogonal divisor $D$ on $\mathbb{P}^2$ or
$\mathbb{F}_a$ as classified in propositions \ref{P2leftorthogonals} and \ref{Hirzebruchleftorthogonals}.
In general, the region containing $G_D$ is not quite a lattice polytope, but rather close to being one,
as we will see in Proposition \ref{degreelowerbound}. This is illustrated in the following example.

\begin{example}
Figure \ref{p2f2preleftorthogonalexamples} shows examples for strongly pre-left-orthogonal divisors on
$\mathbb{P}^2$ and $\mathbb{F}_a$. The dots indicate the set $G_D$, the white dots the subset $G_D^\circ$.
\begin{figure}[htbp]
\begin{center}
\includegraphics[height=10cm]{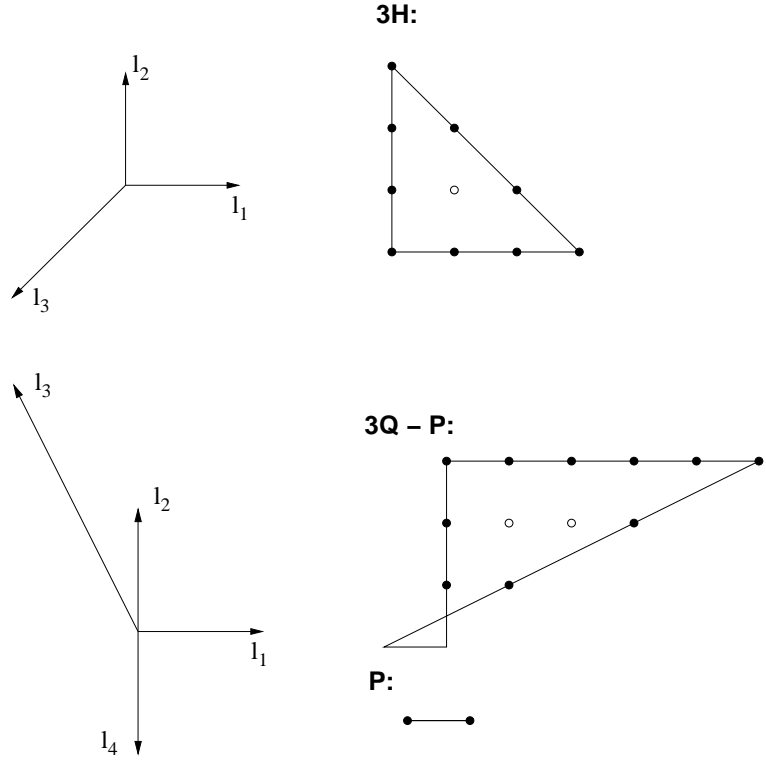}
\end{center}
\caption{The fans for $\mathbb{P}^2$ and $\mathbb{F}_2$ and the regions in $M$ containing $G_D$,
for the cases $D = 3H$ on $\mathbb{P}^2$ and $D = P$, $D = 3Q - P$ on $\mathbb{F}_2$, respectively.}
\label{p2f2preleftorthogonalexamples}
\end{figure}
\end{example}

Consider any pre-left-orthogonal divisor $\beta H$, where $\beta > 0$, on $\mathbb{P}^2$. Then it is easy
to see that formulas (\ref{eulercharp2}) and (\ref{antieulercharp2}),
\begin{equation*}
\chi(\beta H) = \binom{\beta + 2}{2}, \qquad \chi(-\beta H) = \binom{\beta - 1}{2},
\end{equation*}
count $G_{\beta H}$ and $G_{\beta H}^0$, respectively. Similarly, formulas (\ref{eulercharfa}) and
(\ref{antieulercharfa}),
\begin{equation*}
\chi(\alpha P + \beta Q) = (\alpha + 1) (\beta + 1) + a \binom{\beta + 1}{2}, \qquad
\chi(-\alpha P - \beta Q) = (\alpha - 1) (\beta - 1) + a \binom{\beta}{2},
\end{equation*}
count $G_{\alpha P + \beta Q}$ and $G_{\alpha P + \beta Q}^\circ$, respectively.

For the $\gamma_i$, there is a similar interpretation. Assume we have fixed a sequence of blow-ups
$b_1, \dots, b_t$ as in the previous section, where every $b_k$ is toric. For some $k \in [t]$,
there are $p, q, r \in \on$ such that $l_p$ and $l_q$ span a cone in the fan of $X_{k - 1}$ and
$l_r = l_p + l_q$ represents the toric blow-up $b_k$. We have:

\begin{lemma}\label{pullbackexplicitlemma}
Let $p, q \in S \subset \on$ such that the $l_j$ with $j \in S$ span the fan of $X_{k - 1}$
and denote $D = \sum_{i \in S} c_i D_i$ a $T$-invariant divisor. Then $b_k^* D \sim
\sum_{i \in S} c_i D_i + (c_p + c_q) D_r$ and $\gamma_i R_i \sim c_r D_r$ on $X_k$.
\end{lemma}

\begin{proof}
Only the first assertion needs a proof. Let $L$ the matrix whose rows are the $l_i$ with $i \in S$ and
$L'$ the matrix consisting of the same rows as $L$ but with the additional row $l_p + l_q$ added between
$l_p$ and $l_q$. The assertion follows form the commutativity of the following diagram:
\begin{equation*}
\xymatrix{
0 \ar[r] & M \ar@{=}[d] \ar^{L}[r] & \Z^{\vert S \vert} \ar@{^{(}->}[d] \ar[r] & \pic(X_{i - 1})\ar@{^{(}->}[d]
\ar[r] & 0\\
0 \ar[r] & M \ar^{L'}[r] & \Z^{\vert S \vert + 1} \ar[r] & \pic(X_i) \ar[r] & 0.
}
\end{equation*}\end{proof}

For given $\gamma_k \leq 0$, we consider the lattice triangle which is inscribed by the lines
$H_p$, $H_q$, $H_r$ and whose lattice points we can count:

\begin{definition}
Let $l_p$, $l_q$, $l_r$ be as before and $\gamma_k \leq 0$, then we denote
\begin{enumerate}[(i)]
\item $T_{\gamma_k} := \{m \in M \mid l_p(m) \geq -c_p, l_q(m) \geq -c_q, l_r(m) \leq -c_r \}$,
\item $T^-_{\gamma_k} := \{m \in M \mid l_p(m) \geq -c_p, l_q(m) \geq -c_q, l_r(m) < -c_r \}$,
\item $T^+_{\gamma_k} := \{m \in M \mid l_p(m) > -c_p, l_q(m) > -c_q, l_r(m) \leq -c_r \}$.
\end{enumerate}
\end{definition}

As $l_p$ and $l_q$ form a basis of $N$, by translation by some $m \in M$ we can assume without loss
of generality that $c_p = c_q = 0$. Then, using Lemma \ref{pullbackexplicitlemma}, we can directly see
that the lattice points of $T_{\gamma_k}$, $T^+_{\gamma_k}$, $T^-_{\gamma_k}$ are counted by binomial
coefficients. We have $\vert T_{\gamma_k} \vert = \binom{\gamma_k - 1}{2}$, $\vert T_{\gamma_k}^- \vert =
\binom{\gamma_k}{2}$, and $\vert T_{\gamma_k}^+ \vert = \binom{\gamma_k + 1}{2}$.
This is illustrated in the following example.

\begin{example}\label{gammaarrangementexample}
With notation as before, figure \ref{gammaarrangement} shows the local configuration of $l_p$,
$l_q$, $l_r$ and the relative positions of $H_p, H_q, H_r$ for $\gamma_k = -3$.
\begin{figure}[htbp]
\begin{center}
\includegraphics[height=5cm]{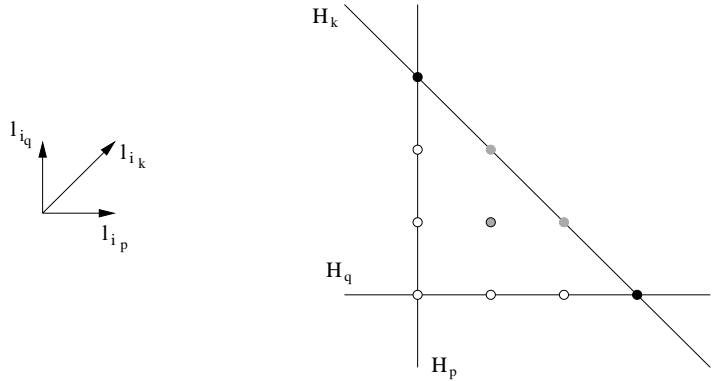}
\end{center}
\caption{Three primitive vectors $l_p$, $l_q$, $l_r$ which pairwise generate $N$ and
the corresponding orthogonal hyperplane arrangement for $\gamma_k = - 3$.}\label{gammaarrangement}
\end{figure}
The dots indicate the $\binom{-3 - 1}{2} = 10$ lattice points in $T_{\gamma_k}$, the gray dots the
$\binom{-3 + 1}{2} = 3$ lattice points in $T_{\gamma_k}^+$ and the circled dots the $\binom{-3}{2} = 6$ lattice
points in $T_{\gamma_k}^-$, with one lattice point in the intersection $T_{\gamma_k}^+ \cap T_i^-$.
\end{example}

By Proposition \ref{negativegammas}, a pre-left-orthogonal divisor $D$ is of the form $(D)_0 +
\sum_{i = 1}^t \gamma_i R_i$ with $(D)_0$ pre-left-orthogonal on $X_0$ and $\gamma_i \leq 0$
for every $i$. The following proposition shows that strong pre-left-orthogonality is equivalent
to that the $T_{\gamma_i}$ cut out the lattice points of $G_{(D)_0}$ in a well-behaved manner.

\begin{proposition}\label{cutoutproposition}
Let $k > 0$ and consider a blow-up $b_k : X_k \longrightarrow X_{k - 1}$ with notation as before. Let
$D$ be a pre-left-orthogonal divisor on $X_{k - 1}$ and $\gamma_k \leq 0$. Then $b_k^* D + \gamma_k R_k$ is
pre-left-orthogonal iff $T_{\gamma_k}^+ \subset G_D^\circ$. If $D$ is strongly pre-left-orthogonal,
then $b_k^* D + \gamma_k R_k$ is strongly pre-left-orthogonal iff $T_{\gamma_k}^+ \subset G_D^\circ$
and $T_{\gamma_k}^- \subset G_D$.
\end{proposition}

\begin{proof}
By Riemann-Roch we get $\chi(\gamma_k R_k) = 1 - \binom{\gamma_k}{2}$ and $\chi(-\gamma_k R_k) = 1 -
\binom{\gamma_k + 1}{2}$. Moreover, we get $\chi(b_k^* D + \gamma_k R_k) = \chi(D) + \chi(\gamma_k R_k)
- 1 = \chi(D) - \binom{\gamma_k}{2}$ and $\chi(-b_k^* D - \gamma_k R_k) = \chi(-D) + \chi(-\gamma_k R_k)
- 1 = \chi(-D) - \binom{\gamma_k + 1}{2}$. So we see that $h^1(-b_k^* D - \gamma_k R_k) = 0$ iff
$T^+_{\gamma_k}$ precisely cuts $\binom{\gamma_k + 1}{2}$ lattice points out of $G_D^\circ$ and
$h^1(b_k^* D + \gamma_k R_k) = 0$ iff $T^+_{\gamma_k}$ precisely cuts $\binom{\gamma_k}{2}$ lattice
points out of $G_D$ and the assertion follows.
\end{proof}

Consequently, we get:

\begin{corollary}\label{tilingcorollary}
Let $D = (D)_0 + \sum_i \gamma_i R_i$ pre-left-orthogonal. Then $D$ is left-orthogonal iff
$G_{(D)_0} \setminus G_D = \coprod_{k = 1}^t T_{\gamma_k}^-$. Moreover, $D$ is strongly left-orthogonal
iff  $G_{(D)_0} \setminus G_D = \coprod_{k = 1}^t T_{\gamma_k}^-$ and $G_{(D)_0}^\circ =
\coprod_{k = 1}^t T_{\gamma_k}^+$.
\end{corollary}

In terms of lattice figures in $M$, strong left-orthogonality can be understood by proposition
\ref{cutoutproposition} and corollary \ref{tilingcorollary} as follows. We start with an almost lattice
polytope associated to a strongly pre-left-orthogonal divisor $(D)_0$ on $X_0$ and successively cut out lattice
points of $G_{(D)_0}$ and $G_{(D)_0}^\circ$ by moving in hyperplanes $H_r$ until $G_{(D)_0 + \sum_k \gamma_k
R_k}^\circ$ is empty and the sets $\{T_{\gamma_k}^+ \mid k \in [t]\}$ and $\{T_{\gamma_k}^- \mid k \in [t]\}$
form a ``tiling'' of $G_{(D)_0} \setminus G_{(D)_0 + \sum_k \gamma_k R_k}$ and $G_{(D)_0}^\circ$, respectively.
We illustrate this in the following example.

\begin{example}\label{tilingexample}
Figure \ref{tilingexamplefigure} shows on the left the fan of $\mathbb{F}_2$ from figure
\ref{p2f2preleftorthogonalexamples} blown up three times by successively adding the primitive vectors
$l_1$, $l_3$, and $l_2$. Note that the numbering of the $R_j$ does not match with the numbering of the $l_i$, but
rather the order in which the $l_i$ were added to the fan.
The right side shows the hyperplane arrangements for five examples of
divisors $D$ all of which have
$(D)_0 = 3Q - P$, with $G_{3Q - P}$ and $G_{3Q - P}^\circ$ shown in figure \ref{p2f2preleftorthogonalexamples}.
In a) the hyperplanes $H_1$, $H_2$, $H_3$ are indicated.
\begin{figure}[htbp]
\begin{center}
\includegraphics[height=11cm]{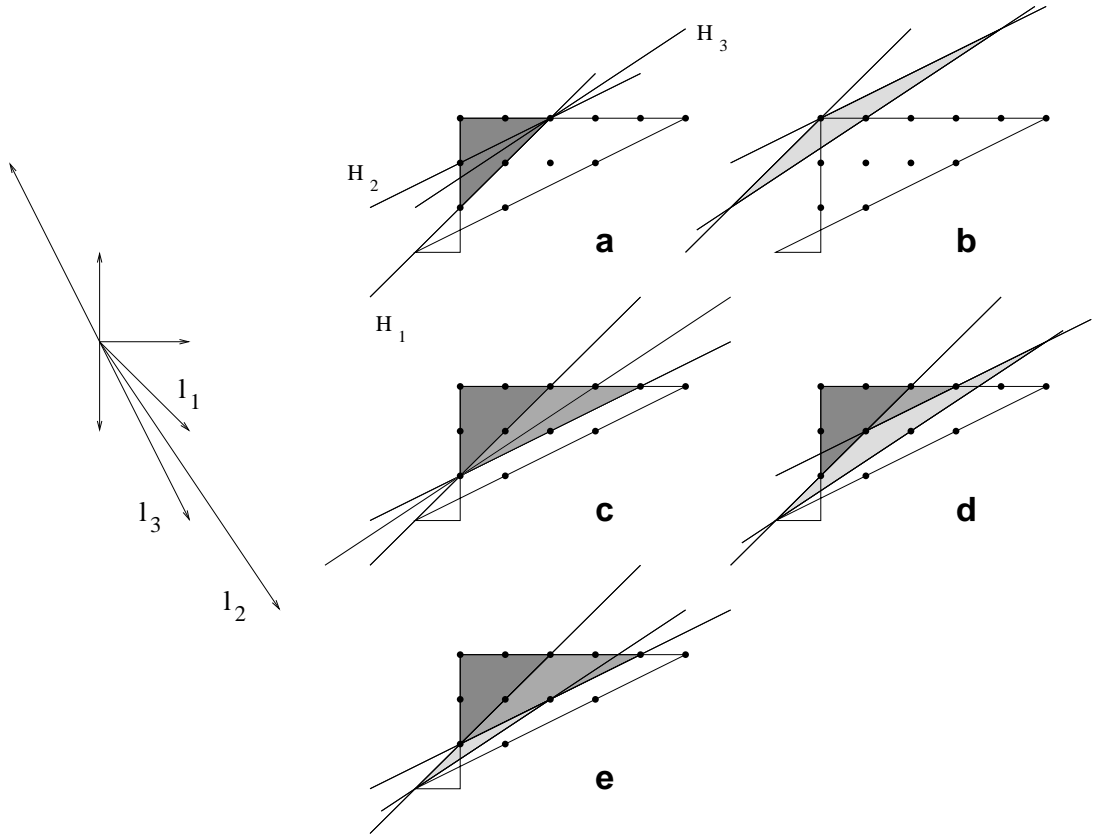}
\end{center}
\caption{The fan of $\mathbb{F}_2$ blown up three times and the hyperplane arrangements corresponding to the
divisors a) $3Q - P - 2 R_1$, b) $3Q - P - 2 R_3$,
c) $3Q - P - 2 R_1 - 2 R_2$, d) $3Q - P - 2 R_1 - R_2 - 2 R_3$, e) $3Q - P - 2 R_1 - 2 R_2 - R_3$.}
\label{tilingexamplefigure}
\end{figure}
The dark gray area indicates $T_{\gamma_1}$, the medium gray indicates $T_{\gamma_2}$, and the light
gray $T_{\gamma_3}$. In a) we have $D = 3Q - P - 2 R_1$; here $T^-_{\gamma_1}$ cuts out three elements
of $G_{3Q - P}$ and $T^+_{\gamma_1}$ cuts out one of $G_{3Q - P}^\circ$. Therefore $D$ is pre-left-orthogonal
in this case. In b) we have $D = 3Q - P - 2 R_1$ and $T^-_{\gamma_3}$ cuts out only one of $G_{3Q - P}$ and
$T^+_{\gamma_1}$ none of $G_{3Q - P}^\circ$. Therefore $D$ is not strongly pre-left-orthogonal. Note that $R_1$
and $R_3$ behave differently because $l_{i_1}$ does form a basis of $N$ with either of the two primitive
vectors which belong to the fan of $\mathbb{F}_2$ and in whose positive span $l_{i_1}$ is contained, whereas
$l_{i_3}$ does not. In the cases
c), d), e), all $T^-_{\gamma_i}$ and $T^+_{\gamma_i}$ cut out the correct number of lattice points of
$G_{3Q - P}$ and $G_{3Q - P}^\circ$, respectively, such that precisely the two elements in $G_{3Q - P}^\circ$
are cut out. So in all these cases $D$ is strongly left-orthogonal.
\end{example}

We will also need to know how we can pass from the coordinates associated to a minimal model $X_0$ to
the $d_i$-coordinates. For this we first illustrate the correspondence between divisors of the form
$\alpha P + \beta Q + \sum_{i = 1}^t \gamma_i R_i$ and polygonal lines of the form $\sum_{i \in \on} d_i l_i
= 0$ in the following example.

\begin{example}\label{f2polylinesexample}
It is convenient to interpret the relation $\sum_{i \in \on} d_i l_i = 0$ as closed polygonal lines. If we
successively place the vectors $d_i l_i$ end to end in $N_\Q$, we obtain a figure which can be viewed as
a polygonal line complex embedded in the arrangement $\{H_i\}_{i \in \on}$, rotated by $90$ degrees.
Figure \ref{f2polylinesfigure} shows examples of divisors on the surface shown in figure
\ref{tilingexamplefigure}.
\begin{figure}[htbp]
\begin{center}
\includegraphics[width=13cm]{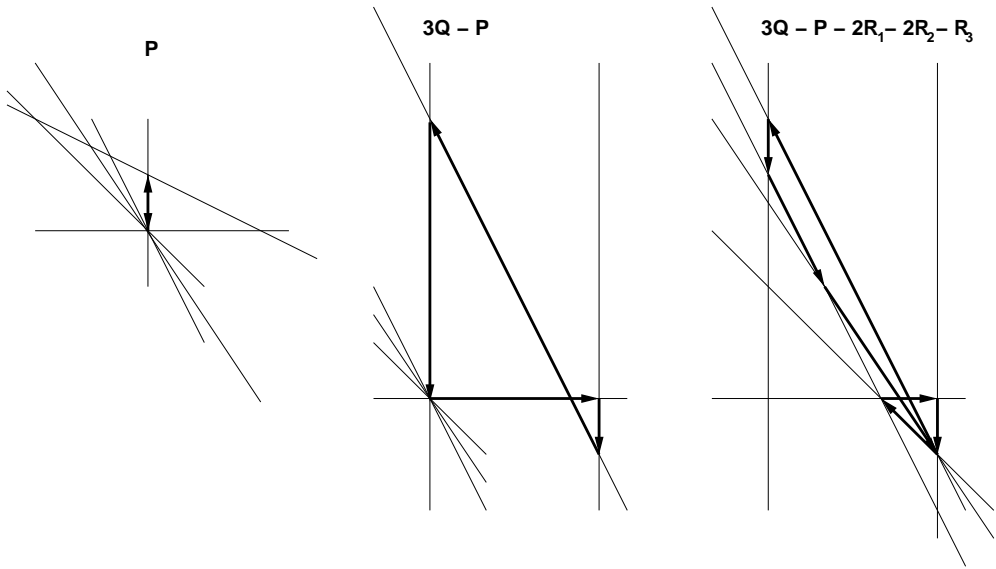}
\end{center}
\caption{The fan of figure \ref{tilingexamplefigure} and the polygonal lines associated to the divisors $P$,
$3Q - P$, and $3Q - P - 2 R_1 - 2 R_2 - R_3$. The picture shows the hyperplane arrangements associated to these
divisors rotated by $90$ degrees and the polygonal lines embedded into them.}\label{f2polylinesfigure}
\end{figure}
Note that the order in which the $d_i l_i$ are placed end to end is not canonical, but there are the two
obvious choices (clockwise or counterclockwise) by which the line complex can be interpreted as being embedded
in the corresponding hyperplane arrangement.
\end{example}

To change from coordinates associated to $X_0$ to $d_i$-coordinates, by linearity it suffices to consider
$\big(d_1(D), \dots, d_n(D)\big)$, where
$D$ is one of $P$, $Q$, $H$, $R_i$, $i \in \ot$. For the following lemma we assume that the fan of
$X_0$ is generated by $l_b, l_c, l_d, l_e$ if $X_0 \cong \mathbb{F}_a$ or by $l_b, l_c, l_d$ if
$X_0 \cong \mathbb{P}^2$. In the first case we assume that $l_b + l_d = a l_c$. With respect to
$R_k$, we choose $l_p, l_q, l_r$ as above. The following lemma is just an observation:

\begin{lemma}\label{dibasislemma}
\begin{enumerate}[(i)]
\item If $X_0 \cong \mathbb{P}^2$, then $d_i(H) = 1$ if $i \in \{b, c, d\}$ and $d_i(H) = 0$ otherwise.
\item If $X_0 \cong \mathbb{F}_a$, then $d_i(P) = 1$ if $i \in \{c, e\}$ and $d_i(P) = 0$ otherwise.
\item If $X_0 \cong \mathbb{F}_a$, then $d_c(Q) = a$, $d_i(Q) = 1$ if $i \in \{b, c\}$, and $d_i(Q) = 0$ otherwise.
\item Without assumptions on $X_0$ we have $d_p(R_k) = d_q(R_k) = 1$, $d_r(R_k) = -1$, and $d_i(R_k) = 0$ otherwise.
\end{enumerate}
\end{lemma}

If we compare figure \ref{f2polylinesfigure} with figure \ref{tilingexamplefigure}, we see that in these
examples, for strongly left-orthogonal $D$, the associated polygonal line contains $G_D$. More generally,
we get:

\begin{lemma}
Let $D = (d_1, \dots, d_n) \in N_1(X)$ be a $T$-invariant curve on a smooth complete toric surface $X$. If,
as a divisor, $D$ is numerically left-orthogonal, then $\chi(D) = \sum_i d_i$.
\end{lemma}

\begin{proof}
Let $E = \sum_{i \in \on} c'_i D_i$, then it follows from the discussion in section \ref{surfacegeometry}
that $D . E = \sum_{i \in \on} d_i c_i'$. We apply this to $E = -K_X = \sum_{i \in \on} D_i$ and use
Lemma \ref{acyclicitylemma} (\ref{acyclicitylemmai}).
\end{proof}

Of course, if $D$ is strongly left-orthogonal, then it follows that $h^0(D) = \sum_i d_i$. If moreover,
$(D)_0$ is strongly pre-left-orthogonal, it follows by induction, starting from the
classification of propositions \ref{P2leftorthogonals} and \ref{Hirzebruchpreleftorthogonals}, that
$G_D \subset \bigcup_{d_{i_k} \geq 0} H_k$, i.e. the positive $d_i$ attribute to the global sections
not only numerically, but the associated line segments bounding $G_D$ actually contain $G_D$.

By Proposition \ref{dinefproposition} a divisor $D$ is nef iff $d_i \geq 0$ for every $i$. Then the
associated polygonal line complex is the boundary of a lattice polytope in $M_\Q$. The figures of
example \ref{f2polylinesexample} show that these strongly left-orthogonal divisors are almost nef,
as in every case $d_i \geq -1$ for every $i \in \on$. This also holds in general:

\begin{proposition}\label{degreelowerbound}
Let $D$ be a strongly left-orthogonal divisor on $X$. Then $\sum_{i \in I} d_i(D) \geq -1$ for every cyclic
interval $I \subset \on$.
\end{proposition}

\begin{proof}
We choose some sequence of equivariant blow-downs to some minimal model $X_0$. Assume first that $(D)_0
= 0$. Then by Lemma \ref{verticalleftorthogonals} $D = R_k$ for some $k$ or $D = R_k - R_l$
for $k \neq l \in \ot$ and $R_k$, $R_l$ incomparable. For $p, q, r$ as above, we have by Lemma
\ref{dibasislemma} that $d_i(R_k) = -1$ for $i = r$, $d_i(R_k) = -1$ for $i \in \{p, q\}$ and
$d_i(R_k) = 0$ else.
So the assertion follows immediately for $D = R_k$. For $D = R_k - R_l$ we have just to take
into account that the $R_k$ and $R_l$ are incomparable. If $(D)_0 \neq 0$ we can assume without
loss of generality that $(D)_0$ is strongly pre-left-orthogonal. Otherwise, we have necessarily
$h^0(D) = h^0\big((D)_0) = 0$ and $-(D)_0$ is strongly pre-left-orthogonal. Then if the statement is
true for the case $-(D)_0$ strongly pre-left-orthogonal, we have $d_i \leq 1$ for every $i$ and therefore
by above discussion that $-d_i \geq -1$ for every $i$.

We show by induction on $(D)_k$, $k = 0, \dots, t$ that the assertion is true for a strongly
pre-left-orthogonal divisor $D$.
For $k = 0$, the assertion is true by inspection of the classification of strongly pre-left-orthogonal
divisors on $\mathbb{P}^2$ (proposition \ref{P2leftorthogonals}) and $\mathbb{F}_a$ (proposition
\ref{Hirzebruchpreleftorthogonals}). It also follows that $d_j = \vert G_{(D)_0} \cap H_i \vert$
if $l_i$ belongs to fan associated to $X_0$, i.e. the $d_i$ count the lattice length plus one of the
bounding faces of the polygonal line inscribing $G_{(D)_0}$. In the induction step we will show that this
is still true for all triples $p, q, r$ and all $k > 0$. For $k > 0$, let $(D)_k - (D)_{k - 1} =
\gamma_k R_k$. Consider the triple $l_p$, $l_q$, $l_r$ as before, by Proposition
\ref{cutoutproposition} it is a necessary condition that $H_p$ and $H_q$ intersect in some
$m \in G_{(D)_k} \setminus G_{(D)_k}^\circ$. Moreover, necessarily $d_p, d_q \geq -\gamma_k - 1$
and the result follows from above characterization of $d_i(R_k)$.
\end{proof}

\begin{remark}
If $a_i = -1$ for some $i$, then we can find a basis of $\pic(X)$ with respect to some minimal model
$X_0$ such that $R_t = D_i$. For any strongly pre-left-orthogonal divisor $D$ it follows that $D =
(D)_{t - 1} + \gamma_t R_t$ for some $\gamma_t \leq 0$. Therefore, we have $d_i \geq 0$. If $a_i \geq 0$,
the divisor $D_i$ necessarily is the strict transform of some torus invariant divisor on $X_0$. So
by the classification \ref{P2leftorthogonals} and \ref{Hirzebruchpreleftorthogonals}, the only cases
with $a_i \geq 0$ and $d_i(D) = -1$ is where $X_0 \cong \mathbb{P}^1 \times \mathbb{P}^1$ and $D$ is the
pullback of $P - Q$ or $Q - P$. Otherwise, if $a_i > 0$, then $d_i \geq 0$.
\end{remark}

\section{Strongly exceptional sequences of invertible sheaves on toric surfaces}\label{torictheorems}

The following results give a full classification of strongly exceptional sequences of invertible sheaves
on smooth complete toric surfaces.

\begin{theorem}\label{toricclassification}
Let $X$ be a smooth complete toric surface, then for every strongly exceptional toric system $\mathcal{A}$
there exists a sequence of blow-downs $X = X_t \rightarrow \cdots \rightarrow X_0$, where $X_0 =
\mathbb{P}^2$ or $X_0 = \mathbb{F}_a$ for some $a \geq 0$ such that the normal form of $\mathcal{A}$ is a
standard augmentation from $X_0$.
\end{theorem}

As a corollary of Theorems \ref{augmentationboundtheorem} and \ref{toricclassification} we thus
obtain:

\begin{theorem}\label{toricbound}
Let $X \neq \mathbb{P}^2$ be a smooth complete toric surface. Then there exists a full strongly
exceptional sequence of invertible sheaves on $X$ if and only if $X$ can be obtained by equivariantly
blowing up a Hirzebruch surface two times (in possibly several points in each step).
\end{theorem}

We will prove Theorem \ref{toricclassification} in the remaining sections. In this section we
will state and prove some of its direct consequences.

\begin{corollary}
Let $X$ be a smooth complete toric surface. If there exists a strongly exceptional sequence of invertible
sheaves on $X$, then $\rk \pic(X) \leq 14$.
\end{corollary}

\begin{proof}
A Hirzebruch surface $\mathbb{F}_a$ has four torus fixed points. So, after blowing up some of these
points, the resulting toric surface has up to $8$ fixed points. After blowing up these, we get a toric
surface $X$ whose fan is generated by at most $16$ lattice vectors and thus $\rk \pic(X) \leq 14$,
and the statement follows from Theorem \ref{toricbound}.
\end{proof}

\begin{example}\label{counterexampleblowup}
Consider the toric surface which is given by the sequence of self-intersection numbers
$-2, -2, -1, -3, -2, 0, 1$. It is easy to see that there is no way to blow-down this surface to any
Hirzebruch surface in only two steps. So by Theorem \ref{toricbound} there does not exist a strongly
exceptional sequence of invertible sheaves on this surface. This is the counterexample which has been
verified by explicit computations in \cite{HillePerling06}. Now consider the blow-up of this surface
given by $-2, -2, -1, -3, -2, -1, -1, 0$. This surface can be blown-down to a $\mathbb{F}_1$ in two steps by
simultaneously blowing down two divisors in each step. Therefore by Theorem \ref{twotimesblowupexistence}
there exist strongly exceptional sequences of invertible sheaves on this surface. More concretely, if
the $\mathbb{F}_1$ is spanned by lattice vectors $l_1, l_2, l_3, l_6$ with $l_3 = l_2 + l_6$, we
subsequently add $l_7 = l_1 + l_6$, $l_8 = l_1 + l_7$, $l_4 = l_3 + l_6$ and $l_5 = l_4 + l_6$. Then, for
example, we get a family of strongly exceptional toric systems by
\begin{equation*}
R_1, R_3 - R_1, P - R_3, sP + Q, P - R_2, R_2 - R_4, R_4, -(s + 1)P + Q - R_1 - R_2 - R_3 - R_4
\end{equation*}
for $s \geq -1$.
\end{example}

For a cyclic strongly exceptional toric system $\mathcal{A}$ on $X$ the associated toric surface
$Y(\mathcal{A})$ has a nef anti-canonical divisor. It turns out that this even is a necessary condition
for $X$ if $X$ itself is a toric surface:

\begin{theorem}\label{toriccyclicnef}
Let $X$ be a smooth complete toric surface. If there exists a cyclic strongly exceptional sequence of
invertible sheaves on $X$, then its anti-canonical divisor is nef.
\end{theorem}

\begin{proof}
By Proposition \ref{canonicalnefproposition} we have to show that $a_i \geq -2$ for every $i$. Assume
that $\mathcal{A} = A_1, \dots, A_n$ is a cyclic strongly exceptional toric system and
assume that $a_i < -2$ for some $i$. We denote $d_i^j := d_i(A_j)$ for every $j \in \on$. Then
$\sum_{j \in \on} d_i^j = a_i + 2 < 0$ by Proposition \ref{canonicalnefproposition}.
Because $\mathcal{A}$ is cyclic and strong,
every sum $\sum_{j \in I} A_j$ is strongly left-orthogonal for every proper cyclic interval
$I \subset \on$. In particular, $\sum_{j \in I} d_i^j \geq -1$ for every such $I$ by Proposition
\ref{degreelowerbound}. Now assume that
there exists $j \in \on$ such that $d_i^j = -1$. Without loss of generality, we can assume that $j = 1$.
Then by choosing a decomposition $\on \setminus \{1\} = I_1 \coprod I_2$, where $I_1, I_2$ are intervals,
we can consider $A_1, A_1', A_2'$, a short toric system of length $3$ as in example
\ref{shorttoricsystemexample1}. Then
$d^1_i + d_i(A_1') \geq -1$ and $d^1_i + d_i(A_2') \geq -1$, hence $d_i(A_1') \geq 0$ and $d_i(A_2')
\geq 0$, and we get $a_i \geq -3$. Now assume that $a_i = -3$. Then there exist at least two $j$ such
that $d_i^j = -1$; because otherwise, if there was only one $j$ with $d_i^j = -1$, the condition that
$\sum_{j = 1}^n d_i^j = -1$ would imply that $d_i^k = 0$ for all $k \neq j$ and thus all the $A_k$
with $k \neq j$ are contained in a hyperplane in $\pic(X)$, which is not possible. Let $j, k$ such that
$d_i^k, d_i^j = -1$. Then $\vert k - j \vert > 1$, as $A_l + A_{l + 1}$ must be strongly left-orthogonal for
every $l \in \on$. So we can consider a short toric system to periodicity $4$:
$A'_1, A'_2, A'_3, A'_4$ with $d_i(A'_1) = d_i(A'_3) = -1$. As $A'_1 + A'_2 + A'_3$ and $A'_2 + A'_3 + A'_4$
must be strongly left-orthogonal, this implies that $d_i(A'_2), d_i(A'_4) \geq 1$ and so $a_i \geq -2$, a
contradiction.
\end{proof}

The converse is also true in the toric case:

\begin{theorem}\label{toriccyclicexistence}
If $X$ is a smooth complete toric surface with nef anti-canonical divisor, then there exists a full cyclic
strongly exceptional sequence of invertible sheaves on $X$.
\end{theorem}

\begin{proof}
The case of $\mathbb{P}^2$ is clear, and Hirzebruch surfaces are covered by \ref{Hirzebruchtoricsystems}.
For the remaining two del Pezzo surfaces the existence follows from Theorem \ref{delPezzocyclicexistence}.
For the other cases, we give in table \ref{cyclictable} a list of examples, one for each surface.
\begin{table}[htbp]
\centering
\begin{tabular}{|c|c|c|}
\hline
5b & \underline{-1}, \underline{-2}, 0, 1, -1 & $H - R_1, R_1, H - R_1 - R_2, R_2, H - R_2$ \\ \hline
6b & \underline{-1}, \underline{-2}, -1, \underline{-1}, 1, -1 & $H - R_1 - R_3, R_1, H - R_1 - R_2, R_2,
H - R_2 - R_3, R_3$ \\ \hline
6c & \underline{-1}, \underline{-2}, 0, 0, \underline{-1}, -2 & $H - R_1 - R_3, R_1, H - R_1 - R_2, R_2,
H - R_2 - R_3, R_3$ \\ \hline
6d & \underline{-1}, \underline{-2}, -2, 0, 1, -2 & $P - R_1, R_1, Q - R_1 - R_2, R_2, P - R_2, Q - P$ \\ \hline
7a & \underline{-1}, -1, -1, \underline{-1}, \underline{-2}, -1, \underline{-2} & $H - R_1 - R_2, R_2,
R_1 - R_2, H - R_1 - R_3 - R_4$,\\
& &  $R_4, R_3 - R_4, H - R_3$ \\ \hline
7b & \underline{-1}, \underline{-2}, 0, -1, \underline{-1}, \underline{-2}, -2 & $H - R_1 - R_3, R_3,
R_1 - R_3, H - R_1 - R_2 - R_4$,\\
& &  $R_4, R_2 - R_4, H - R_2$ \\ \hline
8a & \underline{-1}, -2, \underline{-1}, -2, \underline{-1}, -2, \underline{-1}, -2 & $P - R_1 - R_4,
R_1, Q - R_1 - R_2, R_2, P - R_2 - R_3$, \\
& &  $R_3, Q - R_3 - R_4, R_4$ \\ \hline
8b & \underline{-1}, -2, \underline{-1}, -1, \underline{-2}, \underline{-1}, -2, \underline{-2} &
$H - R_1 - R_2 - R_4, R_4, R_2 - R_4, R_1 - R_2, H - R_1 - R_3$, \\
& & $R_3 - R_5, R_5, H - R_3 - R_5$ \\ \hline
8c & \underline{-1}, \underline{-2}, -2, \underline{-2}, \underline{-1}, -2, 0, -2 & $P - R_1 - R_4,
R_4, R_1 - R_4, P + Q - R_1 - R_3, R_3 - R_2$, \\
& & $R_2, P - R_2 - R_3, - P + Q$\\ \hline
9 & \underline{-1}, \underline{-2}, -2, \underline{-1}, \underline{-2}, -2, \underline{-1}, \underline{-2},
-2 & $H - R_1 - R_4 - R_5, R_4, R_1 - R_4, H - R_1 - R_3 - R_6$, \\
& &  $R_6, R_3 - R_6, H - R_2 - R_3 - R_5, R_2, R_5 - R_2$ \\ \hline
\end{tabular}
\caption{Cyclic strongly exceptional toric systems on toric surfaces with nef anti-canonical divisor.}
\label{cyclictable}
\end{table}
By construction, these toric systems are exceptional and to check that these are indeed cyclic strongly
exceptional is a direct application of Proposition \ref{cutoutproposition} and Corollary
\ref{tilingcorollary}. Note that for 8a and 8c we have given examples which are augmentations
of cyclic strongly toric systems on $\mathbb{P}^1 \times \mathbb{P}^1$ and there is an ambiguity
of assigning $P$ and $Q$. For 8a, both cases are cyclic strongly exceptional. For 8c, we choose $Q$
to be the class of the unique torus invariant prime divisor with self-intersection zero on 8c.
\end{proof}

\section{Straightening of strongly left-orthogonal toric divisors}\label{straighteningsection}

In order to proof Theorem \ref{toricclassification} we classify strongly
left-orthogonal divisors on a given toric surface $X$. For this, we introduce in this section a
procedure for simplifying a given strongly left-orthogonal divisor. We call this procedure a
{\em straightening}. We will classify strongly left-orthogonal divisors up to straightening.

\begin{lemma}\label{wellreductionlemma}
Let $D$ be a $T$-invariant strongly left-orthogonal divisor on $X$ and $i \in \on$ such that $D_i^2 = -1$.
If $d_i(D) < 0$, then either $h^0(D) = 0$ or $D = D_i$.
\end{lemma}

\begin{proof}
We write $D = \gamma_t D_i + (D)_{t - 1}$, where $X_{t - 1}$ is the blow-down of $X$ along $D_i$.
Then $d_i = -\gamma_t$ by Proposition \ref{dinefproposition} and Lemma \ref{coordinatedegreelemma}.
By Proposition \ref{negativegammas} and Remark \ref{negativegammasremark} this implies that $(D)_0$
is not pre-left-orthogonal with respect to the choice of any minimal model $X_0$ for $X$ which factorizes
through $X_{t - 1}$. But then we either have $h^0(D) = 0$ or $(D)_0 = 0$ or both. If $(D)_0 = 0$, then
by Proposition \ref{verticalleftorthogonals} we have $D = R_p - R_q$ for some $p, q \in \ot$ or $D = R_p$
for $p \in \ot$. In the first case, we also get $h^0(D) = 0$, in the second, we necessarily have
$R_p = D_i$ by Lemma \ref{dibasislemma}.
\end{proof}

So for any strongly left-orthogonal divisor $D$ which is not a prime divisor $D_j$, we will assume
without loss of generality that $d_i \geq 0$ for any $i \in \on$ such that $D_i^2 = -1$. Otherwise,
we will just take $-D$ instead of $D$. Let us write $D = \gamma_t D_i + (D)_{t - 1}$ for $X
\rightarrow X_{t - 1}$ the blow-down of $D_i$. If $-1 \leq \gamma_t \leq 0$, then $T_{\gamma_t}^+ =
\emptyset$ and it follows from Lemma \ref{unwindinglemma}, Proposition \ref{cutoutproposition}, and
Corollary \ref{tilingcorollary} that $(D)_{t - 1}$ is strongly left-orthogonal on $X_{t - 1}$.
By iterating, we obtain a sequence of blow-downs $X = X_t \rightarrow \cdots \rightarrow
X_s$, where $s \geq 0$ and $X_s$ lies over some (not necessarily completely specified yet) minimal model
$X_0$. We can write $D = (D)_s + \sum_{i = s + 1}^t \epsilon_i R_i$, where $\epsilon_i \in \{0, - 1\}$
for every $i$ and $R_i$ is the total transform on $X$ of the exceptional
divisor of the blow-up $X_i \rightarrow X_{i - 1}$. The divisor $(D)_s$ now has the property that
either $(D)_s$ coincides with a prime divisor $D_i$ on $X_s$ with $D_i^2 = -1$ or
$d_i\big((D)_s\big) \geq 2$ with respect to every $T$-invariant prime divisor $D_i$ on $X_s$ with
$D_i^2 = -1$. It follows from Corollary \ref{acyclicitylemma} (\ref{acyclicitycorollaryiv}) that
$h^0(D) = h^0\big((D)_s\big) + s - t$.

\begin{definition}
Assume $(D)_s$ is constructed as above and does not coincide with a prime divisor $D_i$ on $X_s$. Then
we call $(D)_s$ a {\em straightening} of $D$. A divisor $D$ is {\em straightened} if
$D = (D)_s$ (and consequently $X = X_s$).
\end{definition}

In the sequel we will keep the index `$s$' to denote that $X_s$ has been chosen with respect to the
straightening of some strongly left-orthogonal divisor. In general, $s \neq 0$ and a straightening $(D)_s$
is not unique. However, we will show that the existence of a straightened divisor imposes a strong condition
on the geometry of $X$.

\begin{proposition}\label{straighteningproposition}
Let $X$ be a smooth complete toric surface and $D$ a straightened divisor on $X$.
Then either $-K_X$ is nef or $X \cong \mathbb{F}_a$ with $a \geq 3$.
\end{proposition}

To prove Proposition \ref{straighteningproposition} we first show an auxiliary statement.
Let $f \in \on$ and denote $e_1, \dots, e_r$, $g_1, \dots, g_u \in \on$ all indices $i$ such that
$l_f$ and $l_i$ form a basis of $N$, where the enumeration is as follows. Consider the line generated by
$l_f$ in $N_\Q$, Then all the $e_i$ are contained in one half plane bounded by this line and
all the $g_j$ in the other. Moreover, we require that for any $i < j$, the vector $l_{e_j}$ is contained in
the cone generated by $l_f$ and $l_{e_i}$, and $l_{g_j}$ is contained in the cone generated by $l_f$ and
$l_{g_i}$, respectively. We denote $S \subset \on$ all $i$ such that $l_i$ is contained in one of the cones
$\sigma_1$, $\sigma_2$, where $\sigma_1$ is generated by $l_{e_1}$ and $l_f$, and $\sigma_2$ is generated
by $l_{g_1}$ and $l_f$. Let $D = \sum_{i \in \on} c_i D_i$ be a $T$-invariant divisor. We denote
\begin{align*}
Z_f & := \{m \in M \mid l_i(m) = -c_f + 1\},\\
Z_f' & := \big\{m \in Z \mid l_i(m) > -c_i \ {\rm for\ all } \ i \in \{f, e_1, \dots, e_j, g_1, \dots, g_k\}\big\},\\
Z_f'' & := \big\{m \in M \mid l_f(m) = 0 \ {\rm and } \ l_i(m) \geq 0 \ {\rm for } \ i \in \{e_1, \dots, e_r, g_1, \dots, g_u\}\big\}.
\end{align*}

\begin{lemma}\label{pivotlemma}
If $a_f \leq -3$, then there exists $m \in Z_f$ such that $l_i(m) > -c_i$ for all $i \in S$.
\end{lemma}

\begin{proof}
It follows from Proposition \ref{smoothblowdownproposition} that there exists
a sequence of blow-downs $X = X_t \rightarrow \cdots \rightarrow X_p$ such that the cones generated by
$l_f$, $l_{e_1}$ and $l_{g_1}$ do not contain any lattice vector which belongs to the fan associated to
$X_p$. Correspondingly, we have injective maps $\phi: [r] \rightarrow [t]$, $\psi: [s] \rightarrow [t]$
such that $R_{\phi_i}$, $R_{\psi_j}$ are the total transforms the exceptional divisors associated to
the primitive vectors $l_{e_i}$ and $l_{g_j}$, respectively. Then for $i < j$, we have $R_{\phi_i} <
R_{\phi_j}$ and $R_{\psi_i} < R_{\psi_j}$, respectively, and $R_{\phi_i}$, $R_{\psi_j}$ incomparable
for all $i, j$. Note that we have the relations $0 = a_f l_f + l_{e_j} + l_{g_k}$, where $a_f = D_f^2$, and
$0 = l_{e_1} + b l_f + l_{g_1}$ for some $b \geq a_f$.
We write $D =(D)_p + \sum_{i = p + 1}^t \gamma_i R_i$. Then the $T_{\gamma_{\phi_i}}^+$,
$T_{\gamma_{\phi_i}}^-$ and $T_{\gamma_{\psi_i}}^+$, $T_{\gamma_{\psi_i}}^-$ have to fulfill the conditions
of Proposition \ref{cutoutproposition} and Lemma \ref{tilingcorollary}.
In particular, we have $d_f = d_f(D) = c_{e_1} + b c_f + c_{g_1} + \sum_{i = 1}^j \gamma_{\phi_i} +
\sum_{i = 1}^k \gamma_{\psi_i}$ with $d_f \geq -1$ by Proposition
\ref{degreelowerbound}.
Let $l_{e_r}(m) = -c_{e_r} + k_{e_r}$ and $l_{g_u}(m) = -c_{g_u} + k_{g_u}$ for some $m \in Z_f$ and $k_{e_r},
k_{g_u} \in \Z$. Then we have $k_{e_r} + k_{g_u} = c_{e_r} + a_f c_f + c_{k_u} - a_f = d_f - a_f$. The number
of solutions such that $k_{e_r}, k_{g_u} > 0$ is given by $\max \{0, d_f - a_f - 1 \geq 1\}$, which is always
nonzero for $a_f \leq -3$. We denote . We
claim that if $a_f \leq -3$ then there exists $m \in Z_f'$ such that $l_i(m) > -c_i$ for all $i \in S$.
Assume there exists $i \in S \setminus \{f, e_1, \dots, e_r, g_1, \dots, g_u\}$ such that
$l_i(m) \leq -c_i$ for some $m \in Z_f'$. Without loss of generality, we assume that $l_i$ is contained in
$\sigma_1$. As $l_i$ and $l_f$ do not form a basis of $N$, then the fact that the hyperplane $H_i$ cuts
out lattice points in $T'$ implies that $H_i$ also cuts out at least the same number of lattice points
$m$ of $Z_f''$. But because $a_f \leq -3$, we have $\vert Z_f' \vert > \vert Z_f'' \vert$ and the
claim follows.
\end{proof}

\begin{proof}[Proof of Proposition \ref{straighteningproposition}]
If there does not exist $f \in \on$ such that $a_f < -2$, then $-K_X$ is nef by Proposition
\ref{canonicalnefproposition}. So if there exists such an $f$ we show that $X_s \cong \mathbb{F}_a$
for $a \geq 3$. With above notation there exists $m \in M$ such that $l_i(m) > -c_i$ for all $i \in S$
by Lemma \ref{pivotlemma}. Assume first that there exists $u \in \on$ such that $l_u = -l_f$.
In this case there do not exist $l_v$ which are contained in one of the cones generated by $l_u$ and $l_{e_1}$
or $l_u$ and $l_{g_1}$, respectively, because any blow-up of one of these cones would require a lattice
vector $l_i$ which forms a basis of $N$ together with $l_u$ and therefore with $l_f$. This lattice vector
then would be one of the $l_{e_i}$ or $l_{g_j}$, which is excluded by assumption.
But then the hyperplane $H_u$ must pass through $Z_f$, as otherwise $h^2(-D) \neq 0$, and $(D)_0
= kP + Q$, where $k \geq -1$, with respect to the minimal model $X_0$ associated to the fan generated by
$l_{e_1}$, $l_{g_1}$, $l_f$, $l_u$. But $G^\circ_{nP + Q} = \emptyset$ and thus $\gamma_i \in \{0, -1\}$ for
all $p < i \leq t$ and in fact $\gamma_i = 0$, as $D$ is straightened. This implies $X = X_0 \cong
\mathbb{F}_{\vert b \vert}$, where $l_{e_1} + b l_f + l_{g_1} = 0$. Such an $l_u$ necessarily exists in
the following cases. If $a > 1$,
then by the classification of toric surfaces $l_f$ must belong to any minimal model
for $X$ which can be obtained by blowing down $X_p$, and there necessarily exists $l_u = -l_f$. If $a = 1$,
then $l_{e_1}$ and $l_{g_1}$ form a basis of $N$ and the blow-up of the cone generated by these two just
yields $l_u$. So either $X_0 = \mathbb{F}_1$ or $X_0 = \mathbb{P}^2$. If $a < -1$, then none of $l_{e_1}$,
$l_{g_1}$, $l_f$ can be blown-down and thus together with $-l_f$ must span the fan of a minimal model
$\mathbb{F}_{\vert b \vert}$.

It remain to consider the cases $b \in \{0, -1\}$ and there is no $u \in \on$ with $l_u = -l_f$.
If $a = 0$, then
$l_{e_1} = -l_{g_1}$ and $l_{e_1}$, $l_{g_1}$, $l_f$ must be part of a fan of any minimal model $X_0$
which is a blow-down of $X_p$. Moreover, there exists $l_{v_1}$ such that $l_f + b l_{v_1} + l_{e_1} = 0$,
where without loss of generality $b > 0$ (and therefore $b > 1$), and all $l_i$ in the fan associated to
$X_p$ for $i$ different from $e_1$, $g_1$, $f$, $v_1$, are contained in the cone generated by $l_{v_1}$ and
$l_{g_1}$. Then we have $(D)_0 = kP + lQ$ with respect to the coordinates in $\pic(X_0)$, where
the fan of $X_0$ is generated by $l_{e_1}$, $l_{g_1}$, $l_f$, $l_{v_1}$. The divisor $(D)_0$ is strongly
pre-left-orthogonal and for any $i \notin \{e_1, g_1, f, v_1\}$, the index of the subgroup of $N$ generated
by $l_f$ and $l_i$ is at least $3$. Let $v_1, \dots, v_w \subset \on$ denote all elements such that
$l_{v_i}$ forms a basis of $N$ together with $l_{g_1}$ and denote $D = (D)_0 + \sum_{i = 2}^w \gamma_{v_i}
R_i +$ rest. Then $\sum_{i = 2}^w \gamma_{v_i} \leq k + 1$ and because the index of the subgroup of $N$
generated by $l_f$ and one of the $l_{v_i}$ with $i > 1$ is at least $3$ and we have $\coprod_{i = 2}^w
T^+_{\gamma_{\eta_i}} \cap Z_f' = \emptyset$, where $\eta: \{2, \dots, w\} \rightarrow \on$ is the injective
map which associates the $R_i$ to the elements $v_2, \dots, v_w$. Hence $Z_f'$ must be empty and therefore
$a_f \geq -2$.

In the last case, $a = -1$, for every $i \notin \{e_1, g_1, f\}$ with $l_i$ part of the fan associated to
$X_p$, by our assumptions the index of the subgroup of $N$ generated by $l_f$ and $l_i$ is at least two
and, similarly as in the previous case, we have $\coprod_{i \in K} T^+_{\gamma_{\eta_i}} \cap Z_f' =
\emptyset$, where $K \subset \on$ denotes those $i$ such that $l_i$ in the complement of $\sigma_1$ and
$\sigma_2$. Hence we have $a_f \geq -2$.
\end{proof}

Using Corollary \ref{tilingcorollary} and Proposition \ref{straighteningproposition} it is a rather
straightforward exercise to go through table \ref{weakdelpezzofigure} and to find all possible straightened
divisors.

\begin{proposition}\label{straightenedclassification}
Table \ref{straightenedlist} shows a complete list of straightened divisors and their associated toric
surfaces.
\end{proposition}

\begin{table}[htbp]
\centering
\begin{tabular}{|c|c|c|} \hline
$\mathbb{P}^2$ & 1 1 1 & H, 2H \\ \hline
$\mathbb{P}^1 \times \mathbb{P}^1$ & 0 0 0 0 & $P + sQ, Q + sP$, where $s \geq -1$ \\ \hline
$\mathbb{F}_1$ & 0 -1 0 1 & $P$, $Q + sP$, where $s \geq 1$ \\ \hline
$\mathbb{F}_2$ & 0 -2 0 2 & $P, 2Q - P, Q + sP$, where $s \geq -1$ \\ \hline
$\mathbb{F}_a$, $a \geq 3$ & 0 -a 0 a & $P$, $Q + sP$, where $s \geq -1$ \\ \hline
6d & \underline{-1} \underline{-2} \underline{-2} 0 1 -2 & $3H - 2 R_1 - R_2 - R_3$
\\ \hline
8a & \underline{-1} -2 \underline{-1} -2 \underline{-1} \underline{-2} -1 \underline{-2} &
$4H - 2(R_1 + R_2 + R_3) - R_4 - R_5$ \\ \hline
8c & \underline{-1} \underline{-2} \underline{-2} -2 \underline{-1} \underline{-2} 0 -2 &
$4H - 2(R_1 + R_2 + R_4) - R_3 - R_5$ \\ \hline
9 & \underline{-1} \underline{-2} -2 \underline{-1} \underline{-2} -2 \underline{-1} \underline{-2} -2 &
$4H - 2(R_1 + R_3 + R_5) - R_2 - R_4 - R_6$ \\ \hline
\end{tabular}
\caption{Classification of straightened divisors. The first column of the table shows the name of the surface
as given in table \ref{weakdelpezzofigure}, the second column shows the self-intersection numbers of the
toric divisors, and the third columns lists the straightened divisors on the surface. The underlined
intersection numbers indicate which divisors are blown-down to obtain a minimal model and the numbering of
the $R_i$ is just the left-to-right order of the underlined divisors.
}\label{straightenedlist}
\end{table}
It turns out that there exist only four straightened divisors which are realized on toric surfaces different
from $\mathbb{P}^2$ or $\mathbb{F}_a$. Their associated hyperplane arrangements and polygonal lines are
shown in figure \ref{straightenedfigure}.

\begin{figure}[htbp]
\begin{center}
\includegraphics[width=13cm]{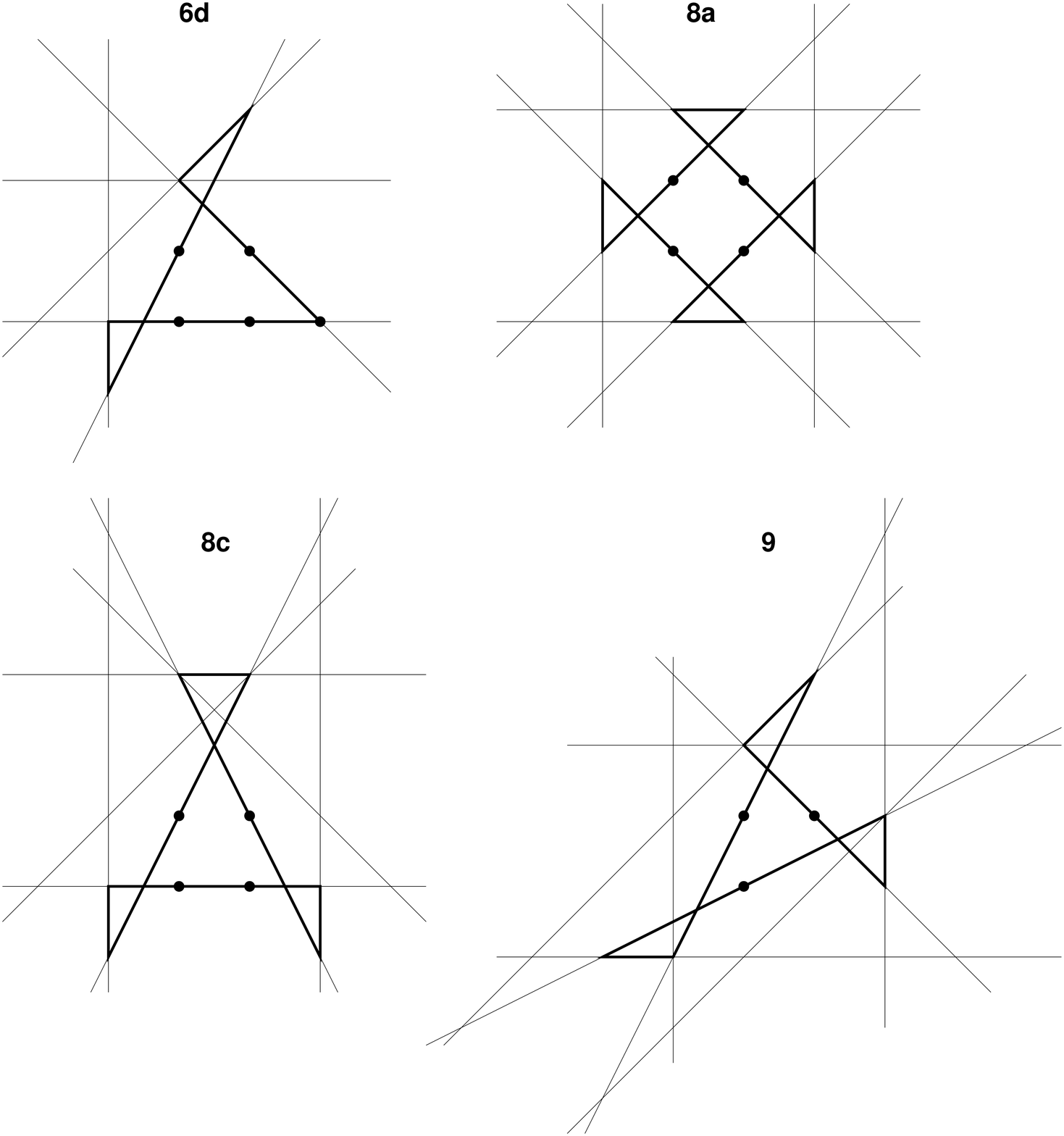}
\end{center}
\caption{The hyperplane arrangements and the polygonal lines associated to the four straightened divisors
which are not realized on $\mathbb{P}^2$ or a Hirzebruch surface. The dots indicate the global sections.}\label{straightenedfigure}
\end{figure}

\section{Proof of Theorem \ref{toricclassification}}\label{toricproofs}

Let $\mathcal{A} = A_1, \dots, A_n$ be a strongly exceptional toric system on $X$. The first step for
proving Theorem \ref{toricclassification} is to consider the straightening of $A := \sum_{i = 1}^{n - 1} A_i$
and to find a preferred coordinate system for $\pic(X)$ with respect to $\mathcal{A}$. The idea here is that
by Proposition \ref{straightenedclassification} there are only the few possibilities for $X_s$ listed in
table \ref{straightenedlist}, which are already close to a minimal model $X_0$. It follows from
Proposition \ref{straighteningextensionlemma} that every strongly exceptional sequence on $X$ is an
augmentation of a sequence on $X_s$. In the case where $X_s$ is the projective plane or a Hirzebruch
surface, we have $X_s = X_0$ and so by definition every augmentation of a strongly exceptional toric
system on $X_s$ is a standard augmentation. If If $X_s$ is isomorphic to 6d, then the assertion of
the theorem follows from Proposition \ref{6dstraightening}. In remaining cases, i.e. $X_s$ is one of
8a, 8c, 9, we show in Proposition \ref{nonaugmentablelemma} that $X = X_s$. These three cases are
analyzed in Propositions \ref{8astraightening}, \ref{8cstraightening}, and \ref{9straightening}, which
show that in every case $\mathcal{A}$ is a standard augmentation on $X_s$. This completes the proof
of Theorem \ref{toricclassification}.

Moreover, we draw the
following corollary from Propositions \ref{8astraightening}, \ref{8cstraightening}, and \ref{9straightening}:

\begin{corollary}
If $X_s$ is one of 8a, 8c, 9, then $X = X_s$ and $\mathcal{A}$ is cyclic.
\end{corollary}

Now we prove the statements mentioned above.

\begin{proposition}\label{straighteningextensionlemma}
Every strongly exceptional toric system has a normal form which is an augmentation of a strongly exceptional
toric system on $X_s$.
\end{proposition}

\begin{proof}
Let $\mathcal{A} = A_1, \dots, A_n$ be a strongly exceptional toric system and $A := \sum_{i = 1}^{n - 1} A_i$
and $(A)_s$ the straightening of $A$. We assume that $X \neq X_s$ and denote $R_t, \dots, R_{s - 1}$ the
total transforms of the exceptional divisors of the blow-ups $b_1, \dots, b_{s - 1}$ and complete these to a
basis of $\pic(X)$ with respect to some $X_0$ which is a blow-down of $X_s$. We may now assume that
$\mathcal{A}$ is in normal form. The divisor $R_t$ represents a torus invariant
prime divisor of self-intersection $-1$ on $X$. Then $A = (A)_{t - 1} + \gamma_t R_t$, where $\gamma_t \in
\{0, -1\}$, and $A_n = (A_n)_{t - 1} + \delta_t R_t$, where $\gamma_t + \delta_t = -1$.
There must be at least two of the $A_i$ which are not contained in the hyperplane $R_t^\bot$, as otherwise
the projection $(A_1)_{t - 1}, \dots, (A_n)_{t - 1}$ would also satisfy properties (\ref{toricsystemdefi})
and (\ref{toricsystemdefii}) of Definition \ref{toricsystemdef}. But it is clear from the proof of
Proposition \ref{dualityproposition} that this is not possible.

So, as $\mathcal{A}$ is in normal form, there must be some $A_i$ such that $A_i = (A_i)_{t - 1} + R_t$
and $(A_i)_0 = 0$. Let $i \in I = [i_1, i_2] \subset [n - 1]$ be the maximal interval such that $(A_j)_0 = 0$
for every $j \in I$. Then the sequence $\mathcal{A}_I = A_{i_1}, \dots, A_{i_2}$ must be of one of the forms 
(\ref{shortorthogonal}) or (\ref{longorthogonal}) of Lemma \ref{glueinglemma}. Moreover, there cannot be
any other $j \in \on \setminus I$ such that $(A_j)_0 = 0$ and $A_j = (A_j)_{t - 1} + R_t$ as this would
necessarily contradict property (\ref{toricsystemdefii}) of Definition \ref{toricsystemdef}.
If $\mathcal{A}_I$ is of the form of Lemma \ref{glueinglemma} (\ref{longorthogonal}), we have two
possibilities.

First, $A_i = R_t$, which implies $A_{i - 1} = (A_{i - 1})_{t - 1} - R_t$ (respectively
$A_n = (A_n)_{t - 1} - R_t$ if $i = 1$) and $A_{i + 1} = (A_{i + 1})_{t - 1} - R_t$ and $(A_j)_{t - 1} = 0$
for every other $j \in \on$. Therefore we can consider the projection $(A_1)_{t - 1}, \dots,
(A_{i - 1})_{t - 1}, (A_{i + 1})_{t - 1}, \dots, (A_n)_{t - 1}$ which is a strongly exceptional toric system
in $\pic(X_{t - 1})$.

Second, $A_i = R_t - R_k$ for some $k < t$ and thus $\chi(A_i) = 0$, then, as in proposition
\ref{normalorderingproposition}, we can reorder the toric system by replacing $A_i$ by $-A_i$,
$A_{i - 1}$ by $A_{i - 1} + A_i$ and $A_{i + 1}$ by $A_{i + 1} + A_i$, respectively, such that it remains
strongly exceptional. In particular, we can reorder it such that $A_j$ becomes $R_t$ for some $j \in I$
and apply the same argument as before.

If $\mathcal{A}_I$ is of the form of Lemma \ref{glueinglemma} (\ref{shortorthogonal}), we can consider the
divisors $A_{i_1 - 1}$ and $A_{i_2 + 1}$, where we identify $i_1 - 1$ with $n$ if $i_1 = 1$. Note that
$i_2 - i_1 < t$, so that $i_1 - 1 \neq i_2 + 1$. Now again by reordering, we can change $\mathcal{A}$
such that either $A_{i_1} = (A_{i_1})_{t - 1} - R_t$ and $A_{i_1 - 1} = (A_{i_1 - 1})_{t - 1} + R_t$,
or $A_{i_2} = (A_{i_2})_{t - 1} - R_t$ and $A_{i_2 + 1} = (A_{i_2 + 1})_{t - 1} + R_t$. But then by our
assumption on $I$ and $\mathcal{A}$ being of normal form, one of $i_1 - 1$, $i_2 + 1$ must be equal to $n$.
But above we have seen that $\delta_t \leq 0$, which is a contradiction, and $\mathcal{A}_I$ cannot be of
the form of Lemma \ref{glueinglemma} (\ref{shortorthogonal}).

Altogether we have seen now that $\mathcal{A}$ is an extension of a strongly exact toric system on
$X_{t - 1}$ and the proposition follows by induction.
\end{proof}

\begin{proposition}\label{6dstraightening}
Let $X$ be a toric surface isomorphic to $6d$ and $\mathcal{A} = A_1, \dots, A_6$ a strongly exceptional toric
system on $X$ such that $A = (A)_s = \sum_{i = 1}^5 A_i = 3H - 2 R_1 - R_2 - R_3$ in the coordinates indicated
in table \ref{straightenedlist}. Then $\mathcal{A}$ is the augmentation of a standard sequence
on $X_2$.
\end{proposition}

\begin{proof}
Clearly $A_6 = R_1$, so $A_5 = (A_5)_2 - R_1$ and $A_1 = (A_1)_2 - R_1$. If we consider the projection
$(A_1)_2, \dots, (A_5)_2$ and denote $A_I := \sum_{i \in I} A_i$ for every interval $I \subset [4]$, then
$(A_I)_2 = A_I$ if $1 \notin I$ and $(A_I)_2 - R_1 = A_I$ if $1 \in I$ and thus $A_I$ is strongly
left-orthogonal for every such $I$ and thus $(A_1)_2, \dots, (A_5)_2$ is a strongly exceptional toric
system on $X_2$ and $\mathcal{A}$ an augmentation.
\end{proof}

Denote $P_{(A)_s} := \{m \in M_\Q \mid l_i(m) \geq -c_i\}$ the rational polytope containing $G_{(A)_s}$.

\begin{lemma}\label{nonaugmentablelemma}
\begin{enumerate}[(i)]
\item\label{nonaugmentablelemmai}
Let $X$ be a toric surface and $\mathcal{A} = A_1, \dots, A_n$ a strongly exceptional toric system on $X$
such that $A = (A)_s$ and $P_{A_s}$ has no corners in $M$. Then $\mathcal{A}$ cannot be augmented to a strongly
exceptional sequence on any toric blow-up of $X$.
\item\label{nonaugmentablelemmaii}
In the cases where $X_s$ is one of 8a, 8c, 9, the polytope $P_{(A)_s}$ has no corners.
\end{enumerate}
\end{lemma}

\begin{proof}
Write $(A)_s = \sum_{i = 1}^n c_i D_i$. From \ref{cutoutproposition} it follows that for $(A)_s - R_{i_1}$ to be
strongly left-orthogonal, there must exist a lattice point $m \in G_D$ and $l_i, l_j$ such that $l_i(m) =
-c_i$ and $l_j(m) = -c_j$, i.e. $m$ is a corner of $P_{(A)_s}$, and moreover, $l_{i_1}$ must be contained in
the positive span of $l_i$ and $l_j$. So it follows that $(A)_s$ cannot be a straightening of a divisor living
on some blow-up of $X$ of the form $(A)_s - R_{i_1} - \dots - R_{i_k}$, where $i_1, \dots, i_k > t$. Now
consider $\mathcal{A}' = A_1', \dots, A_{n + k}$ a toric system which is an augmentation of $\mathcal{A}$.
As $(A)_s = (A')_{s'}$, where $s' = s + k$, the augmentation process can only happen between $A_{n - 1}$ and
$A_n$, or between $A_n$ and $A_1$. But then there exists $n' > l > {n - 1}$ such that $\sum_{i = 1}^l A_i'
= A_s - R_{i_l}$ with $i_l > t$, which cannot be strongly left-orthogonal, which proves
(\ref{nonaugmentablelemmai}). For (\ref{nonaugmentablelemmaii}) we refer to figure \ref{straightenedfigure}.
\end{proof}

We observe that the condition of lemma \ref{nonaugmentablelemma} are fulfilled for the remaining three cases.

\begin{proposition}\label{8astraightening}
Let $X$ be a toric surface isomorphic to 8a and $\mathcal{A} = A_1, \dots, A_8$ a strongly exceptional toric
system on $X$ such that $A = (A)_s = \sum_{i = 1}^7 A_i = 4H - 2(R_1 + R_2 + R_3) - R_4 - R_5$ in the
coordinates indicated in table \ref{straightenedlist}. Then $\mathcal{A}$ is cyclic strongly
exceptional and its normal form is an extension of the standard toric system on $\mathbb{P}^2$.
Without bringing it into normal form, the toric system cannot be extended to a strongly exceptional
toric system on any toric blow-up of $X$.
\end{proposition}

\begin{proof}
The latter assertion follows by Lemma \ref{nonaugmentablelemma}. To prove the first claim, we have to check
that for any nonempty cyclic interval $\emptyset \neq I \subsetneq [8]$ the divisor $A_I := \sum_{i \in I}
A_i$ is strongly left-orthogonal. By assumption, this is true for every $I$ which does not contain $n$, and
it thus remains to check the complementary intervals $\on \setminus I$ for $n \notin I$. For $A_8 = -H + R_1
+ R_2 + R_3$ we have $\chi(A_8) = 0$ and with $K_X^2 = 4$ it follows that $\chi(A_I) \leq 4$ for every
$\emptyset \neq I \subsetneq [8]$ by Lemma \ref{acyclicitylemma} (\ref{acyclicitycorollaryii}). Using
Proposition \ref{cutoutproposition} and Corollary \ref{tilingcorollary} together with formulas (\ref{eulercharp2})
and (\ref{antieulercharp2}), it is a straightforward exercise to determine
all strongly left-orthogonal divisors with Euler characteristic at most $4$. These are shown in table
\ref{totallyacyclic8a}.
\begin{table}[htbp]
\centering
\begin{tabular}{|c|c|}\hline
$\chi(D)$ & $D$\\ \hline
$0$ & $R_i - R_j$ with $\{i, j\} \neq \{1, 5\}, \{3, 4\}$,Ê\\
& $\pm(H - R_i - R_j - R_k)$ with $i, j, k$ pairwise distinct and $\{i, j, k\} \neq \{1, 2, 5\}, \{2, 3, 4\}$ \\ \hline
$1$ & $R_i$ for $i \in \{1, 2, 3, 4, 5\}$, \\
& $H - R_i - R_j$ for $i \neq j$, \\
& $2H - R_1 - R_2 - R_3 - R_4 - R_5$, \\ \hline
$2$ & $H - R_i$ for $i \in \{1, 2, 3, 4, 5\}$, \\
& $2H - \sum_{i \neq j} R_j$ for $i \in \{1, 2, 3, 4, 5\}$ \\ \hline
$3$ & $H$, $2H - R_i - R_j - R_k$ with $i, j, k$ pairwise distinct,Ê\\
& $3H - 2R_i - \sum_{j \neq i} R_j$ for any $i$ \\ \hline
$4$ & $2H - R_i - R_j$ for $i \neq j$ \\
& $3H - 2 R_i - \sum_{j \neq i, k} R_j$ for $k \neq i$ and $(i, k) \neq (1, 5), (3, 4)$, \\
& $4H - 2(R_i + R_j + R_k) - R_l - R_m$ for $i, j, k, l, m$ pairwise distinct, \\
& $5H - 3R_i - 2(R_j + R_k + R_l) - R_m$ for $i, j, k, l, m$ pairwise distinct and $i \in \{1, 4, 5\}$ \\ \hline
\end{tabular}
\caption{Strongly left-orthogonal divisors with Euler characteristic $\leq 4$ on the variety 8a}
\label{totallyacyclic8a}
\end{table}
We see that almost all elements in this table can be paired, i.e. if some $D$ is in the table, then also
$-K_X - D$ is. So, because $-K_X = \sum_{i = 1}^8 A_i$, it follows that if $A_I$ is in the table, then
$A_{\on \setminus I}$ is and the proposition follows. The only exceptions which cannot be completed to a
strongly left-orthogonal pair are $2H - R_3 - R_4$, $2H - R_1 - R_5$, $3H - R_2 - R_3 - R_4 - 2R_5$,
$3H - R_1 - R_2 - 2 R_4 - R_5$, $4H - 2(R_1 + R_2 + R_5) - R_3 - R_4$, $4H - 2(R_2 + R_3 + R_4)
- R_1 - R_5$, and $5H - 3 R_i - 2(R_j + R_k + R_l) - R_m$. We show that these cannot be of the form
$A_I$ for $I \subset [n - 1]$.

The case $5H +$ rest can be excluded at once, as by assumption $\mathcal{A}$ is in normal form with
respect to $X_0$, hence we always have $(A_I)_0 = \beta H$ with $\beta < 4$.
With respect to $\mathcal{A}$ and $I = [k, l]$ with $1 \leq k < l < n$, we consider the following four
divisors: $C_1$, $A_I$, $C_2$, $A_8$, where $A_I$ as before and $A_8 = -H + R_1 + R_2 + R_3$ as
before, and $C_1 := \sum_{j = 1}^{k - 1} A_j$, $C_2 := \sum_{j = l + 1}^{n - 1} A_j$, where $C_1 = 0$ if
$k = 1$ and $C_2 = 0$ if $l = n - 1$. Because of the properties of toric systems, we have that
$A_8 . (C_1 + C_2) = A_I . (C_1 + C_2) \in \{0, 1, 2\}$, depending on the $C_i$ being nonzero or not.

Now let us assume that $A_I = 2H - R_3 - R_4$. Then $C_1 + C_2 = -K_X - A_8 - A_I = 2H - R_1 - 2(R_2 + R_3)
- R_5$ and $A_8 . (C_1 + C_2) = 3$, which is not possible.

If $A_I = 3H - R_2 - R_3 - R_4 - 2 R_5$, we get $C_1 + C_2 = H + R_5 - 2 R_1 - R_2 - R_3$ and
$(C_1 + C_2) . A_8 = 3$. Therefore this case is also excluded.

If $A_I = 4H - 2(R_1 + R_2 + R_5) - R_3 - R_5)$, then $(C_1 + C_2) = R_5 -R_3$ and $A_8 . (C_1 + C_2)
= -1$, which is not possible.

The remaining three cases differ only by enumeration from the first three and can be excluded
analogously. Altogether, under the conditions of the proposition, the strongly exceptional toric
system $\mathcal{A}$ is always cyclic. If we bring it into normal form by inverting $A_8$, we get
that $A' = 2H - R_4 - R_5$ and $(A')_s = 2H$. So by Proposition \ref{straighteningextensionlemma}
and the subsequent remark, the toric system is an extension of the toric system $H, H, H$ on
$\mathbb{P}^2$.
\end{proof}

\begin{proposition}\label{8cstraightening}
Let $X$ be a toric surface isomorphic to 8c and $\mathcal{A} = A_1, \dots, A_8$ a strongly exceptional toric
system on $X$ such that $A = (A)_s = \sum_{i = 1}^7 A_i = 4H - 2(R_1 + R_2 + R_4) - R_3 - R_5$ in the
coordinates indicated in table \ref{straightenedlist}. Then $\mathcal{A}$ is cyclic strongly
exceptional and its normal form is an extension of the standard toric system on $\mathbb{P}^2$.
Without bringing it into normal form, the toric system cannot be extended to a strongly exceptional
toric system on any toric blow-up of $X$.
\end{proposition}

\begin{proof}
In this case the arguments are completely analogous to the proof of proposition \ref{8astraightening}.
The only difference is the classification of strongly left-orthogonal divisors with Euler characteristic at
most four, which is shown in table \ref{totallyacyclic8c}.
\begin{table}[htbp]
\centering
\begin{tabular}{|c|c|}\hline
$\chi(D)$ & $D$\\ \hline
$0$ & $\pm(R_i - R_j)$ with $i \in \{1, 2, 3\}$, $j \in \{4, 5\}$,Ê\\
& $\pm(H - R_i - R_j - R_k)$ with $i \neq j \in \{1, 2, 3\}$, $k \in \{4, 5\}$ \\ \hline
$1$ & $R_i$ for any $i$, \\
& $H - R_i - R_j$ for $i \neq j$, \\
& $2H - R_1 - R_2 - R_3 - R_4 - R_5$, \\ \hline
$2$ & $H - R_i$ for any $i$, \\
& $2H - \sum_{i \neq j} R_j$ for any $i$ \\ \hline
$3$ & $H$, $2H - R_i - R_j - R_k$ with $i, j, k$ pairwise distinct,Ê\\
& $3H - 2R_i - \sum_{j \neq i} R_j$ for any $i$ \\ \hline
$4$ & $2H - R_i - R_j$ for $i \neq j$ \\
& $3H - 2 R_i - \sum_{j \neq i, k} R_j$ for $k \neq i$ and $(i, k) \neq (4, 5), (2, 3), (1, 3), (1, 2)$, \\
& $4H - 2(R_i + R_j + R_k) - R_l - R_m$ for $i, j, k, l, m$ pairwise distinct,\\
& $i, j \in \{1, 2, 3\}$, $k \in \{4, 5\}$, \\
& $5H - 3R_i - 2(R_j + R_k + R_l) - R_m$ for $i, j, k, l, m$ pairwise distinct and $i \in \{4, 5\}$ \\ \hline
\end{tabular}
\caption{Strongly left-orthogonal divisors with Euler characteristic $\leq 4$ on the variety 8c.}\label{totallyacyclic8c}
\end{table}
In table \ref{testtotallyacyclic8c} we list the divisors $D$ from table \ref{totallyacyclic8c} which are
candidates for some $A_I$ and do not have a strongly left-orthogonal partner together with $C := A - D$,
and the intersection numbers $C . D$, $C . A_8$. As we can see, we get in every case that the
intersection numbers are not compatible with $A_I$ coming of a toric system.
\begin{table}[htbp]
\centering
\begin{tabular}{|c|c|c|c|}\hline
$D$ & $C$ & $C . D$ & $C . A_8$ \\ \hline
$2H - R_4 - R_5$ & $2H - 2(R_1 + R_2) - R_3 - R_4$ & $3$ & $3$ \\ \hlineÊ
$3H - 2 R_5 - R_1 - R_2 - R_3$ & $H + R_5 - R_1 - R_2 - 2 R_4$ & $3$ & $0$ \\ \hline
$3H - 2 R_3 - R_1 - R_4 - R_5$ & $H + R_3 - R_1 - 2 R_2 - R_4$ & $3$ & $2$ \\ \hline
$3H - 2 R_3 - R_2 - R_4 - R_5$ & $H + R_3 - 2 R_1 - R_2 - R_4$ & $3$ & $2$ \\ \hline
$3H - 2 R_2 - R_3 - R_4 - R_5$ & $H - 2 R_1 - R_4$ & $2$ & 1 \\ \hline
\end{tabular}
\caption{Testing intersection numbers of some divisors of table \ref{totallyacyclic8c}.}\label{testtotallyacyclic8c}
\end{table}
So, under the conditions of the proposition,
the strongly exceptional toric system $\mathcal{A}$ is always cyclic. If we bring it into normal
form by inverting $A_8$, we get that $A' = 2H - R_3 - R_5$ and $(A')_s = 2H$. So by
Proposition \ref{straighteningextensionlemma} and the subsequent remark, the toric system
is an extension of the toric system $H, H, H$ on
$\mathbb{P}^2$.
\end{proof}

\begin{proposition}\label{9straightening}
Let $X$ be a toric surface isomorphic to 9 and $\mathcal{A} = A_1, \dots, A_9$ a strongly exceptional toric
system on $X$ such that $A = (A)_s = \sum_{i = 1}^7 A_i = 4H - 2(R_1 + R_3 + R_5) - R_2 - R_4 - R_6$ in the
coordinates indicated in table \ref{straightenedlist}. Then $\mathcal{A}$ is cyclic strongly
exceptional and its normal form is an extension of the standard toric system on $\mathbb{P}^2$.
Without bringing it into normal form, the toric system cannot be extended to a strongly exceptional
toric system on any toric blow-up of $X$.
\end{proposition}

\begin{proof}
The proof is analogous to propositions \ref{8astraightening} and \ref{8cstraightening}. Here,
we have $\chi(A) = 3$, and table \ref{totallyacyclic9} shows the strongly left-orthogonal
divisors with Euler characteristic $\leq 3$.
\begin{table}[htbp]
\centering
\begin{tabular}{|c|c|}\hline
$\chi(D)$ & $D$\\ \hline
$0$ & $R_i - R_j$ with $\{i, j\} \neq \{1, 2\}, \{3, 4\}, \{5, 6\}$,Ê\\
& $\pm(H - R_i - R_j - R_k)$ with $i, j, k$ pairwise distinct, $\{i, j, k\} \setminus \{1, 2\} \neq \{5\}, \{6\}$; \\
& $\{i, j, k\} \setminus \{3, 4\} \neq \{1\}, \{2\}$; $\{i, j, k\} \setminus \{5, 6\} \neq \{3\}, \{4\}$, \\
& $2H - R_1 - R_2 - R_3 - R_4 - R_5 - R_6$ \\ \hline
$1$ & $R_i$ for any $i$, \\
& $H - R_i - R_j$ for $i \neq j$, \\
& $2H - \sum_{j \neq i} R_j$ for any $i$, \\ \hline
$2$ & $H - R_i$ for any $i$, \\
&  $2H - \sum_{k \neq i, j} R_k$ for any $i \neq j$, \\
&  $3H - 2 R_i- \sum_{j \neq i} R_j$ for any $i$ \\ \hline
$3$ & $H$, $2H - R_i - R_j - R_k$ with $i, j, k$ pairwise distinct,Ê\\
& $3H - 2R_i - \sum_{k \neq i, j} R_j$ for any $i \neq j$ and $j \neq i + 1$ if $i$ odd, \\
& $4H - 2(R_i + R_j + R_k) - R_l - R_m - R_n$ with $i, j, k, l, m, n$ pairwise distinct, \\
& $\{i, j, k\} \setminus \{1, 2\} \neq \{5\}, \{6\};$ 
$\{i, j, k\} \setminus \{3, 4\} \neq \{1\}, \{2\}$; $\{i, j, k\} \setminus \{5, 6\} \neq \{3\}, \{4\}$,\\
& $5H - 2(R_1 + R_2 + R_3 + R_4 + R_5 + R_6)$ \\ \hline
\end{tabular}
\caption{Strongly left-orthogonal divisors with Euler characteristic $\leq 3$ on the variety 9}\label{totallyacyclic9}
\end{table}
The unpaired divisor $5H - 2(R_1 + R_2 + R_3 + R_4 + R_5 + R_6)$ can be excluded as once, as
$\mathcal{A}$ is in normal form. For the other cases, we make use of the $\Z_3$-symmetry of the
table and consider only three cases, and the others follow the same way by exchanging indices.

Assume first $A_I = 3H - 2R_2 - R_3 - R_4 - R_5 - R_6$, then $C := A - A_I = H - 2 R_1 - R_3 - R_5$.
Then $C . A_9 = C. (-H + R_1 + R_3 + R_5) = 3$, which is not possible.

The next case is
$A_I = (2H - R_1 - R_2 - R_5)$. Then $C = 2H - R_1 - 2 R_3 - R_5 - R_6$ and $C. A_9 = -1$, which
is not possible.

The last case is $A_I = 4H - 2(R_1 + R_2 + R_5) - R_3 - R_4 - R_6$. Then $C = R_2 - R_3$ and
$C . A_9 = -1$, and this case also is excluded.

Again, altogether we get that under the conditions of the proposition,
the strongly exceptional toric system $\mathcal{A}$ is always cyclic. If we bring it into normal
form by inverting $A_9$, we get that $A' = 2H - R_2 - R_4 - R_6$ and $(A')_s = 2H$. So by
Proposition \ref{straighteningextensionlemma} and the subsequent remark, the toric system
is an extension of the toric system $H, H, H$ on $\mathbb{P}^2$.
\end{proof}


\end{document}